\documentclass[10pt]{article}
\usepackage{amsfonts}
\usepackage{amsmath,hhline,epsfig,latexsym}
\usepackage{amssymb}
\usepackage{graphicx}
\usepackage{comment}
\usepackage[title]{appendix}
\usepackage{verbatim}
\usepackage[english]{babel}
\usepackage{hyperref}
\usepackage{lineno,hyperref}
\usepackage{color}
\usepackage{xcolor} 
\makeatletter
\renewcommand{\paragraph}{\@startsection{paragraph}{4}{0ex}%
   {-3.25ex plus -1ex minus -0.2ex}%
   {1.5ex plus 0.2ex}%
   {\normalfont\normalsize\bfseries}}
\makeatother
\makeatletter
\def\blfootnote{\xdef\@thefnmark{}\@footnotetext}
\makeatother
\stepcounter{secnumdepth}
\stepcounter{tocdepth}
\usepackage{amsmath}
\providecommand{\U}[1]{\protect\rule{.1in}{.1in}}
\newenvironment{proof}[1][Proof]
{\noindent\textbf{#1:} }{\hfill$\Box$}                    

\newtheorem{teo}{Theorem}[section]
\newtheorem{defi}{Definition}[section]
\newtheorem{prop}{Proposition}[section]
\newtheorem{lema}{Lemma}[section]

\newcommand{\R}{\mathbb{R}}
\newcommand{\NN}{\mathbb{N}}

\newcommand{\dis}{\displaystyle}

\newcommand{\Oo}{\mathcal{O}}

\newcommand{\te}{\theta}
\newcommand{\D}{\Delta}
\newcommand{\ue}{u^{\varepsilon}}

\newtheorem{theorem}{Theorem}[section]
\newtheorem{remark}[theorem]{Remark}

\title{\textbf{Small-time global exact controllability to the trajectories for the viscous Boussinesq system}}

\author{
	\textsc{F. W. Chaves-Silva}\thanks{Department of Mathematics, Federal University of Para\'iba, UFPB,  CEP 58050-085, Jo\~ao Pessoa-PB, Brazil. E-mail: {\tt fchaves@mat.ufpb.br}.}
	\and
	\textsc{E. Fern\' andez-Cara}\thanks{University of Sevilla, Dpto. E.D.A.N, Aptdo 1160, 41080 Sevilla, Spain. E-mail: {\tt cara@us.es}.}
	\and
	\textsc{K. Le Balc'h}\thanks{Institut de Mathématiques de Bordeaux, Bureau 207, 351 Cours de la Libration, 33400 Bordeaux, France. E-mail: {\tt kevin.le-balch@math.u-bordeaux.fr}.}
	\and
	\textsc{J. L. F. Machado}\thanks{Department of Mathematics, Federal University of Pernambuco, UFPE, CEP 50740-545, Recife,
PE, Brazil. E-mail: {\tt  lucasmachado@dmat.ufpe.br}. Partially supported by CNPq (Brazil).}
\thanks{Federal Institute of Cear\'a, IFCE, CEP 62320-000, Tiangu\'a, CE, Brazil.E-mail: {\tt  lucas.machado@ifce.edu.br}.}\quad
	\and
	\textsc{D. A. Souza}\thanks{Department of Mathematics, Federal University of Pernambuco, UFPE, CEP 50740-545, Recife,
PE, Brazil. E-mail: {\tt diego.souza@dmat.ufpe.br}.}
}
\date{}

\begin{document}
\maketitle


\begin{abstract}
	In this paper, we deal with the global exact controllability to the trajectories of the Boussinesq system.  We consider $2$D and $3$D smooth bounded domains. 
	The velocity field of the fluid  must satisfy a Navier slip-with-friction boundary condition and a Robin boundary condition is imposed to the temperature. We assume that one can act on the velocity and the temperature on an arbitrary small part of the boundary. The proof relies on three main arguments. First, we transform the problem into a distributed controllability problem by using a domain extension procedure. Then, we prove a global approximate controllability result by following the strategy of Coron, Marbach, Sueur in \cite{coron_marbach}, which deals with the Navier-Stokes equations. This part relies on the controllability of the inviscid Boussinesq system and asymptotic boundary layer expansions. Finally, we conclude with a local controllability result that we establish with the help of a linearization argument and appropriate Carleman estimates.
\end{abstract}

\


\tableofcontents

    
\section{Introduction}

	Let $\Omega\subset\R^n$ ($n = 2, 3$) be a smooth bounded domain with $\Gamma := \partial\Omega$ and let $\Gamma_c \subset \Gamma$ be a non-empty open subset 
	 which  intersects all connected components of $\Gamma$. It will be said that $\Gamma_c$ is the control boundary. 
	Let us set
	\begin{equation*}
	   L^2_c(\Omega)^n := \{ u \in L^2(\Omega)^n\ :\ \nabla \cdot u = 0 \ \text{in}\ \Omega,\ u\cdot \nu = 0\ \text{on}\ \Gamma\setminus\Gamma_c\},
	\end{equation*}
	where $\nu=\nu(x)$ is the outward unit normal vector to $\Omega$ at the points $x \in \Gamma$.
	For a given  vector field $f$, we denote by $[f]_{tan}$ the tangential part of $f$, $D(f)$ the deformation tensor and $N(f)$ the tangential Navier boundary operator, respectively given by the following formula:
	\begin{equation}\label{def_N}
	\begin{alignedat}{2}
		[f]_{tan}&:=f-(f\cdot \nu)\nu,\\
		D(f)&:=\dfrac{1}{2}\left(\nabla f+\nabla f^t\right),\\
		N(f)&:=[D(f)\nu+Mf]_{tan},
	\end{alignedat}
	\end{equation}
	where $M = M(t,x)$ is a smooth, symmetric matrix-valued function related to the rugosity of the boundary, called the {\it friction matrix}. We also set 
	\begin{equation}\label{def_R}
	R(\theta):=\dfrac{\partial \theta}{\partial \nu}+m\theta,
	\end{equation}
	where $m=m(t,x)$ is a smooth function again related  to the properties of the boundary, known as the
	{\it heat transfer coefficient}.

Let $T>0$ be a final time. We will consider the Boussinesq system
\begin{equation}\label{eq_bou}
\left\{
    \begin{array}{lll}
    		\partial_t u -\Delta u+(u\cdot\nabla)u+\nabla p= \theta e_n & \hbox{in} & (0,T) \times \Omega,\\
    		  \noalign{\smallskip}\dis
    		\partial_t\te-\D \te+u\cdot \nabla \te=0& \hbox{in} & (0,T) \times \Omega,\\
    		\noalign{\smallskip}\dis
    		\nabla\cdot u=0& \hbox{in} &(0,T) \times \Omega,\\
    		u\cdot \nu=0, \quad N(u)=0& \hbox{on} & (0,T) \times \left(\Gamma\setminus\Gamma_c\right),\\
   		\noalign{\smallskip}\dis
    		R(\te)=0 &\hbox{on} &(0,T) \times \left( \Gamma\setminus\Gamma_c\right), \\
    			u(0,\cdot\,)=u_0, \quad \te(0,\cdot\,)=\te_0 &\hbox{in} & \Omega,
    \end{array}
    \right.
\end{equation}
where the functions $u=u(t,x)$, $\te = \te(t,x)$ and $p=p(t,x)$ must be respectively viewed as the velocity field, the temperature and the pressure  of the fluid 
	and  $e_n$ is the $n$-th vector of the canonical basis of $\R^n$, i.e., $e_n =  (0,1)$ if $n=2$ and $e_n=(0,0,1)$ if $n=3$.\\
	\indent In the controlled system \eqref{eq_bou}, at any time $t \in [0,T]$, $(u, \theta)(t,\cdot\,): \Omega \rightarrow \R^n \times \R$ will be interpreted as the \textit{state} of the system and its restriction $(u, \theta)(t, \cdot\,): \Gamma_c \rightarrow \R^n \times\R$ will be regarded as the associated \textit{control}.

\subsection{Main result}	
In this section, we state the main result of the paper, which concerns small-time global boundary exact controllability to the trajectories of \eqref{eq_bou}.\\
\indent Let us introduce the following notation:
\begin{equation}
    \label{eq:defspaces}
    W_T (\Omega):= [C^0_w([0,T];L^2_c(\Omega)^n)\cap L^2(0,T;H^1(\Omega)^n)]\times C_w^0([0,T];L^2(\Omega))\cap L^2(0,T;H^1(\Omega))].
\end{equation}
\indent We have the following result:
\begin{teo}\label{theo_main}
	Let $T>0$ be a positive time, let $(u_0,\te_0) \in L^2_c(\Omega)^n\times  L^2(\Omega)$ be an initial state and let
	$(\overline{u}, \overline{\theta}) \in W_T(\Omega)$ be a weak trajectory
	of \eqref{eq_bou}. Then, there exists 
	a weak controlled solution   to \eqref{eq_bou} in $W_T(\Omega)$ that satisfies
	\begin{equation}  \label{eq:fin_cond}
		\left( u,\theta\right)(T,\cdot\,) = \left( \overline{u},\overline{\theta}\right)(T,\cdot\,).
	\end{equation}
\end{teo}

Several comments are in order.
\begin{remark}
  \rm  For the precise notions of weak trajectory and weak controlled solution,  see Definition \ref{def_traj_contr} below. Essentially, we require to belong to $W_T(\Omega)$ and satisfy the PDEs in \eqref{eq_bou} in the weak (distributional) sense.
\end{remark}

      \begin{remark}
	  \rm   In Theorem \ref{theo_main}, we do not indicate explicitly which  are the controls. Indeed, once the controlled solution is constructed, we see that the associated controls are the appropriate traces of the solution on $(0,T)\times\Gamma_c$. 
	 \end{remark}

\begin{remark}
  \rm  Theorem \ref{theo_main} is stated as an existence result. The lack of uniqueness comes from two main reasons:
    \begin{itemize}
        \item If we do not specify any restriction, there exist many controls that drive the solution to \eqref{eq_bou} to the desired trajectory.
        \item Even if we select a criterion in order to fix the control without ambiguity, it is not known if weak solutions are unique in the $3$-D case (in $2$-D, it is known that weak solutions are unique; see \cite{boyer,LERAY2} for the Navier-Stokes case).
    \end{itemize}
\end{remark}


\subsection{Bibliographical comments}

We now present some existing results in the literature which are related to Theorem \ref{theo_main}.
	
	There are several papers where the controllability properties of the Boussinesq equations are investigated. Most of them are local results covering boundary conditions of various kinds. For instance, in \cite{Fursikov_Imanuvilov_Local_Bou}, the local exact boundary controllability to the trajectories was obtained with boundary controls acting over the whole boundary; in \cite{Fursikov_Imanuvilov_Exac_Bou}, the exact controllability with distributed controls and periodic boundary conditions was analyzed; in \cite{guerrero_bou_D}, the author proved the local exact controllability to the trajectories with Dirichlet boundary conditions; 
	this situation is also handled with a reduced number of controls in \cite{carreno, Cara_NS_BOU,Cara_NS_BOU6, Burgos}. For incompressible ideal fluids, this subject has
	been investigated by Coron~\cite{Coron2,Coron3} and~Glass~\cite{Glass1,Glass2,glass_3D} and also by Fern\'andez-Cara {\it et al.} \cite{bouss_cara} when heat effect are considered.
	
	On the other hand, the literature on the Navier-Stokes and Boussinesq equations with Navier-slip boundary conditions is scarce. Let us recall some controllability results obtained for the Navier-Stokes system: in \cite{Coron_2D}, a small-time global result for the 2D  equations has been proved where the exact controllability can be achieved in the interior of the domain and the information about the solution near the boundaries is unknown; the residual boundary layers are too strong to be handled satisfactorily during the control design strategy. 
	Guerrero proved in \cite{guerrero_CEL_NS} the local exact controllability to the trajectories with general nonlinear Navier boundary conditions.
	Finally, the small-time global exact controllability with Navier slip-with-friction boundary conditions towards weak trajectories  was proved in  \cite{coron_marbach}  by  Coron,  Marbach and Sueur. 
	 This article answers the famous open question by J.-L. Lions concerning global null-controllability of the Navier-Stokes equations with boundary conditions of this kind. 
	In what concerns the Boussinesq system with Navier-slip boundary conditions, see \cite{Bou_N_1998, kim} for some local results.
	
	\subsection{Strategy of the proof}
	
	We present in this section the main of ideas and results needed for the proof of Theorem \ref{theo_main}. 
	
	\begin{itemize}
	\item In Section \ref{sect_domain}, we will reduce the task to a \textit{distributed controllability} problem by applying a classical domain extension technique. Then, we will limit our considerations to \textit{smooth initial data} by using the smoothing effect of the uncontrolled Boussinesq system.
	\end{itemize}

	\begin{itemize}
	\item In Section \ref{sect_app_cont}, starting from a sufficiently smooth initial data, we prove a \textit{global approximate controllability} result. In order to do this, we follow the strategy performed by Coron, Marbach and Sueur in \cite{coron_marbach} in the Navier-Stokes case. 
	\end{itemize}

	\begin{itemize}
	\item In Section \ref{sec_local}, we prove a \textit{local controllability result} by using Carleman inequalities for the adjoint of the linearized system and a fixed-point strategy.
	\end{itemize}
	
	\begin{itemize}
	\item In Section \ref{control_trajec}, we combine all  these  arguments and achieve the proof.
	\end{itemize}

	In general, the notation will be abridged. For instance, if $u\in H^2(\Omega)^n$ and $\te\in H^1(\Omega)$,  $\|(u,\te)\|_{H^2\times H^1}$ will stand for the norm of $(u,\theta)$ in the space $H^2(\Omega)^n\times H^1(\Omega)$. The scalar product and norm in $L^2$ spaces will be denoted  by $(\,\cdot\,,\,\cdot\,)$ and $\|\,\cdot\,\|$, respectively. The symbol $C$ will stand for a generic positive constant.
	



\section{Domain extension and smoothing effect}\label{sect_domain}
\subsection{Domain extension}
\label{sec:domainextension}

	We consider an extended bounded domain $\mathcal{O}$ in such a way that $\Gamma_c \subset \mathcal{O}$ 
	and $\Gamma\setminus\Gamma_c\subset \Gamma_{\Oo} := \partial\mathcal{O}$.  In the sequel, we  will denote  by $\tilde{\nu}=\tilde{\nu}(x)$ the outward unit normal vector to $\Oo$ at the points $x\in\partial\Oo$. We will assume that $M$ and $m$ are extended to $[0,T]\times \partial\Oo$ as smooth functions such that $M$ is symmetric on $ (0,T)\times\partial\Oo$. This allows to speak of $N(u)$ and $R(\theta)$ on $(0,T)\times \partial \Oo$.
	
%
%
	 We will also need the space
	 \begin{equation*} 
	 L_{div}^2(\Oo)^n := \{ u \in L^2(\Oo)^n\ :\ \nabla \cdot u = 0\ \text{in}\ \Oo, \ \ u \cdot \nu = 0\ \text{on}\ \partial\Oo\}.
	 \end{equation*}
	 The following proposition enables us to extend initial conditions to the whole domain $\mathcal{O}$.
	 
	 \begin{prop}\label{prop_extension}
	     Let $(u_0, \theta_0) \in L_c^2(\Omega)^n \times L^2(\Omega)$ be given. There exist $(u_*, \theta_*) \in L^2(\mathcal{O})^{n+1}$ and $\sigma_*\in C_c^\infty( \Oo\setminus \overline{\Omega})$ such that
	     \begin{align}
	     \label{eq:extension}
	         &u_* = u_0\quad\hbox{and}\quad \theta_* = \theta_0\ \text{in}\ \Omega,\qquad \nabla\cdot u_*= \sigma_*\ \text{in}\  \mathcal{O},\qquad u_*\cdot \tilde\nu=0\ \text{on}\  \partial\mathcal{O}.
	     \end{align}
	     Furthermore, $(u_*,\te_*)$ and $\sigma_*$ can be chosen depending continuously on $(u_0,\te_0)$ in the following sense: 
	     	     \begin{align}
	     \label{eq:extension2}
	        \|u_*\|+\|\sigma_*\|\leq C\|u_0\|,\qquad    \|\te_*\|\leq C\|\te_0\|        .  
	        	     \end{align}
	 \end{prop}
	 
	 \begin{proof}
Let $\te_* \in L^2(\mathcal{O})$ be the extension by zero of $\te_0$ to the whole domain $\mathcal{O}$, then we have 
$$
  \|\te_*\|\leq \|\te_0\|.
$$

\indent Next, to find an appropriate extension for $u_0$, we first notice that, since $u_0 \in L_c^2(\Omega)^n$, the normal trace $u_0\cdot \nu$ has a sense in
	$ H^{-1/2}(\partial \Omega)$, see \cite[Chapter IV, Section 3.2]{boyer}. 
	Let  us split 
	$$
		\dis\Gamma_c=\bigcup_{i=1}^{k} \Gamma_c^i,
	$$
	where $\Gamma^i_c$ represent the parts of $\Gamma_c$ in each connected component of $\Gamma$ and we stand for $(\Oo\setminus\overline{\Omega})^i$ its related extension. 
	Also, let $\omega^i\subset\subset(\Oo\setminus\overline{\Omega})^i$ be a non-empty open subset and $\sigma^i_*\in C_c^\infty( \omega^i)$ such that
	\[
	\int_{(\mathcal{O}\setminus\overline{\Omega})^i}\sigma_*^i\,  =-\langle u_0\cdot \nu,1\rangle_{H^{-1/2}(\Gamma_c^i),H^{1/2}(\Gamma_c^i)}.
	\]
	The following non homogeneous elliptic problem admits a unique solution $w^i\in H^1((\mathcal{O}\backslash\overline\Omega)^i)$:
\[
\left\{
    \begin{array}{lll}
    		-\Delta w^i= -\sigma_*^i & \hbox{in} &  (\mathcal{O}\backslash\overline\Omega)^i,\\
    		  \noalign{\smallskip}\dis
    		 \frac{\partial w^i}{\partial \nu}=u_0\cdot \nu& \hbox{on} &\Gamma_c^i,\\
   		\noalign{\smallskip}\dis
    	 	 \frac{\partial w^i}{\partial \widetilde\nu}=0& \hbox{on} & \partial(\mathcal{O}\backslash\overline\Omega)^i \setminus\Gamma_c^i.
    \end{array}
    \right.
\]	
			
Let us set
\[
u_*:=\left\{
    \begin{array}{lll}
    		u_0 & \hbox{in} & \Omega,\\
    		  \noalign{\smallskip}\dis
    		 \nabla w^i & \hbox{in} &(\mathcal{O}\backslash\overline\Omega)^i, \quad \text{for}\quad i=1,\ldots, k.
    \end{array}
    \right.
\]		
	It is then clear that $u_*\in L^2(\mathcal{O})^n$, $\nabla\cdot u_*= \sigma_*$ in $\mathcal{O}$  and $u_*\cdot \widetilde\nu=0$ on $\partial \Oo$. On the other hand, we see that, by construction, \eqref{eq:extension} and \eqref{eq:extension2} are satisfied.	
%
%
	\end{proof}\\
		
    Let us now present the notion of solution used throughout the paper. To this purpose, let us  introduce the following notations $\mathcal{O}_T:=(0,T)\times\mathcal{O}$ and $\varLambda_T:=(0,T)\times\partial\mathcal{O}$. In the sequel, when there is no ambiguity, we will also denote by $\nu$ the outward unit normal to $\Oo$.
\begin{defi}\label{def_traj_contr}
	Let ~$T>0$ and $(u_0,\te_0) \in L^2_c(\Omega)^n\times L^2(\Omega)$ be given. It will be said that $(u,\te)\in W_T(\Omega)$ is a weak controlled trajectory of \eqref{eq_bou} 
	if it is the restriction to $(0,T)\times\Omega$ of a weak Leray solution, still denoted by $(u,\te)$, in the space $W_T(\Oo)$, to the nonlinear system
\begin{equation}\label{def_solu}
	\left\{
    \begin{array}{lll}
    		\partial_t u-\Delta u +(u\cdot\nabla)u+\nabla p= \theta e_n + v  &\hbox{in} & \mathcal{O}_T,\\
    		\noalign{\smallskip}\dis
    		\partial_t \theta-\D \te +u\cdot \nabla \te=w  &\hbox{in} & \mathcal{O}_T,\\
    		\noalign{\smallskip}\dis
    		\nabla\cdot u=\sigma  &\hbox{in} &\mathcal{O}_T,\\
    		\noalign{\smallskip}\dis
    		u\cdot \nu=0, \quad N(u)=0&\hbox{on} & \varLambda_T,\\
    		\noalign{\smallskip}\dis
    		R(\te)=0&\hbox{on} & \varLambda_T,\\
    		\noalign{\smallskip}\dis
    		u(0,\cdot\,)=u_*, \quad \te(0,\cdot\,)=\te_* &\hbox{in} & \mathcal{O},
    \end{array}
    \right.
\end{equation}
	where $v\in H^1(0,T;L^2(\Oo)^n)\cap C^0([0,T]; H^1(\Oo)^n)$ and $w \in H^1(0,T;L^2(\Oo))\cap C^0([0,T]; H^1(\Oo))$ are forcing terms supported in \,$ \overline{\Oo} \setminus \overline{\Omega}$, $\sigma \in C^{\infty}([0,T]\times\mathcal{O})$ is a nonhomogeneous divergence condition also supported in \,$ \overline{\Oo} \setminus \overline{\Omega}$ and $(u_*, \theta_*)$ is an extension of $(u_0, \theta_0)$ furnished by Proposition \ref{prop_extension}, satisfying $\nabla \cdot u_* = \sigma(0,\cdot\,)$.
\end{defi}

	Let us state an existence result of weak solution to \eqref{def_solu},
	whose proof is sketched in Appendix~\ref{appendix_existence}: 
\begin{prop}\label{prop_existence}
	Let us assume that $T>0$, $(u_*,\te_*)\in L^2(\mathcal{O})^n\times L^2(\mathcal{O})$ satisfies $u_*\cdot \nu=0$ on $\partial\mathcal{O}$, $\sigma\in C^\infty([0,T]\times\mathcal{O})$ satisfies $\sigma(0,\,\cdot\,)=\nabla\cdot u_*$, 
	$v \in H^1(0,T;L^2(\Oo)^n)\cap C^0([0,T]; H^1(\Oo)^n)$ and $w \in H^1(0,T;L^2(\Oo))\cap C^0([0,T]; H^1(\Oo))$. Then there exists at least one weak Leray solution $(u,\te)$ to \eqref{def_solu}.
\end{prop}

\subsection{Smoothing effect of the uncontrolled Boussinesq system}\label{slip_cond}

The goal of this section is to show that, starting from $L^2$ initial data, for any small time interval, one can find a time such that the solution is sufficiently smooth. More precisely, we have the following result:
\begin{lema}\label{lemma9}
	Let us assume that $T>0$ and $(\overline{u},\overline{\te})\in C^\infty([0,T]\times \overline{\Oo})^{n+1}$ is such that $\nabla\cdot \overline{u}=0$ in $\mathcal{O}_T$ and $\overline{u}\cdot \nu=0$ on $\varLambda_T$.
Then there exists a smooth function $\Psi: \R^+ \mapsto \R^+$ with $\Psi(0)=0$ such that, for any $(r_*,q_*)\in L_{div}^2(\Oo)^n\times  L^2(\Oo)$ and any
	weak Leray-Hopf solution $(r,q)\in W_T(\Oo)$ to:
	\begin{equation}\label{equation_lemma9}
		\left\{
   		 \begin{array}{lll}
    		\partial_t r-\Delta r+(r\cdot\nabla)r+(\overline{u}\cdot\nabla)r+(r\cdot\nabla)\overline{u}+\nabla \pi = qe_n &\hbox{in}& \mathcal{O}_T,\\
    		\noalign{\smallskip}\dis
    		\partial_t q -\D q+(r+\overline{u})\cdot\nabla q+r\cdot\nabla \overline{\te}=0&\hbox{in}& \mathcal{O}_T,\\
    		\noalign{\smallskip}\dis
    		\nabla\cdot r=0&\hbox{in}& \mathcal{O}_T,\\
    		\noalign{\smallskip}\dis
    		r\cdot \nu=0, \quad N(r)=0&\hbox{on}& \varLambda_T,\\
    		\noalign{\smallskip}\dis
    		R(q)=0&\hbox{on}& \varLambda_T,\\
    		\noalign{\smallskip}\dis
    		r(0,\cdot\,)=r_*, \quad q(0,\cdot\,)=q_*&\hbox{in}& \mathcal{O},
    	\end{array}
    	\right.
	\end{equation}
	the following property holds:
	\begin{equation}
		\exists \ t_0 \in [0,T]; \ \ \|(r,q)(t_0,\cdot\,)\|_{H^3}\leq \Psi\left( \|(r_*,q_*)\|\right).
	\end{equation}
\end{lema}

	The proof of this lemma is quite classical but, for completeness, will be given  in Appendix \ref{appen:B}.


\section{Approximate controllability problem}\label{sect_app_cont}

	In this section, the goal is to prove the following approximate controllability result starting from sufficiently smooth initial data.
\begin{prop}\label{lemma_trajec_aprox}
	Let us assume that $T>0$ and $(\overline{u},\overline{p},\overline{\te},\overline{v}, \overline{w},\overline{\sigma})\in C^\infty([0,T]\times \overline{\Oo};\mathbb{R}^{2n+4})$ is a smooth trajectory  of \eqref{def_solu}, with $\overline{v}$ and $\overline{w}$ supported in $ \overline{\Oo} \setminus \overline{\Omega}$.
	Let  $(u_*,\te_*) \in  [H^3(\mathcal{O})^n\cap L^2_{div}(\mathcal{O})^n]\times H^3(\Oo)$ be an initial state. 	
	Then, for any $\delta>0$ there exist regular controls $v$, $w$ and $\sigma$, again supported in $ \overline{\Oo} \setminus \overline{\Omega}$ and an associated  weak solution  to 		\eqref{def_solu} satisfying
	\begin{equation*}
	    \|(u, \theta)(T,\cdot\,)-(\overline{u}, \overline{\theta})(T,\cdot\,)\|\leq \delta.
	\end{equation*}
\end{prop}

	For the proof, we will follow the strategy introduced by Coron, Marbach, Sueur in \cite{coron_marbach}. Let us explain how it works:
	\begin{itemize}
	    \item First, a scale change associated to a small parameter $\varepsilon >0$ is introduced and \eqref{def_solu} is transformed into a Boussinesq system with small viscosity $\varepsilon$ that must be solved in the (long)  time interval $[0,T/\varepsilon]$ starting from a small initial state, see \eqref{equa_u_layer_}. The advantage of this scaling is that we can benefit from the nonlinear terms  $(u \cdot \nabla) u$ and $u \cdot \nabla \theta$.
	    \item By taking formally $\varepsilon = 0$, we obtain the inviscid Boussinesq system. For this hyperbolic system, we construct a particular nontrivial trajectory that connects $(0,0) \in \mathbb{R}^{n+1}$ to itself and sends any particle outside the physical domain before the final  time $T$.
	    \item By linearizing the inviscid Boussinesq system around the previous trajectory, we obtain a new hyperbolic linear system that is small-time globally null-controllable. Actually, what we are doing here is to apply the so called  return method, due to Coron, see \cite{coron_1992}. Note that the linearization around the trivial state leads to a noncontrollable system.
	    \item In the particular case of the special slip boundary condition, that is, $M$ such that $[\nabla\times u]_{\tan}=0$ on $\varLambda_T$ and $m\equiv 0$, we immediately conclude by estimating the remainder terms. We do not need to use the long interval time $[0,T/\varepsilon]$ to control, since the solution is already small at the intermediate time $T \in (0,T/\varepsilon)$.
	    \item Unfortunately, in the general case, a boundary layer appears. This phenomenon was already taken into account in \cite{iftimie_Sueur} for the Navier-Stokes PDEs. Thus, we have to introduce some corrector terms in the asymptotic expansion of the solution depending on $\varepsilon$ in order to estimate the residual layers. The boundary layer decays but not enough. Hence, the corrector is not sufficiently small at the final time $T/\varepsilon$ and we still cannot conclude.
	    \item In order to overcome this difficulty, we adapt the well-prepared dissipation method, introduced by Marbach in \cite{marbach_burgers}. The idea is to design a control strategy that reinforces the action of  the natural dissipation of the boundary layer after the intermediate time $T$. A desired small state is obtained at final time and we can finally achieve the proof.
	\end{itemize}
	
	In the sequel, we will frequently need vector functions $(u,p,\te,v,w,\sigma)$ representing adequate states $(u,p,\te)$, controls $(v,w)$ and auxiliary functions $\sigma$, corresponding to some linear or nonlinear systems. In all cases, it will be implicitly assumed that $v$ and $w$ vanish outside $ \overline{\Oo} \setminus \overline{\Omega}$. 
	\subsection{Time scaling}
	 Let us introduce $u^{\varepsilon}$, $p^{\varepsilon}$, etc., with
	\begin{equation}\label{scaling}
	\begin{array}{ccc}
		u^{\varepsilon}(t,x):= \varepsilon u(\varepsilon t,x),& 
		p^{\varepsilon}(t,x):=\varepsilon^2p(\varepsilon t,x),& 
		\te^{\varepsilon}(t,x):=\varepsilon^2\te(\varepsilon t,x),\\ 
		v^{\varepsilon}(t,x):=\varepsilon^2v(\varepsilon t,x),& 
		w^{\varepsilon}(t,x):=\varepsilon^3w(\varepsilon t,x)& 
		\sigma^{\varepsilon}(t,x):=\varepsilon\sigma(\varepsilon t,x).
	\end{array}
	\end{equation}
	In these new variables, the original system \eqref{def_solu} reads
	\begin{equation}\label{equa_u_layer_}
	\left\{
    \begin{array}{lll}
    		\partial_t u^\varepsilon-\varepsilon\Delta u^\varepsilon+(u^\varepsilon\cdot\nabla)u^\varepsilon+\nabla p^\varepsilon = \theta^\varepsilon e_n + v^\varepsilon&\hbox{in}& (0,T/\varepsilon)\times\mathcal{O},\\
    	\noalign{\smallskip}\dis
    		\partial_t\te^\varepsilon-\varepsilon\D \te^\varepsilon+u^\varepsilon\cdot \nabla \te^\varepsilon=w^\varepsilon&\hbox{in}& (0,T/\varepsilon)\times\mathcal{O},\\
    	\noalign{\smallskip}\dis
    \nabla\cdot u^\varepsilon=\sigma^\varepsilon&\hbox{in}& (0,T/\varepsilon)\times\mathcal{O},\\
    		u^\varepsilon\cdot \nu=0, \quad N(u^\varepsilon)=0&\hbox{on}& (0,T/\varepsilon)\times\partial\mathcal{O},\\
    	\noalign{\smallskip}\dis
    		R (\te^\varepsilon)=0&\hbox{on}& (0,T/\varepsilon)\times\partial\mathcal{O},\\
    	\noalign{\smallskip}\dis
    		u^\varepsilon(0,\cdot\,)=\varepsilon u_*, \quad \te^\varepsilon(0,\cdot\,)=\varepsilon^2\te_*&\hbox{in}& \mathcal{O}.
    \end{array}
    \right.
	\end{equation}
	\indent Instead of working hard in a small time interval, we now work easily during a large time interval $[0,T/\varepsilon]$. The counterpart is the small viscosity that we find now in \eqref{equa_u_layer_}, that can be viewed as a singular perturbation of a nonlinear inviscid system.\\
	\indent To prove Proposition \ref{lemma_trajec_aprox}, 
	it is sufficient to prove that
	\begin{equation*}
		\left\|u^{\varepsilon}(T/\varepsilon,\cdot\,)-\varepsilon\overline{u}(T,\cdot\,)\right\|=o(\varepsilon)\quad 
		\hbox{and}\quad
		\left\|\te^{\varepsilon}(T/\varepsilon,\cdot\,)-\varepsilon^2\overline{\te}(T,\cdot\,)\right\|=o(\varepsilon^2).
	\end{equation*}

	\subsection{The special case of the slip boundary condition}

	In this section, we consider a special situation where the fluid perfectly slips.
	In this case, the proof of Proposition \ref{lemma_trajec_aprox} is more simple  (there is no boundary layer). For the moment, we will also assume that the smooth target trajectory  is zero, i.e., $(\overline{u},\overline{p}, \overline{\te},\overline{v}, \overline{w},\overline{\sigma}) \equiv 0$.

	Thus, the friction coefficient $M$ is assumed to be the Weingarten map (or shape operator) $M_w$. 
	Thanks to \cite[Lemma $1$]{coron_marbach}, on the uncontrolled boundary one has
\begin{equation}\label{boud_cond_slip}
	u\cdot \nu=0 \quad \hbox{and} \quad [\nabla\times u]_{tan}=0 \quad \hbox{on}\quad \varLambda_T.
\end{equation}

\subsubsection{Ansatz with no correction term}
Let us consider an asymptotic expansion of the solution:
\begin{equation}\label{expans}
    \begin{array}{lll}
		\ue=u^0+\varepsilon u^1+\varepsilon r^{\varepsilon},&
		p^{\varepsilon}=p^0+\varepsilon p^1+\varepsilon \pi^{\varepsilon},&
		\te^{\varepsilon}=\te^0+\varepsilon^2 \te^1+\varepsilon^2 q^{\varepsilon}\\
		v^{\varepsilon}=v^0+\varepsilon v^1,&
		w^{\varepsilon}=w^0+\varepsilon^2 w^1,&
		\sigma^{\varepsilon}=\sigma^0.
	\end{array}
\end{equation}
There is  some intuition behind \eqref{expans}. The first term $(u^0, p^0, \theta^0, v^0, w^0)$ is the solution to an inviscid system, take $\varepsilon = 0$ in \eqref{equa_u_layer_}. It models a smooth reference trajectory around which we linearize the original system. This is exactly what we have to do when we apply the return method of Coron, see \cite{coron_1992}. It  will be chosen in such a way that the flow flushes the initial data off the physical domain before time $T$. The second term $(u^1, p^1, \theta^1, v^1, w^1)$ takes into account the initial data $(u_*, \theta_*)$.
Then, $(r^\varepsilon, \pi^\varepsilon, q^\varepsilon)$ contains higher order terms. At the end, we need to prove that $\|(r^{\varepsilon}, q^{\varepsilon})(T,\cdot\,)\|= o(1)$, in order to be able to conclude.

\subsubsection{Inviscid flow} 
\label{sec:inviscid}
\indent By taking $\varepsilon =0$ in \eqref{equa_u_layer_}, we obtain the following system 
\begin{equation}\label{solu_line}
    \left\{
    \begin{array}{lll}
    		\partial_t u^0+(u^0\cdot\nabla)u^0+\nabla p^0 = \theta^0 e_n + v^0& \text{in}& \mathcal{O}_T,\\
    	\noalign{\smallskip}\dis
    		\partial_t \te^0+u^0\cdot \nabla \te^0=w^0& \text{in} & \mathcal{O}_T,\\
    	\noalign{\smallskip}\dis
    		\nabla\cdot u^0=\sigma^0& \text{in}& \mathcal{O}_T,\\
    	\noalign{\smallskip}\dis
    		u^0\cdot \nu=0&\text{on}& \varLambda_T,\\
    	\noalign{\smallskip}\dis
    		u^0(0,\cdot\,)=u^0(T,\cdot\,)=0& \text{in}& \mathcal{O},\\
    	\noalign{\smallskip}\dis
    		\te^0(0,\cdot\,)=\te^0(T,\cdot\,)=0& \text{in}& \mathcal{O},
    \end{array}
	\right.
\end{equation}
where $v^0$, $w^0$ and $\sigma^0$ are smooth forcing terms spatially supported in $\overline{\Oo} \setminus \overline{\Omega}$. 
We want to control \eqref{solu_line} during the shorter time interval $[0, T]$ instead of $[0, T/\varepsilon]$. Let us introduce the flow $\Phi^0=\Phi^0(s;t,x)$ associated to $u^0$, ie., for any  $(t,x)$, $\Phi^0(\,\cdot\,,t,x)$ solves
\begin{equation}\label{def_flow}
    \left\{
    \begin{array}{lll}
    	\partial_s\Phi^0(s;t,x)&=&u^0(s,\Phi^0(s;t,x)),\\
		\Phi^0(s;t,x)|_{s=t}&=&x.
\end{array}
	\right.
\end{equation}
Hence, we look for trajectories such that:
\begin{equation}\label{prop_flow}
\forall \ x \in \overline{\mathcal{O}}, \ \exists \  t_x \in (0,T), \ \Phi^0(t_x;0,x)\not\in \overline{\Omega}.
\end{equation}
This property is obvious for the points $x$ already located in $\overline{\mathcal{O}} \setminus  \overline{\Omega}$. For points $x\in  \overline{\Omega}$, we use the following result, whose proof can be found in \cite{Coron2,Coron3,Coron_2D,coron_marbach} in the $2$D case and \cite{glass_3D} in the $3$D case:

\begin{lema}\label{exis_sol_u0}
There exists a non-zero solution to \eqref{solu_line} $(u^0,p^0,\te^0,v^0,w^0,\sigma^0)\in C^\infty([0,T]\times\overline{\mathcal{O}};\mathbb{R}^{2n+4})$  such that the associated flow $\Phi_0$, defined in  \eqref{def_flow}, satisfies \eqref{prop_flow}. Moreover, we can choose $u^0$, $\te^0$ and $w^0$ such that
\begin{equation}\label{prop_u0}
\te^0=w^0=0 \ \ \text{and} \ \ \nabla\times u^0=0 \ \ in \ \ [0,T]\times \overline{\mathcal{O}}
\end{equation}
and $u^0,p^0,\te^0,v^0,w^0$ and $\sigma^0$ are compactly supported in time in $(0,T)$. 
\end{lema}

Note that, in the proof of this result, the assumption that $\Gamma_c$ intersects all connected components of $\Gamma$ must be used.

In the sequel, if needed, it will be assumed that $u^0,p^0,\te^0,v^0,w^0$ and $\sigma^0$ have been extended  by zero after time $T$.

\subsubsection{Flushing}
\label{sec:flushing}
 Let $(u^1, \theta^1)$ be the solution to the linear problem 
\begin{equation}\label{u1_equa}
    \left\{
    \begin{array}{lll}
    		\partial_t u^1+(u^0\cdot\nabla)u^1+(u^1\cdot\nabla)u^0+\nabla p^1 = \Delta u^0+ v^1& \text{in}& \mathcal{O}_T,\\
    	\noalign{\smallskip}\dis
    		\partial_t \te^1+u^0\cdot \nabla \te^1=w^1& \text{in}& \mathcal{O}_T,\\
    	\noalign{\smallskip}\dis
    		\nabla\cdot u^1=0& \text{in}&\ \mathcal{O}_T,\\
    	\noalign{\smallskip}\dis
    		u^1\cdot \nu=0&\text{on}& \varLambda_T,\\
    	\noalign{\smallskip}\dis
    		u^1(0,\cdot\,)=u_*, \quad \te^1(0,\cdot\,)=\theta_*& \text{in}& \mathcal{O},
    \end{array}
    \right.
\end{equation}
where $v^1$ and $w^1$ are forcing terms spatially supported in $ \overline{\Oo} \setminus \overline{\Omega}$. 
Thanks to \eqref{prop_u0}, we have $\Delta u^0=\nabla(\nabla\cdot u^0)=\nabla \sigma^0$. Thus, it is smooth
and can be absorbed by the source term $v^1$. Of course, \eqref{u1_equa} is a linear  uncoupled system.


\begin{lema}\label{exist_sol_u1}
Let us assume that $(u_*,\theta_*)\in [H^3(\mathcal{O})^n\cap L^2_{div}(\mathcal{O})^n]\times H^3(\Oo)$. There exist forcing terms 
\begin{equation}\label{eq:suppForcingterms1}
    v^1\in C^1([0,T];H^1(\Oo)^n)\cap C^0([0,T];H^2(\Oo)^n),\ w^1 \in C^1([0,T];H^2(\Oo))\cap C^0([0,T];H^3(\Oo)),
\end{equation}        
  with
\begin{equation} \label{eq:suppForcingterms2}
    \text{supp}(v^1,w^1) \subset\subset \overline{\Oo}\setminus\overline{\Omega}
\end{equation}
such that the associated solution $(u^1,\te^1)$ to \eqref{u1_equa} satisfies $(u^1, \theta^1)(T,\cdot\,)=(0,0)$ in $\Oo$. Moreover, $(u^1, \theta^1)$ is bounded (with respect to $\varepsilon$) in $L^\infty(0,T; H^3(\Oo)^n)\times L^\infty(0,T; H^3(\Oo))$. 
\end{lema}

\begin{proof}
First, note that the result for $u^1$ is proved in \cite[Lemma 3]{coron_marbach}. 

For $\te^1$, we have a similar situation and we can apply the same arguments. For completeness, let us sketch the main ideas. We will use the smooth partition of unity $\eta_\ell$  for $1\leq \ell\leq L$ defined in \cite[Appendix A]{coron_marbach} which is related to $\Phi^0$ as follows: 
 thanks to \eqref{prop_flow}, we can find $\varepsilon>0$ and balls $B_\ell$ for $1\leq \ell\leq L$ covering $\overline{\Oo}$ such that
\begin{equation}\label{particao_unidade}
\forall \ell, \ \exists t_\ell\in (\varepsilon, T-\varepsilon), \  \exists m_\ell\in \{1, 
\cdots,\mathcal{M}\}\quad\hbox{such that}\quad \Phi^0(s;0,B_\ell)\subset Q_{m_\ell}\quad \forall s\in (t_\ell-\varepsilon, t_\ell+\varepsilon), 
\end{equation}
where the $Q_{m_\ell}$ are squares (or cubes) that never intersect $\overline{\Omega}$; hence, every ball spends a positive amount of time within a given square (cube) where we can use a localized control to act on the  $\te^1$ profile. Here, it is assumed that the $\eta_\ell$ satisfy  $0\leq \eta_\ell(x)\leq 1$, $\sum\eta_\ell=1$ and $supp(\eta_\ell)\subset B_\ell$.\\
\indent Let us introduce a smooth function $\beta: \R\to [0,1]$ with $\beta=1$ on $(-\infty, -\varepsilon)$ and $\beta=0$ on~$(\varepsilon, +\infty)$. 

For each $\ell$, we consider the solution $\overline{\te}_\ell$ to
\begin{equation*}
    \left\{
    \begin{array}{lll}
    		\partial_t\overline{\te}_\ell+u^0\cdot \nabla \overline{\te}_\ell=0& \text{in}& (0,T)\times\overline{\mathcal{O}},\\
    	\noalign{\smallskip}\dis
    		\overline{\te}_\ell(0,\cdot\,)=\eta_\ell\theta_*& \text{in}& \overline{\mathcal{O}},
    \end{array}
    \right.
\end{equation*}
and we set $\te_\ell(t,x):=\beta(t-t_\ell)\overline{\te}_\ell(t,x)$. Since $\beta(T-t_\ell)=0$ and $\beta(-t_\ell)=1$, $\te_\ell$ solves 
\begin{equation*}
    \left\{
    \begin{array}{lll}
    		\partial_t\te_\ell+u^0\cdot \nabla \te_\ell=w_\ell& \text{in}& (0,T)\times\overline{\mathcal{O}},\\
    	\noalign{\smallskip}\dis
    		\te_\ell(0,\cdot\,)=\eta_\ell\theta_*,\quad \te_\ell(T,\cdot\,)=0& \text{in}& \overline{\mathcal{O}},
    \end{array}
    \right.
\end{equation*}
where $w_\ell(t,x):=\beta^\prime(t-t_\ell)\overline{\te}_\ell$. Thanks to \eqref{particao_unidade}, since $\beta^\prime$ vanish outside $(-\varepsilon,\varepsilon)$, it is easy to see that $w_\ell$ is supported in $Q_{m_\ell}$. 

At this point, we take 
\begin{equation*}
\dis\te^1:=\sum_\ell \te_\ell \quad \hbox{and}\quad w^1:=\sum_\ell w_\ell
\end{equation*}
and we see that the second PDE and the second initial condition in \eqref{u1_equa} are satisfied. Thanks to this explicit construction, the spatial regularity of $w^1$ and  $\overline{\te}_\ell$ are the same. Then, $w^1\in C^1([0,T],H^2(\Oo))\cap C^0([0,T], H^3(\Oo))$. The fact that $\theta^1$ is bounded in $L^{\infty}(0,T; H^3(\mathcal{O}))$ readily comes from the fact that each $\theta_\ell$ is bounded in $L^{\infty}(0,T; H^3(\mathcal{O}))$. This ends the proof.

\end{proof}

Lemma \ref{exist_sol_u1} is a null-controllability result. Thanks to the linearity and reversibility of \eqref{u1_equa}, it leads  to an  exact controllability result:
\begin{lema}
\label{lem:constructionu1theta^1}
Let us assume that $(u_*,\theta_*), (u_T,\theta_T)\in [H^3(\mathcal{O})^n\cap L^2_{div}(\mathcal{O})^n]\times H^3(\Oo)$. There exist $v^1$ and $w^1$ as in \eqref{eq:suppForcingterms1} and \eqref{eq:suppForcingterms2} such that the associated solution  to \eqref{u1_equa} satisfies $(u^1, \theta^1)(T, \cdot\,) =(u_T, \theta_T)$. Moreover, $(u^1, \theta^1)$ is bounded (with respect to $\varepsilon$) in $L^\infty(0,T; H^3(\Oo)^n)\times L^\infty(0,T; H^3(\Oo))$. 
\end{lema}

\subsubsection{Equations and estimates for the remainder}
	The equations for $r^\varepsilon$, $\pi^\varepsilon$ and $q^\varepsilon$ in the extended domain $\mathcal{O}$ are
\begin{equation}\label{remain_equa}
    \left\{
    \begin{array}{lll}
    		\partial_t r^\varepsilon-\varepsilon\Delta r^\varepsilon+(u^\varepsilon\cdot\nabla)r^\varepsilon+\nabla \pi^\varepsilon =  f^\varepsilon-A^\varepsilon r^\varepsilon+\varepsilon q^\varepsilon e_n+\varepsilon\theta^1e_n&\hbox{in}& \mathcal{O}_T,\\
    	\noalign{\smallskip}\dis
    		\partial_t q^{\varepsilon}-\varepsilon\Delta q^\varepsilon+u^\varepsilon\cdot \nabla q^\varepsilon= h^\varepsilon- B^\varepsilon r^\varepsilon&\hbox{in}& \mathcal{O}_T,\\
    	\noalign{\smallskip}\dis
    		\nabla\cdot r^\varepsilon=0&\hbox{in}& \mathcal{O}_T,\\
    	\noalign{\smallskip}\dis
    		r^\varepsilon\cdot \nu=0, \quad [	\nabla\times r^\varepsilon ]_{tan}=-[ \nabla\times u^1 ]_{tan}&\hbox{on}& \varLambda_T,\\
    	\noalign{\smallskip}\dis
    	R(q^\varepsilon)=-R(\theta^1)&\hbox{on}&\varLambda_T,\\
    	\noalign{\smallskip}\dis
    		r^\varepsilon(0,\cdot\,)=0, \quad q^\varepsilon(0,\cdot\,)=0&\hbox{in}& \mathcal{O},
    \end{array}
    \right.
\end{equation}
where we have introduced
\begin{eqnarray}
f^\varepsilon:=\varepsilon\Delta u^1-\varepsilon(u^1\cdot\nabla)u^1,&
A^\varepsilon r^\varepsilon:=(r^\varepsilon\cdot\nabla)(u^0+\varepsilon u^1),\\
	\noalign{\smallskip}\dis
	h^\varepsilon:=	\varepsilon\Delta\theta^1- \varepsilon u^1\cdot\nabla \theta^1,&
	B^\varepsilon r^\varepsilon:=\varepsilon r^\varepsilon\cdot\nabla\theta^1.\label{B_slip}
\end{eqnarray}
	We can establish energy estimates for the remainder by multiplying  \eqref{remain_equa}$_1$ by $r^\varepsilon$ and \eqref{remain_equa}$_2$ by $q^\varepsilon$. Indeed,  after integration by parts, and thanks to the interpolation 
 inequality in \cite[Theorem $III.2.36$]{boyer}), we easily obtain the following estimates
\begin{equation}\label{eq:estiqpartialq}
-\int_{\partial\Oo}q^\varepsilon \dfrac{\partial q^\varepsilon}{\partial \nu} \,d \Gamma=
\int_{\partial\Oo}m|q^\varepsilon|^2 \,d \Gamma+\int_{\partial\Oo} q^\varepsilon R(\theta^1)\,d \Gamma,
\end{equation}
\begin{equation}\label{eq:estiqpartialq2}
\left|\int_{\partial\Oo}q^\varepsilon R(\theta^1)\,d \Gamma\right|
\leq \|q^\varepsilon\|_{L^2(\partial \Oo)}\left\| R(\theta^1)\right\|_{L^2(\partial \Oo)}
\leq C \|q^\varepsilon\|_{H^1}\left\|  \theta^1 \right\|_{H^2},
\end{equation}
\begin{equation}\label{eq:estiqpartialq3}
\left|\int_{\partial\Oo}m|q^\varepsilon|^2 \,d \Gamma\right|
\leq C \|q^\varepsilon\|_{L^2}\|q^\varepsilon\|_{H^1}
\end{equation}
 and   
\begin{align}
&\dfrac{d}{dt}(\|r^\varepsilon\|^2+\|q^\varepsilon\|^2)+\varepsilon(\|\nabla\!\times r^\varepsilon\|^2+\|\nabla q^\varepsilon\|^2)\notag\\
&\leq
C(\varepsilon+\|\sigma^0\|_{L^\infty}+\|f^\varepsilon\|+\|A^\varepsilon\|_{L^\infty}+\|B^\varepsilon\|_{L^\infty}+\|h^\varepsilon\|)(\|r^\varepsilon\|^2+\|q^\varepsilon\|^2)\notag\\
&\qquad+(2\varepsilon\|u^1\|^2_{H^2}+\|f^\varepsilon\|+C\varepsilon\|\theta^1\|^2_{H^2}+\|h^\varepsilon\|),\label{eq:BeforeGronwallNoSlip}
\end{align}
	where the boundary term for $r^\varepsilon$ is bounded in a similar way as in \cite[Section $2.5$]{coron_marbach}.

By applying Gronwall's inequality and Lemma \ref{exist_sol_u1}, we deduce that
\begin{equation}\label{gron_re_qe}
\|r^\varepsilon\|^2_{L^\infty(L^2)}+\|q^\varepsilon\|^2_{L^\infty(L^2)}+\varepsilon\left(\|\nabla\times r^\varepsilon\|^2+\|\nabla q^\varepsilon\|^2\right)=O(\varepsilon).
\end{equation}
Consequently, at  time $T$, since $(u^0, \theta^0)(T,\cdot\,) =(u^1, \theta^1)(T,\cdot\,) = (0,0)$, we find: 
$$ \|u^{\varepsilon}(T, \cdot\,)\| \leq \|\varepsilon r^{\varepsilon}(T, \cdot\,)\| \leq O(\varepsilon^{3/2})
\quad 
\text{and} \quad  \|\theta^{\varepsilon}(T, \cdot\,)\| \leq \|\varepsilon^2 q^{\varepsilon}(T, \cdot\,)\| \leq O(\varepsilon^{5/2}).  $$ 
This concludes the proof of Proposition \ref{lemma_trajec_aprox} in the slip boundary condition case.

\begin{remark} \rm
\label{remark:dirichlet}
In the previous proof, we have used in a crucial way the homogeneous Robin boundary conditions satisfied by $\theta^{\varepsilon}$. Indeed, we have used \eqref{eq:estiqpartialq}, among others.  Contrarily, with homogeneous Dirichlet boundary conditions on $q^\varepsilon$, we 	have
\begin{equation}
\label{eq:estiqpartialqDir}
\left|\int_{\partial\Oo}q^\varepsilon \dfrac{\partial q^\varepsilon}{\partial \nu}d \Gamma\right|=\left|\int_{\partial\Oo}\theta^1 \dfrac{\partial q^\varepsilon}{\partial \nu}d \Gamma\right|
\leq C \|q^\varepsilon\|_{H^2}\|  \theta^1 \|_{H^1}.
\end{equation}
But unfortunately, the norm $\|q^\varepsilon\|_{H^2}$ cannot be absorbed by the left hand side of \eqref{eq:BeforeGronwallNoSlip}.
\end{remark}

\subsection{The case of Navier slip-with-friction boundary conditions}\label{boundary_layer}

We come back to the general case, i.e. Navier slip-with-friction boundary conditions. 

\subsubsection{Ansatz with correction term}

Let us introduce a smooth function
	$\varphi: \R^n\mapsto\R$  such that $\varphi=0$ on $\partial \Oo$, $\varphi>0$ in $\Oo$, $\varphi<0$ outside of $\overline{\Oo}$ and  $|\varphi(x)|=dist(x,\partial \Oo)$ in a small neighborhood of $\partial \Oo$. Then, $\nu=-\nabla \varphi$ near $\partial \Oo$ and $\nu$ can be extended smoothly within the full domain $\Oo$. 
	
Following the original boundary layer expansion 
	for Navier slip-with-friction boundary conditions proved in \cite{iftimie_Sueur} by Iftimie and Sueur, we introduce the following expansions of the variables and the forcing terms: 
		\begin{equation}\label{expan}
		\left\{
  	 	 \begin{array}{lll}
			\ue(t,x)&=&u^0(t,x)+\sqrt{\varepsilon }\rho\left(  t,x,\dfrac{\varphi(x)}{\sqrt{\varepsilon}}\right)  +\varepsilon u^1(t,x)+\dots+\varepsilon r^{\varepsilon}(t,x),\\
		\noalign{\smallskip}\dis
			p^{\varepsilon}(t,x)&=&p^0(t,x)+\varepsilon p^1(t,x)+\dots+\varepsilon \pi^{\varepsilon}(t,x),\\
		\noalign{\smallskip}\dis
			\te^{\varepsilon}(t,x)&=&\te^0(t,x)+\varepsilon^2 \te^1(t,x)+\varepsilon^2 q^{\varepsilon}(t,x),
		\end{array}\right.
		\end{equation}
		\begin{equation*}
		\left\{
   		 \begin{array}{lll}
		v^{\varepsilon}(t,x)&=&v^0(t,x)+\sqrt{\varepsilon }v^\rho\left(  t,x,\dfrac{\varphi(x)}{\sqrt{\varepsilon}}\right) +\varepsilon v_1^1(t,x),\\
		\noalign{\smallskip}\dis
			w^{\varepsilon}(t,x)&=&w^0(t,x)+\varepsilon^2 w^1(t,x),\\
		\noalign{\smallskip}\dis
			\sigma^{\varepsilon}(t,x)&=&\sigma^0(t,x).
		\end{array}\right.
	\end{equation*}

 
    Compared to the previous expansion \eqref{expans}, since $u^0$ cannot satisfy the Navier slip-with-friction boundary condition on $\partial\Oo$, the expansion \eqref{expan} introduces a boundary correction $\rho$. This profile is expressed in terms of both the slow space variable $x \in \Oo$ and a fast scalar variable $z=\varphi(x)/\sqrt{\varepsilon}$. In the equations of \eqref{expan}, the missing terms 
    will help us to prove that the remainder is small; the details are given in Section \ref{sec_est_rem}. We use the profiles $(u^0, \theta^0)$ and $(u^1, \theta^1)$ (extended by zero for $t>T$) introduced in the previous sections, see Sections \ref{sec:inviscid} and \ref{sec:flushing}. 
	The following sections are devoted to analyze and estimate    the terms of the expansion in \eqref{expan}.
	
	The boundary layer corrector will be given as the solution to an initial boundary value problem with a boundary condition associated to the extra variable.
	As in \cite{iftimie_Sueur}, the  boundary layer correction will be described 	by a tangential vector field $\rho=\rho(t,x,z)$ satisfying the equation:
	\begin{equation}\label{layer_equation}
		\left\{
	    \begin{array}{lll}
    			\partial_t \rho+[(u^0\cdot\nabla)\rho+(\rho\cdot\nabla)u^0]_{tan}+u_\flat^0z\partial_z\rho-\partial_{zz}\rho = v^\rho&\hbox{in}& \R_+\times\overline{\Oo}\times\R_+,\\
    		\noalign{\smallskip}\dis
    			\partial_z \rho(t,x,0)=g^0(t,x)&\hbox{in}& \R_+\times\overline{\Oo},\\
    		\noalign{\smallskip}\dis
    			\rho(0,x,z)=0&\hbox{in}& \overline{\Oo}\times \R_+,
  	  \end{array}
 	   \right.
	\end{equation}
	where we have used the following notation:
	\begin{eqnarray}
    		u^0_\flat(t,x):= -\dfrac{u^0(t,x)\cdot \nu(x)}{\varphi(x)} &\hbox{in}& \R_+\times\Oo,\label{layer_cond_1}\\
    	\noalign{\smallskip}\dis
    		g^0(t,x):=2\chi(x)N(u^0)(t,x)&\hbox{in}& \R_+\times\Oo,\label{layer_cond_2}
\end{eqnarray}
	with a smooth cut-off function $\chi$ satisfying $\chi = 1$ in a neighbourhood of the boundary $\partial\Oo$.

\indent We can formally obtain \eqref{layer_equation} by plugging the expansion $u^0 + \sqrt{\varepsilon}\rho(t,x, \varphi(x)/\sqrt{\varepsilon})$ into 
\eqref{equa_u_layer_} and keeping the terms of order $\sqrt{\varepsilon}$.\\
\indent 
The following points are in order:
\begin{itemize}
\item 
$v^\rho$ must be viewed as a smooth control whose spatial support is located outside of $\overline{\Omega}$. With the help of the transport term, this control will enable us to modify the behavior of $\rho$ inside the physical domain $\Omega$.

\item 
$\rho$ depends on $n + 1$ spatial variables ($n$ slow variables $x_i$ and one fast variable $z$); it is thus not set in curvilinear coordinates. It is implicitly  assumed that $\nu$ actually refers to the extension $-\nabla \varphi$ of the normal and, in turn, this furnishes extensions of the identities in \eqref{def_N}.

\item We will check that the construction above satisfies $v^\rho\cdot \nu=0$. Since the equation is linear, it preserves the relation $\rho(0,x,z)\cdot \nu(x)=0$ at initial time. Thus, the boundary profile will be tangential, even inside the domain. Actually, this is the reason why the equation \eqref{layer_equation} is linear; 
see \cite[Section 2]{iftimie_Sueur} for more details.

\item In \eqref{layer_cond_2}, the role of the function $\chi$ is to ensure that $\rho$ is compactly supported near $\partial\Oo$.

\item Since $u^0$ is smooth and tangent to the boundary, a Taylor expansion proves that $u^0_\flat$ is smooth in $\overline{\Oo}$.

\item The boundary layer profile $\rho$ does not depend on $\varepsilon$.
\end{itemize}

\subsubsection{Well-prepared dissipation method}

	Unlike in the previous section, where $T$ is the fixed time control, we will use here virtually long time intervals $[0,T/\varepsilon]$ to dissipate the boundary layer.\\
	\indent The most natural strategy would be to use that $u^0$ is equal to $0$ after time $T$. Then \eqref{layer_equation} would be reduced to a heat equation posed on the half line $\R^+$ with homogeneous Neumann boundary conditions and the boundary layer would decay. Unfortunately, this decay is too slow: one can prove that $\sqrt{\varepsilon} \rho(T/\varepsilon,\cdot\,, \varphi(\,\cdot\,)/\sqrt{\varepsilon}) = O(\varepsilon)$, see \cite[Section 3.2]{coron_marbach}.
	Therefore, by dividing by $\varepsilon$, $u(T,\cdot\,) = O(1)$ and this is not enough for using the local result at the end.\\
	\indent This is why we use the source $v^{\rho}$ to prepare the dissipation of the boundary layer.\\
	\indent Let us define the following weighted Sobolev spaces
	\begin{equation}
	    H^{s, k}(\R) := \left\{f \in H^s(\R)\ ;\ \sum_{|\alpha| = 0}^{s} \int_{\R} (1+|z|^2)^k |\partial^{\alpha} f(z)|^2 dz < +\infty\right\},
	\end{equation}
	endowed with the corresponding (natural) norms. In \cite[Lemma $7$]{coron_marbach}, the following result is proved:
\begin{lema}\label{lemma_limita_boundary_layer}
Let $k\geq 1$ and $u^0 \in C^\infty([0,T]\times \overline{\Oo})$ be a fixed reference flow in \eqref{solu_line}. 
There exists $v^\rho \in C^\infty(\R_+\times\overline{\Oo}\times\R_+)$ with $v^\rho\cdot \nu=0$, such that the $x$-support is included in 
$\overline{\Oo}\setminus \overline{\Omega}$, the time support is compact in $(0,T)$ and, for any $s, p \in \NN$ and any $0\leq m\leq k$, the associated boundary layer profile satisfies:
\begin{equation}\label{est_lemma_rho}
|\rho(t,\cdot\,,\cdot\,)|_{H^p_x(H^{s,m}_z)}\leq C \left| \dfrac{\log(2+t)}{2+t}\right|^{\frac{1}{4}+\frac{k}{2}-\frac{m}{2}},
\end{equation}
where the positive constant $C$ depends on $p$, $s$, $m$ and $u^0$ but not on $t$.
\end{lema}
The interest of Lemma \ref{lemma_limita_boundary_layer} is twofold.
\begin{itemize}
    \item 
    The estimate \eqref{est_lemma_rho} will be used to show that the source terms generated by the boundary layer are integrable in long time and the equation satisfied by the remainder term is well-posed.
    \item It will be also used to prove that the boundary layer is sufficiently small at time $T/\varepsilon$.
\end{itemize}

\begin{remark} \rm
    A more ambitious idea would be to design a control strategy to get exactly\break $\rho(T, \cdot, \varphi(\cdot\,)/\sqrt{\varepsilon}) \equiv 0$. But, unfortunately, it can be proved that  \eqref{layer_equation} is not  null-controllable at time $T$, see \cite[Section $3.5$]{coron_marbach}.
\end{remark}

\subsubsection{Technical profiles}\label{est_ramainder}

	\indent For a function $f=f(t,x,z)$, we will use the notation $\{f\}$ to denote its values at points $(t,x,z)$ with $z=\varphi(x)/\sqrt{\varepsilon}$. The full decomposition will be the following
\begin{equation}\label{full_expan}
\begin{array}{ccl}
	 u^\varepsilon&=&u^0+\sqrt{\varepsilon}\{\rho\}+\varepsilon u^1+\varepsilon\nabla \zeta^\varepsilon+\varepsilon\{\beta\}+\varepsilon r^\varepsilon,\\
    		\noalign{\smallskip}\dis
	p^\varepsilon&=&p^0+\varepsilon\{\psi\}+\varepsilon p^1+\varepsilon \mu^\varepsilon+\varepsilon\pi^\varepsilon,\\
    		\noalign{\smallskip}\dis
    		\te^{\varepsilon}&=&\te^0+\varepsilon^2 \te^1+\varepsilon^2 q^{\varepsilon},\\
    		\noalign{\smallskip}\dis
    	v^{\varepsilon}&=&v^0+\sqrt{\varepsilon }\{v^\rho\} +\varepsilon v^1,\\
    	    	\noalign{\smallskip}\dis
    	w^{\varepsilon}&=&w^0+\varepsilon^2 w^1,\\
    	\noalign{\smallskip}\dis
	\sigma^{\varepsilon}&=&\sigma^0.
\end{array}
\end{equation}

The functions $\beta$, $\zeta^\varepsilon$ and $\psi$ are given as follows: 
\begin{equation}\label{equation_beta}
		\beta(t,x,z)=-2e^{-z}N(\rho)(t,x,0)-\nu(x)\int_z^{+\infty}\nabla_x\cdot \rho(t,x,z^\prime)dz^\prime,
\end{equation}

\begin{equation}\label{fusca}
	\left\{
    \begin{array}{lll}
    		\Delta \zeta^\varepsilon= -\{\nabla\cdot \beta\}&\hbox{in}& \Oo,\\
    	\noalign{\smallskip}\dis
    		\partial_\nu\zeta^\varepsilon=-\beta(t,\,\cdot\,,0)\cdot \nu&\hbox{on}& \partial\Oo,
    \end{array}
    \right.
\end{equation}

\begin{equation}\label{tang_part_layer}
\psi=\psi(t,x,z)~\,\text{satisfies}\,~[(u^0\cdot \nabla)\rho+(\rho\cdot\nabla)u^0]\cdot \nu=\partial_z\psi \,~\text{and}~\,\psi(t,x,z)\longrightarrow0\quad\text{as} \quad z\longrightarrow +\infty.
\end{equation}

It is proved in \cite[Section $4.2$]{coron_marbach} that the definitions \eqref{equation_beta}, \eqref{fusca} and \eqref{tang_part_layer} are compatible with \eqref{equa_u_layer_} and, furthermore, the following estimates hold:

\begin{align}\label{est_tec_profile_1}
\|\beta(t,\cdot,\cdot\,)\|_{H_x^p(H^{s,k}_z)}&\leq C\|\rho(t,\cdot,\cdot\,)\|_{H_x^{p+1}(H_z^{s+1,k+2})},\\
\|\zeta^\varepsilon(t,\cdot\,)\|_{H^4}&\leq C\left(\varepsilon^{-3/4}\|\beta(t,\cdot\,,\cdot\,)\|_{H_x^{4}(H_z^{2,0})}+\|\rho(t,\cdot\,,\cdot\,)\|_{H_x^{3}(H_z^{0,1})}\right),\\
\|\zeta^\varepsilon(t,\cdot\,)\|_{H^3}&\leq C\left(\varepsilon^{-1/4}\|\beta(t,\cdot\,,\cdot\,)\|_{H_x^{3}(H_z^{1,0})}+\|\rho(t,\cdot\,,\cdot\,)\|_{H_x^{2}(H_z^{0,1})}\right),\\
\|\zeta^\varepsilon(t,\cdot\,)\|_{H^2}&\leq C\left(\varepsilon^{1/4}\|\beta(t,\cdot\,,\cdot\,)\|_{H_x^{2}(H_z^{0,0})}+\|\rho(t,\cdot\,,\cdot\,)\|_{H_x^{1}(H_z^{0,1})}\right)\label{eq:estizetaH2},\\
\label{est_tec_profile_ult}
\|\psi(t,\cdot,\cdot\,)\|_{H^1_x(H_z^{0,0})}&\leq C\|\rho(t,\cdot,\cdot\,)\|_{H^2_x(H_z^{0,2})}.
\end{align}

\subsubsection{Equation and estimates of the remainder}\label{sec_est_rem}

We will now analyze the remainder defined in \eqref{full_expan}, which is in fact a solution in the extended domain $\Oo$ to
\begin{equation}\label{remain_equa_friction}
	\left\{
    \begin{array}{lll}
    		\partial_t r^\varepsilon-\varepsilon\Delta r^\varepsilon+(u^\varepsilon\cdot\nabla)r^\varepsilon+\nabla \pi^\varepsilon = \{f^\varepsilon\}-\{A^\varepsilon r^\varepsilon\}+\varepsilon q^\varepsilon e_n+\varepsilon\te^1e_n&\hbox{in}& (0,+\infty)\times\Oo,\\
    	\noalign{\smallskip}\dis
    		\partial_t q^{\varepsilon}-\varepsilon\Delta q^\varepsilon+u^\varepsilon\cdot \nabla q^\varepsilon=\{h^\varepsilon\}- B^\varepsilon r^\varepsilon&\hbox{in}& (0,+\infty)\times\Oo,\\
    	\noalign{\smallskip}\dis
    		\nabla\cdot r^\varepsilon=0&\hbox{in}& (0,+\infty)\times\Oo,\\
    	\noalign{\smallskip}\dis
	    r^\varepsilon\cdot \nu=0, \quad    N(r^\varepsilon)=-N(g^\varepsilon)  &\hbox{on}& (0,+\infty)\times\partial\Oo,\\
	\noalign{\smallskip}\dis
	    R( q^\varepsilon)=-R( \te^1)&\hbox{on}& (0,+\infty)\times\partial\Oo,\\
	\noalign{\smallskip}\dis
	    r^\varepsilon(0,\cdot\,)=0, \quad q^\varepsilon(0,\cdot\,)=0&\hbox{in}& \Oo,
    \end{array}
    \right.
\end{equation}
where $g^\varepsilon:=u^1+\nabla\zeta^\varepsilon+\beta|_{z=0}$.
Let us introduce the amplification operators $A^\varepsilon$ and $B^\varepsilon$, given by
\begin{equation}\label{def_A_}
A^\varepsilon r^\varepsilon:=(r^\varepsilon\cdot\nabla)(u^0+\sqrt{\varepsilon}\rho+\varepsilon u^1+\varepsilon\nabla\zeta^\varepsilon+\varepsilon \beta)-(r^\varepsilon\cdot \nu)(\partial_z\rho+\sqrt{\varepsilon}\partial_z\beta)
\end{equation}
and
\begin{equation}
\label{eq:defB}
B^\varepsilon r^\varepsilon:=\varepsilon r^\varepsilon\cdot\nabla \theta^1
\end{equation}
and the forcing terms $f^\varepsilon$ and $h^\varepsilon$, with
\begin{equation}\label{def_f_}
\begin{array}{rl}
f^\varepsilon :=&\!\!\!(\Delta \varphi \partial_z \rho -2(\nu\cdot\nabla)\partial_z\rho+\partial_{zz} \beta)+\sqrt{\varepsilon}(\Delta\rho+\Delta\varphi\partial_z \beta-2(\nu\cdot\nabla)\partial_z \beta)\\
\noalign{\smallskip}\dis
&\!\!\!+\varepsilon(\Delta \beta+\Delta u^1 +\Delta\nabla\zeta^\varepsilon)-((\rho+\sqrt{\varepsilon}(\beta+u^1+\nabla\zeta^\varepsilon))\cdot\nabla)(\rho+\sqrt{\varepsilon}(\beta+u^1+\nabla\zeta^\varepsilon))\\
\noalign{\smallskip}\dis
&\!\!\!-(u^0\cdot\nabla)\beta-(\beta\cdot\nabla)u^0-u^0_\flat z\partial_z \beta+(\beta+u^1+\nabla\zeta^\varepsilon)\cdot \nu\partial_z(\rho+\sqrt{\varepsilon}\beta)\\
\noalign{\smallskip}\dis
&\!\!\!-\nabla \psi-\partial_t \beta
\end{array}
\end{equation}
and
\begin{equation}\label{def_h_}
h^\varepsilon:=\varepsilon\Delta\theta^1-(\sqrt{\varepsilon}\rho+\varepsilon(u^1+\nabla\zeta^\varepsilon+\beta))\cdot\nabla\te^1.
\end{equation}

We have to estimate the size of the remainder $(r^\varepsilon,q^\varepsilon)$ at final time and check that it is small. 
	We begin by establishing an energy estimate. 
	Thus, we multiply equation \eqref{remain_equa_friction}$_1$ by $r^\varepsilon$ and the equation \eqref{remain_equa_friction}$_2$ by $q^\varepsilon$ and integrate by parts. We proceed as before, term by term,  
%
the unique different being the terms coming from the boundary. 

We recall the following identity, see \cite[Lemma 2.2]{iftimie_Sueur} which will be used throughout the paper 
\begin{equation}
    \label{eq:identity_useful}
    \int_{\Oo} (-\Delta u) \cdot v = 2 \int_{\Oo} D(u) \cdot D (v) - 2 \int_{\partial\Oo} [D(u)\nu 
  ]_{tan} \cdot v \,  d\Gamma_{\mathcal{O}},
\end{equation}
where $u$ and $v$ are smooth vector fields such that $v$ is divergence free and tangent to the boundary. 

 Therefore, it follows that 
\begin{equation*}
-\varepsilon\int_\Oo \Delta r^\varepsilon\cdot r^\varepsilon=2\varepsilon\|D(r^\varepsilon)\|^2 +2\varepsilon  \int_{\partial\Oo} \left([M r^\varepsilon]_{tan} -N(g^\varepsilon)\right)\cdot r^\varepsilon \,  d\Gamma_{\mathcal{O}},
\end{equation*}
and we estimate the boundary term as follows 
\begin{equation}\label{est_r1}
\!\!\!\!\!\begin{array}{rl}
2\left|\displaystyle\int_{\partial\Oo}  \left([M r^\varepsilon]_{tan} -N(g^\varepsilon)\right) \cdot r^\varepsilon \, d\Gamma_{\mathcal{O}}\right|=&\!\!\!2\left|\displaystyle\int_{\partial \Oo}Mr^\varepsilon\cdot r^\varepsilon-N(g^\varepsilon)\cdot r^\varepsilon \,  d\Gamma_{\mathcal{O}}\right|\\
\noalign{\smallskip}\dis
\leq&\!\!\!\lambda\|\nabla r^\varepsilon\|^2+C_\lambda(\|r^\varepsilon\|^2+\|N(g^\varepsilon)\|^2_{L^2(\partial\Oo)})\\
\noalign{\smallskip}\dis
\leq&\!\!\!\lambda\|\nabla r^\varepsilon\|^2+C_\lambda(\|r^\varepsilon\|^2+\|g^\varepsilon\|^2_{L^2(\partial\Oo)}+\|D (g^\varepsilon)\|^2_{L^2(\partial\Oo)})\\
\noalign{\smallskip}\dis
\leq&\!\!\!\lambda\|\nabla r^\varepsilon\|^2+C_\lambda(\|r^\varepsilon\|^2+\|g^\varepsilon\|^2_{H^2}),
\end{array}
\end{equation}
for any $\lambda>0$, where $C_\lambda$ is a constant depending  on $\lambda$. Let us absorb the term $\|\nabla r^\varepsilon\|^2$ in \eqref{est_r1}. Thanks to the classical Korn's inequality (see Lemma \ref{lemma:korn_inequality}), since $\nabla\cdot r^\varepsilon=0$ in $\Oo$ and $r^\varepsilon\cdot \nu=0$ on $\partial \Oo$, we have 
\begin{equation*}\label{korn}
\|r^\varepsilon\|^2_{H^1}\leq C_K\|r^\varepsilon\|^2+C_K\|D (r^\varepsilon)\|^2.
\end{equation*}
for some $C_K>0$. Choosing $\lambda=1/(2C_K)$, we get:
\begin{align*}
\dfrac{d}{dt}\|r^\varepsilon\|^2+\varepsilon\|D( r^\varepsilon)\|^2&\leq \Bigl(\|\sigma^0\|_{\infty}+C\varepsilon+\|\{f^\varepsilon\}\|+2\|\{A^\varepsilon\}\|_\infty \Bigr)\|r^\varepsilon\|^2\\
&\qquad+\Bigl( C\varepsilon\|g^\varepsilon\|^2_{H^2}+\|\{f^\varepsilon\}\|+\varepsilon\|\te^1\|^2\Bigr)+\varepsilon\|q^\varepsilon\|^2\\
\end{align*}
and
\begin{align*}
\dfrac{d}{dt}\|q^\varepsilon\|^2+\varepsilon\|\nabla q^\varepsilon\|^2&\leq\Bigl(\|\sigma^0\|_\infty+\|\{h^\varepsilon\}\|+\|\{B^\varepsilon\}\|_\infty+C\varepsilon\Bigr)\|q^\varepsilon\|^2\\
&\qquad +\Bigl(\|\{h^\varepsilon\}\|+C\varepsilon\|\te^1\|_{H^2}^2\Bigr)+ \|\{B^\varepsilon\}\|_\infty\|r^\varepsilon\|^2.
\end{align*}     
Adding the two estimates above, we get
\begin{align*}
&\dfrac{d}{dt}(\|r^\varepsilon\|^2+\|q^\varepsilon\|^2)+\varepsilon(\|D( r^\varepsilon)\|^2+\!\|\nabla q^\varepsilon\|^2)\\
&\qquad\leq
\Bigl(\|\sigma^0\|_{\infty}+C\varepsilon+\|\{f^\varepsilon\}\|+2\|\{A^\varepsilon\}\|_\infty+\|\{h^\varepsilon\}\|+\|B^\varepsilon\|_\infty\Bigr)\\
&\qquad\times\Bigl(\|r^\varepsilon\|^2+\|q^\varepsilon\|^2\Bigr)+\Bigl(C\varepsilon\|g^\varepsilon\|^2_{H^2}+\|\{f^\varepsilon\}\|+\|\{h^\varepsilon\}\|+C\varepsilon\|\te^1\|_{H^2}^2\Bigr).       
\end{align*}              
Applying Gronwall's inequality in the interval $(0,T/\varepsilon)$, and using the fact that the initial datum is equal to $0$ and the estimates
\begin{eqnarray}
\|\{A^\varepsilon\}\|_{L^1(L^\infty)}
+\|B^\varepsilon\|_{L^1(L^\infty)}&=&O(1),\label{est_A_rem}\\
\noalign{\smallskip}\dis
\varepsilon\|\te^1\|^2_{L^2(H^2)}+\varepsilon\|g^\varepsilon\|^2_{L^2(H^2)}&=&O(\varepsilon^{\frac{1}{4}}),\label{est_g_rem}\\
\noalign{\smallskip}\dis
\|\{f^\varepsilon\}\|_{L^1(L^2)}+\|\{h^\varepsilon\}\|_{L^1(L^2)}&=&O(\varepsilon^{\frac{1}{4}}),\label{est_f_rem}
\end{eqnarray}
we obtain:
\begin{equation}\label{gron_re_qe_friction}
\|r^\varepsilon\|^2_{L^\infty(L^2)}+\|q^\varepsilon\|^2_{L^\infty(L^2)}+\varepsilon\|D( r^\varepsilon)\|^2_{L^2(L^2)}+\varepsilon\|\nabla q^\varepsilon\|^2_{L^2(L^2)}=O(\varepsilon^{\frac{1}{4}}).
\end{equation}

The estimates \eqref{est_A_rem}, \eqref{est_g_rem} and \eqref{est_f_rem} hold on the whole interval $[0,+\infty)$. The estimates for $\{A^\varepsilon\}$, $g^\varepsilon$ and $\{f^\varepsilon\}$ can be found in \cite[Section 4.4]{coron_marbach}.
Here, we give some details to obtain the estimates for  $B^\varepsilon$, $\theta^1$ and $\{h^\varepsilon\}$, which are new. \\
\indent First, the estimates of  $B^\varepsilon$ and $\theta^1$ are straightforward by using \eqref{eq:defB}, and Lemma \ref{exist_sol_u1}. Indeed, we easily get that $\|B^\varepsilon\|_{L^1(L^\infty)} = O(1)$~and that~$\varepsilon\|\te^1\|^2_{L^2(H^2)} =O(\varepsilon)$.\\
\indent Now, let us justify the estimate of $\{h^\varepsilon\}$. Note that the fast scaling variable enables us to win a factor $\varepsilon^{1/4}$, see \cite[Lemma 3]{iftimie_Sueur}. In what follows, we estimate each one of the terms in \eqref{def_h_}. The first term is $O(\varepsilon)$, thanks to the regularity of $\theta^1$. The second one can be treated as follows
\begin{align*}
   \|\sqrt{\varepsilon}\{\rho(t,\cdot\,)\} \cdot \nabla \theta^1(t,\cdot\,)\|& \leq C \sqrt{\varepsilon} \|\{\rho(t,\cdot\,)\}\|_{H^1}  \|\nabla \theta^1(t,\cdot\,)\|_{H^1}\\
   & \leq C \left(\sqrt{\varepsilon} \|\rho(t,\cdot\,,\cdot\,)\|_{H_x^1 (H_{z}^{0,0})} + \|\{\partial_z \rho(t,\cdot)\}\| \right) \|\nabla \theta^1(t,\cdot\,)\|_{H^1}\\
   & \leq C \left(\sqrt{\varepsilon} \|\rho(t,\cdot\,,\cdot\,)\|_{H_x^1( H_{z}^{0,0})}  + \varepsilon^{1/4} \|\rho(t,\cdot\,,\cdot\,)\|_{H_x^1( H_{z}^{1,0})}\right) \|\nabla \theta^1(t,\cdot\,)\|_{H^1}\\
   & \leq C \varepsilon^{1/4} \|\rho(t,\cdot\,,\cdot\,)\|_{H_x^1 (H_z^{1,0})}\|\nabla \theta^1(t,\cdot\,)\|_{H^1}.
\end{align*}
Then, integrating by parts  this last inequality with respect to time over $(0,T/\varepsilon)$, using Lemma \ref{lemma_limita_boundary_layer} for $k=4$ 
 and also the fact that $\theta^1$ is bounded in $L^{\infty}(0,T;H^3(\Oo))$, we get: 
\begin{equation*}
    \|\sqrt{\varepsilon}\{\rho(\cdot\,,\cdot\,)\} \cdot \nabla \theta^1\|_{L^1 (L^2)} = O(\varepsilon^{1/4}).
\end{equation*}
 Now, for the third term, by  using \eqref{est_tec_profile_1} and \eqref{eq:estizetaH2} we have the following:
\begin{align*}
   \|\varepsilon\nabla_x \zeta^{\varepsilon}(t,\cdot\,) \cdot \nabla \theta^1(t,\cdot\,)\|& \leq C \varepsilon \|\nabla_x \zeta^{\varepsilon}(t,\cdot\,)\|_{H^1}  \|\nabla \theta^1(t,\cdot\,)\|_{H^1}\\
   & \leq C \varepsilon \|\zeta^{\varepsilon}(t,\cdot\,)\|_{H^2}\|\nabla \theta^1(t,\cdot\,)\|_{H^1}\\
   & \leq C \varepsilon\left(\varepsilon^{1/4}\|\beta(t,\cdot\,,\cdot\,)\|_{H_x^{2}(H_z^{0,0})}+\|\rho(t,\cdot\,,\cdot\,)\|_{H_x^{1}(H_z^{0,1})}\right)\|\nabla \theta^1(t,\cdot\,)\|_{H^1}\\
   & \leq C \varepsilon \|\rho(t,\cdot\,,\cdot\,)\|_{H_x^3 (H_z^{1,2})}\|\nabla \theta^1(t,\cdot\,)\|_{H^1}.
\end{align*}
Integrating  by parts this last inequality, with respect to time, and using again Lemma \ref{lemma_limita_boundary_layer} for $k=3$ and the fact that $\theta^1$ is bounded in $L^{\infty}(0,T;H^3(\Oo))$, we find that 
\begin{equation*}
    \|\varepsilon\nabla_x \zeta^{\varepsilon}(\cdot\,,\cdot\,) \cdot \nabla \theta^1(\cdot,\cdot\,)\|_{L^1 (L^2)} = O(\varepsilon^{1/4}).
\end{equation*}
The last term can be estimated in a similar way, using \eqref{est_tec_profile_1}.

\subsection{Towards the trajectory}\label{sec_appro_traj}
	
	In this section, we deduce a small-time global approximate controllability result to smooth trajectories. For that, we will use Lemma \ref{lemma_limita_boundary_layer} and the estimates on the remainder term \eqref{gron_re_qe_friction}. 

	Let $(u^\varepsilon,p^\varepsilon, \te^\varepsilon)$ be the solution to the equation \eqref{equa_u_layer_} on the time interval $[0,T/\varepsilon]$.
	 First, during the interval $[0,T]$, we use the expansions
\begin{equation}\label{expan_1}
	\begin{array}{ccl}
	u^\varepsilon&=&u^0+\sqrt{\varepsilon}\{\rho\}+\varepsilon u^{1,\varepsilon}+\varepsilon\nabla \zeta^\varepsilon+\varepsilon\{\beta\}+\varepsilon r^\varepsilon,\\
	\noalign{\smallskip}\dis
	\te^\varepsilon&=&\te^0+\varepsilon^2 \te^{1,\varepsilon}+\varepsilon^2 q^\varepsilon,
	\end{array}
\end{equation}
	where $u^{1,\varepsilon}(0,\cdot\,)=u_*$, $\te^{1,\varepsilon}(0,\cdot\,)=\te_*$ and $u^{1,\varepsilon}(T,\cdot\,)=\overline{u}(\varepsilon T,\cdot\,)$, $\te^{1,\varepsilon}(T,\,\cdot\,)=\overline{\te}(\varepsilon T,\,\cdot\,)$. The couple $(u^{1,\varepsilon},\te^{1,\varepsilon})$ solves, together with some $p^{1,\varepsilon}$, the usual first-order system \eqref{u1_equa} and the profiles $u^{1,\varepsilon}$ and $\te^{1,\varepsilon}$ depend on $\varepsilon$. However, since the reference trajectory belongs to $C^\infty$, all the required estimates can be made independent of $\varepsilon$.  In a second step, for large times $t\geq T$, we change our expansions and set:
\begin{equation}\label{expan_2}
	\begin{array}{ccl}
	u^\varepsilon&=&\sqrt{\varepsilon}\{\rho\}+\varepsilon \overline{u}(\varepsilon t,\,\cdot\,)+\varepsilon\nabla \zeta^\varepsilon+\varepsilon\{\beta\}+\varepsilon r^\varepsilon,\\
	\noalign{\smallskip}\dis
	p^\varepsilon&=&\varepsilon^2 \overline{p}(\varepsilon t,\,\cdot\,)+\varepsilon \mu^\varepsilon+\varepsilon\pi^\varepsilon,\\
	\noalign{\smallskip}\dis
	v^\varepsilon&=&\sqrt{\varepsilon}v^\rho+\varepsilon^2\overline{v},\\
	\noalign{\smallskip}\dis
	\te^\varepsilon&=&\varepsilon^2 \overline{\te}(\varepsilon t,\,\cdot\,)+\varepsilon^2  q^\varepsilon,\\
    	\noalign{\smallskip}\dis
    	w^{\varepsilon}&=&\varepsilon^3 \overline{w}.
\end{array}
\end{equation}
Note that, for $t\geq T$, we have $u^0=0$ and the profile $(u^1,\te^1)$ is the main trajectory and changing \eqref{expan_1} by \eqref{expan_2} allow us to get rid of some 
	terms in the equation satisfied by the remainder. Indeed, terms such as $\varepsilon \Delta u^1$, $\varepsilon(u^1\cdot \nabla)u^1$, $\varepsilon u^1\cdot \nabla \te^1$ and $\varepsilon\Delta \te^1$ will not appear any more in \eqref{def_f_} and \eqref{def_h_} because they are already taken into account by $\left(\overline{u},\overline{\te}\right)$. Actually, despite the presence of the profile $(u^1,\te^1)$  in both steps, the estimates obtained for the remainder profile are as in Section \ref{sec_est_rem}.
	
	 Let us introduce
\begin{equation*}
	u^{(\varepsilon)}(t,x):=\frac{1}{\varepsilon}u^\varepsilon\left(\frac{t}{\varepsilon},x\right) 
	\ \ \ \ \text{and} \ \ \ \ 
	\te^{(\varepsilon)}(t,x):=\frac{1}{\varepsilon^2}\te^\varepsilon\left(\frac{t}{\varepsilon},x\right).
\end{equation*}
	Then, 
	thanks to \eqref{est_tec_profile_1}, \eqref{eq:estizetaH2} and \eqref{gron_re_qe_friction}, we see that
\begin{align*}
  \left\|u^{(\varepsilon)}(T,\cdot\,)-\overline{u}(T,\cdot\,)\right\| &= \left\|\varepsilon^{-1/2}       \left\{\rho\left(T/\varepsilon,\cdot\,\right)\right\}+\nabla \zeta^\varepsilon(T/\varepsilon,\cdot\,)+
		 \left\{\beta\left(T/\varepsilon,\cdot\,\right)\right\} + r^\varepsilon\left(T/\varepsilon,\cdot\,\right)\right\| \\
   & \leq \varepsilon^{-1/2}   \left\|     \left\{\rho\left(T/\varepsilon,\cdot\,\right)\right\}\right\|+
   \varepsilon^{1/4}\|\beta(T/\varepsilon,\cdot\,,\cdot\,)\|_{H_x^{2}(H_z^{0,0})}\\
   &\,\,\,\,\,\,\,\,\,+\|\rho(T/\varepsilon,\cdot\,,\cdot\,)\|_{H_x^{1}(H_z^{0,1})}+\left\|\left\{\beta\left(T/\varepsilon,\cdot\,\right)\right\} \right\|+
   \left\|r^\varepsilon\left(T/\varepsilon,\cdot\,\right) \right\|\\
   &\leq \varepsilon^{-1/2}   \left\|     \left\{\rho\left(T/\varepsilon,\cdot\,\right)\right\}\right\|+
   \varepsilon^{1/4}\|\rho(T/\varepsilon,\cdot\,,\cdot\,)\|_{H_x^{3}(H_z^{1,2})}\\
   &\,\,\,\,\,\,\,\,\,+\|\rho(T/\varepsilon,\cdot\,,\cdot\,)\|_{H_x^{1}(H_z^{0,1})}+\left\|\rho\left(T/\varepsilon,\cdot\,,\cdot\,\right) \right\|_{H_x^{1}(H_z^{1,2})}+O(\varepsilon^{\frac{1}{8}}).
\end{align*}
	 We can use \eqref{est_lemma_rho} to estimate the terms containing $\rho$ in the estimates above. First, note that
	 \[
	 \lim_{s\to+\infty}\dfrac{\log s}{s^{1/2}}=0
	 \]
	 and, consequently, there exists a positive constant $C>0$ such that 
	 \[
	 \dfrac{\log s}{s}\leq Cs^{-1/2} \quad \forall s \geq 1.
	 \]
	 Then, by taking $\varepsilon$ sufficiently small, the following is found for $k \geq 2$:
	 \[
	 \begin{array}{l}
	 \varepsilon^{-\frac{1}{2}}\left\|     \left\{\rho\left(T/\varepsilon,\cdot\,\right)\right\}\right\|
	 =\varepsilon^{-\frac{1}{2}}\|\rho(T/\varepsilon,\cdot\,,\cdot\,)\|_{H^0_x(H^{0,0}_z)}
	 \leq C \varepsilon^{-\frac{1}{2}}\left| \dfrac{\log(2+T/\varepsilon)}{2+T/\varepsilon}\right|^{\frac{1}{4}+\frac{k}{2}}\leq C \varepsilon^{-\frac{1}{2}+ \frac{1}{8}+\frac{k}{4}},\\
	 	\noalign{\smallskip}\dis
	 \varepsilon^{\frac{1}{4}}\|\rho(T/\varepsilon,\cdot\,,\cdot\,)\|_{H_x^{3}(H_z^{1,2})}
	 \leq C \varepsilon^{\frac{1}{4}}\left| \dfrac{\log(2+T/\varepsilon)}{2+T/\varepsilon}\right|^{\frac{1}{4}+\frac{k}{2}-1}
	 \leq C \varepsilon^{\frac{1}{4}+\frac{1}{8}+\frac{k}{4}-\frac{1}{2}},\\
	 	\noalign{\smallskip}\dis
	 \|\rho(T/\varepsilon,\cdot\,,\cdot\,)\|_{H_x^{1}(H_z^{0,1})}
	 \leq C \left| \dfrac{\log(2+T/\varepsilon)}{2+T/\varepsilon}\right|^{\frac{1}{4}+\frac{k}{2}-\frac{1}{2}}
	 \leq C \varepsilon^{\frac{1}{8}+\frac{k}{4}-\frac{1}{4}},\\
	 	\noalign{\smallskip}\dis
	 |\rho(T/\varepsilon,\cdot\,,\cdot\,)|_{H_x^{1}(H_z^{1,2})}
	 \leq C \left| \dfrac{\log(2+T/\varepsilon)}{2+T/\varepsilon}\right|^{\frac{1}{4}+\frac{k}{2}-1}
	 \leq C \varepsilon^{\frac{1}{8}+\frac{k}{4}-\frac{1}{2}}.
	\end{array}
	 \]
Finally, we choosing $k$ large enough, we conclude that
\begin{equation*}
	\left\|u^{(\varepsilon)}(T,\cdot\,)-\overline{u}(T,\cdot\,)\right\|=O(\varepsilon^{\frac{1}{8}})
\end{equation*}
	and, from \eqref{gron_re_qe_friction}, we have
		\begin{equation*}
		\left\|\te^{(\varepsilon)}(T,\cdot\,)-\overline{\te}(T,\cdot\,)\right\|=\left\|q^\varepsilon\left(T/\varepsilon,\cdot\,\right)\right\|= O(\varepsilon^{\frac{1}{8}}) .
	\end{equation*}
This concludes the proof of Proposition \ref{lemma_trajec_aprox} holds.


\section[Local control
 of the Boussinesq system]{Local controllability
 of the Boussinesq system with nonlinear boundary conditions}
%
\label{sec_local}


	
	Let  $\omega_c$ and $\omega$ be two non-empty open subsets such that $\omega_c \subset \subset \omega \subset \subset \Oo \setminus \overline{\Omega}$ and let $\chi_{\omega}$ be a cut-off function such that $\chi_{\omega}=0$ outside $\omega$ and $\chi_{\omega} = 1$ in $\omega_c$.

	\indent The goal of this section is to prove the local exact controllability to trajectories for the Boussinesq system with distributed controls:
\begin{equation}\label{equa_local}
	\left\{
		\begin{array}{lll}
    			\partial_t u-\Delta u+(u\cdot\nabla)u+\nabla p = \theta e_n + v \chi_{\omega} &\hbox{in} & \mathcal{O}_T,\\
    			\noalign{\smallskip}\dis
    			\partial_t \theta-\D \te+u\cdot \nabla \te=w \chi_{\omega} &\hbox{in} & \mathcal{O}_T,\\
    			\noalign{\smallskip}\dis
    			\nabla\cdot u=0 &\hbox{in} &\mathcal{O}_T,\\
    			\noalign{\smallskip}\dis
    			u\cdot \nu=0, \quad N(u)+ [f(u)]_{tan}=0
    			&\hbox{on} & \varLambda_T,\\
    			\noalign{\smallskip}\dis
    			R (\te)+g(\te)=0
    			&\hbox{on} & \varLambda_T,\\
    			\noalign{\smallskip}\dis
    			u(0,\cdot\,)=u_*, \quad \te(0,\cdot\,)=\te_* &\hbox{in} & \mathcal{O},
		\end{array}
	\right.
\end{equation}
where $f\in C^3(\R^n;\R^n)$ with a symmetric Jacobian matrix (or equivalently, $f$ is an irrotational field) and $ g\in C^3(\R)$. Note that, to prove Theorem \ref{theo_main}, we only need to prove a local-controllability result for \eqref{def_solu}, with linear Navier boundary conditions $N(u)$ as in \eqref{def_N} and linear Robin boundary conditions $R(\te)$ as \eqref{def_R}. However, for sake of completeness, we establish a local controllability result for the Boussinesq system with nonlinear Navier boundary conditions on the velocity field and nonlinear Fourier boundary conditions on the temperate. 

Since \eqref{equa_local} is nonlinear, we first begin by proving a (global) null-controllability result for the following system
\begin{equation}\label{equa_local_linear}
	\left\{
		\begin{array}{lll}
    			\partial_t z-\Delta z+((a+b)\cdot\nabla)z+(z\cdot\nabla)b+\nabla q = h e_n + v \chi_{\omega} &\hbox{in} & \mathcal{O}_T,\\
    			\noalign{\smallskip}\dis
    			\partial_t h-\D h+(a+b)\cdot \nabla h+z\cdot \nabla c=w \chi_{\omega} &\hbox{in} & \mathcal{O}_T,\\
    			\noalign{\smallskip}\dis
    			\nabla\cdot  z=0 &\hbox{in} &\mathcal{O}_T,\\
    			\noalign{\smallskip}\dis
    			z\cdot \nu=0, \quad [D(z)\nu+Az]_{tan}=0&\hbox{on} & \varLambda_T,\\
    			\noalign{\smallskip}\dis
    			{\partial h\over \partial \nu} + B h=0&\hbox{on} & \varLambda_T,\\
    			\noalign{\smallskip}\dis
    			z(0,\cdot\,)=z_*, \quad h(0,\cdot\,)=h_* &\hbox{in} & \mathcal{O},
		\end{array}
	\right.
\end{equation}
where the vector fields $a$, $b$, the scalar function $c$, the symmetric matrix $A$ and the scalar function $B$  satisfy the following assumptions:
\begin{align}
    &a, b \in L^{\infty}(0,T;L_{div}^2(\Oo)^n\cap L^{\infty}(\Oo)^n),\ a_t, b_t\in L^2(0,T;L^r(\mathcal{O})^n), \label{hipo_d}\\
    &  c \in L^{\infty}(0,T; L^{\infty}(\Oo)), \quad c_t\in L^2(0,T;L^r(\mathcal{O})),\label{hypod1}\\
    \label{hipo_A_3}
	&A\in P:=H^{1-\ell}(0,T;W^{\vartheta_1,\vartheta_1+1} (\partial \mathcal{O})^{n\times n}) \cap H^{(3-\ell)/2}(0,T;H^{\vartheta_2} (\partial \mathcal{O})^{n\times n}),\\
	    \label{hipo_B_3}
	& B\in Q:= H^{1-\ell}(0,T;W^{\vartheta_1,\vartheta_1+1} (\partial \mathcal{O})) \cap H^{(3-\ell)/2}(0,T;H^{\vartheta_2} (\partial \mathcal{O})),
\end{align}
 with $0 < \ell < 1/2$ arbitrarily close to $1/2$,
$r=2n$, $\vartheta_2 = (1/2)(3-n)+(1-\ell)(n-2)$  and $\vartheta_1> 1$ (arbitrarily small) if $n=3$ and  $\vartheta_1 = 1$ if $n =2$. By Sobolev embeddings, we readily have $P\hookrightarrow L^\infty((0,T)\times \partial \mathcal{O})^{n\times n}$ and $Q\hookrightarrow L^\infty((0,T)\times \partial \mathcal{O})$. \\
\indent It is well-known, that the null-controllability of system \eqref{equa_local_linear} is equivalent to prove an observability estimate for the adjoint system:
	
\begin{equation}\label{equa_local_adjunta}
	\left\{
		\begin{array}{lll}
    			-\partial_t \varphi-\Delta \varphi-(a\cdot\nabla)\varphi-D(\varphi) b+\nabla \pi = c\nabla \psi &\hbox{in} & \mathcal{O}_T,\\
    			\noalign{\smallskip}\dis
    			-\partial_t \psi -\D \psi-(a+b)\cdot \nabla \psi=\varphi \cdot e_n &\hbox{in} & \mathcal{O}_T,\\
    			\noalign{\smallskip}\dis
    			\nabla\cdot \varphi=0 &\hbox{in} &\mathcal{O}_T,\\
    			\noalign{\smallskip}\dis
    			\varphi\cdot \nu=0, \quad [D(\varphi)\nu+A\varphi]_{tan}=0&\hbox{on} & \varLambda_T,\\
    			\noalign{\smallskip}\dis
{\partial \psi\over \partial \nu} + B \psi=0&\hbox{on} & \varLambda_T,\\
    			\noalign{\smallskip}\dis
    			\varphi(T,\cdot\,)=\varphi_*, \quad \psi(T,\cdot\,)=\psi_* &\hbox{in} & \mathcal{O}.
		\end{array}
	\right.
\end{equation}
The desired observability inequality will be a consequence of a global Carleman inequality for  \eqref{equa_local_adjunta}, see Proposition \ref{Carlemantheo} below.



\subsection{Carleman estimates}
	
	Before stating the required Carleman inequality, let us introduce several classical weights in the study of Carleman inequalities for parabolic equations, see \cite{Fursikov_Imanuvilov}. The basic weight will be a function $\eta^0 \in C^2(\overline{\Oo})$ verifying
\begin{equation}\label{def_eta_0}
	\begin{array}{ccc}
    		\eta^0>0\quad \hbox{in} \quad\Oo, & \eta^0\equiv 0 \quad\hbox{on}\quad \partial\Oo, & |\nabla\eta^0|>0 \quad\hbox{in} \quad\overline{\Oo\setminus\omega^{'}},    		
	\end{array}
\end{equation}
where $\omega^{'}\subset\subset \omega_c$ is a non-empty open set.\\
\indent Thus, for any $\lambda>0$ we set:
\begin{equation*}
	\begin{array}{cc}
    		\alpha(x,t)=\dfrac{e^{2\lambda\|\eta^0\|_\infty}-e^{\lambda\eta^0(x)}}{t^4(T-t)^4}, & \xi(x,t)=\dfrac{e^{\lambda\eta^0(x)}}{t^4(T-t)^4},\\
		\noalign{\smallskip}\dis
    		\alpha^*(t)=\max_{x\in \overline{\mathcal{O}}}\alpha(x,t), & \dis \widehat{\alpha}(t)= \min_{x\in \overline{\mathcal{O}}}\alpha(x,t),\\
    		\noalign{\smallskip}\dis
    		\xi^*(t)=\min_{x\in \overline{\mathcal{O}}}\xi(x,t), & \dis \widehat{\xi}(t)= \max_{x\in \overline{\mathcal{O}}}\xi(x,t).
	\end{array}
\end{equation*}
We also introduce the following notation:
\begin{equation}\label{def_I}
	\begin{array}{ccc}
    		\dis I(s,\lambda; \varphi)&=&s^3\lambda^4\displaystyle\iint_{\mathcal{O}_T} e^{-2s\alpha}\xi^3|\varphi|^2\,+s\lambda^2\iint_{\mathcal{O}_T} e^{-2s\alpha}\xi|\nabla\varphi|^2\,\\
    		\noalign{\smallskip}\dis
    		&&+s^{-1}\displaystyle\iint_{\mathcal{O}_T} e^{-2s\alpha}\xi^{-1}(|\varphi_t|^2+|\Delta\varphi|^2),
	\end{array}
\end{equation}
where $s$ and $\lambda$ are positive real numbers and $\varphi= \varphi(t,x)$.


We have the following Carleman inequality for \eqref{equa_local_adjunta}:

\begin{prop}\label{Carlemantheo}
 Assume that the assumptions \eqref{hipo_d}, \eqref{hypod1}, \eqref{hipo_A_3}, \eqref{hipo_B_3} are fulfilled. There exist positive constants $\widetilde{\lambda}$, $\widetilde{s}$ and $C=C(\Oo,\omega_c)$ such that, for any $(\varphi_*,\psi_*) \in L_{div}^{2}(\Oo) \times  L^2(\Oo)$, the corresponding
solution to \eqref{equa_local_adjunta} verifies:
\begin{equation}\label{carleman_local}
\dis I(s,\lambda; \varphi)+\dis I(s,\lambda; \psi)\leq C(1+T^2)s^{15/2}\lambda^8\displaystyle\iint_{(0,T)\times\omega_c} e^{-4s\hat{\alpha}+2s\alpha^*}\hat{\xi}^{15/2}(|\varphi|^2+|\psi|^2),
\end{equation}
for all
$\lambda \geq \widetilde{\lambda}$ and $s\geq\widetilde{s}$. Furthermore, $\widetilde{\lambda}$ and $\widetilde{s}$ have the form  $\widetilde{\lambda}=\widetilde{\lambda}_0e^{\widetilde{\lambda}_1T}$ and $\widetilde{s}=\widetilde{s}_0e^{\lambda \widetilde{s}_1}(T^4+T^8)$, where  $\widetilde{\lambda}_0$, $\widetilde{\lambda}_1$ and $\widetilde{s}_0$ only depend on $\|a\|_\infty$, $\|b\|_\infty$, $\|c\|_\infty$,  $\|a_t\|_{L^2(L^r)}$, $\|b_t\|_{L^2(L^r)}$, $\|c_t\|_{L^2(L^r)}$,  $\|A\|_P$ and { $\|B\|_Q$}, and $\widetilde{s}_1$ only depend on $\Oo$ and
$\omega_c$.
\end{prop}
The proof of Proposition \ref{Carlemantheo} follows the arguments of \cite{guerrero_CEL_NS}. For completeness and because of the presence of the equation satisfied by $\psi$, we provide its details in Appendix \ref{sec:proofCarleman}.


\subsection{Null controllability of the linearized system}

	In this section we will deduce the null-controllability of the linear system \eqref{equa_local_linear} as a consequence of the inequality \eqref{carleman_local}. We introduce the following notation for denoting the space where the control is found:
	\begin{equation}
	    \mathcal{H} := H^1(0,T;L^2(\Oo)^n)\cap C^0([0,T];H^1(\Oo)^n)\times  H^1(0,T;L^2(\Oo))\cap C^0([0,T];H^1(\Oo)).
	\end{equation}
	

\begin{prop}\label{result_null_linear}
	Let $(z_*,h_*)\in L^2_{\text{div}}(\mathcal{O})^n\times  L^2(\Oo)$ and let us suppose that \eqref{hipo_d}, \eqref{hypod1}, \eqref{hipo_A_3} and  \eqref{hipo_B_3} holds. Then, there exist controls $(v,w) \in \mathcal{H}$ 
	 such that the corresponding solution to \eqref{equa_local_linear} satisfies
	\begin{equation}\label{cond}
		z(T,\cdot\,)=0\quad\hbox{and}\quad h(T,\cdot\,)=0.
	\end{equation}
	Moreover, the following estimate holds 
	\begin{equation*}
		\|\kappa^{1/2}v\chi_\omega\|+\|\kappa^{1/2}w\chi_\omega\|+\|v\|_{H^1(L^2)}+\|v\|_{L^\infty(H^1)}+\|w\|_{ H^1(L^2)}+\|w\|_{ L^\infty(H^1)}\leq C(\|z_*\|+\|h_*\|),
	\end{equation*}
	where the positive constant C, depending only on $\Oo$, $\omega$, $T$, $\|a\|_\infty$, $\|b\|_\infty$, $\|c\|_\infty$, $\|a_t\|_{L^2(L^r)}$, $\|b_t\|_{L^2(L^r)}$, 
	$\|c_t\|_{L^2(L^r)}$, $\|A\|_P$ and  $\|B\|_Q$, and $\kappa(t)=e^{4s\hat{\alpha}-2s\alpha^*}\hat{\xi}^{-15/2}$ and $s$, $\lambda$, $\hat{\alpha}$, 
$\alpha^*$ and $\hat{\xi}$  are defined at the beginning of the previous section.
\end{prop}
\begin{proof}
	It follows the ideas of \cite{guerrero_CEL_NS}. It is based on a penalized Hilbert Uniqueness Method. Thus, let $(z_*,h_*)\in L^2_{\text{div}}(\mathcal{O})^n\times  L^2(\Oo)$, and  for each $\varepsilon > 0$, let consider the extremal problem
\begin{equation}\label{inf}
\left\{
\begin{array}{llr}
	\dis \hbox{Minimize }   \ \dfrac{1}{2}\iint_{\Oo_T}\kappa(t)\left(|v|^2+|w|^2\right)\chi_\omega\,
	+\dfrac{1}{2\varepsilon}\left(\|z(T, \,\cdot\,)\|^2_{L^2}+\|h(T,\,\cdot\,)\|^2_{{ L^2}}\right) 														\\ 
	\noalign{\smallskip}
	\dis \hbox{Subject to }~v\in L^2(\Oo_T),
	~w\in L^2(\Oo_T)
	~~\text{and}~~(z,h,v,w)~~\text{solves}~~
	\eqref{equa_local_linear}. \\ 
	\noalign{\smallskip}
\end{array}
\right.
\end{equation}

	There exists a (unique) solution to \eqref{inf}, denoted by $(z^\varepsilon, h^\varepsilon, v^\varepsilon, w^\varepsilon)$, with $\kappa(t)^{1/2}(v^\varepsilon,w^\varepsilon)\in L^2(\Oo_T)^n\times L^2(\Oo_T)$, since the functional in \eqref{inf} is coercive, strictly convex and $C^1$ in the Hilbert space $L^2(\Oo_T)^{n+1}$. The associated  Euler-Lagrange equation yields
\begin{equation}\label{I50}
	\iint_{\Oo_T}\kappa(t) \left( v^\varepsilon\cdot v+w^\varepsilon w\right)\chi_\omega\,\,
	+\dfrac{1}{\varepsilon}\int_\Oo \left[ z^\varepsilon(T,\cdot\,)\cdot Z(T,\cdot\,)+h^\varepsilon(T,\cdot\,)H(T,\cdot\,)\right]=0
\end{equation}
 for all $(v,w)\in L^2(\Oo_T)^n\times L^2(\Oo_T)$, where $(Z,H)$ is, together with some $\Pi$, the solution to the  system
\begin{equation}\label{eq_97}
	\left\{
		\begin{array}{lll}
    			\partial_t Z-\Delta Z+((a+b)\cdot\nabla)Z+(Z\cdot\nabla)b+\nabla \Pi = H e_n + v \chi_\omega &\hbox{in} & \mathcal{O}_T,\\
    			\noalign{\smallskip}\dis
    			\partial_t H-\D H+(a+b)\cdot \nabla H+Z\cdot \nabla c=w\chi_\omega &\hbox{in} & \mathcal{O}_T,\\
    			\noalign{\smallskip}\dis
    			\nabla\cdot Z=0 &\hbox{in} &\mathcal{O}_T,\\
    			\noalign{\smallskip}\dis
    			Z\cdot \nu=0, \quad \left[D(Z)\nu+AZ\right]_{tan}=0&\hbox{on} & \varLambda_T,\\
    			\noalign{\smallskip}\dis
    			{\partial H\over \partial \nu}+BH=0&\hbox{on} & \varLambda_T,\\
    			\noalign{\smallskip}\dis
    			Z(0,\cdot\,)=0, \quad H(0,\cdot\,)=0 &\hbox{in} & \mathcal{O}.
		\end{array}
	\right.
\end{equation}

    Let us now introduce the solution $(\varphi^\varepsilon, \pi^\varepsilon, \psi^\varepsilon)$ to the following homogeneous adjoint system:
\begin{equation*}
	\left\{
		\begin{array}{lll}
    			-\partial_t \varphi^\varepsilon-\Delta \varphi^\varepsilon-(a\cdot\nabla)\varphi^\varepsilon-D\varphi^\varepsilon b+\nabla \pi^\varepsilon = c\nabla \psi^\varepsilon 
			&\hbox{in} & \mathcal{O}_T,\\
    			\noalign{\smallskip}\dis
    			-\partial_t \psi^\varepsilon-\D \psi^\varepsilon-(a+b)\cdot \nabla \psi^\varepsilon=\varphi^\varepsilon \cdot e_n &\hbox{in} & \mathcal{O}_T,\\
    			\noalign{\smallskip}\dis
    			\nabla\cdot \varphi^\varepsilon=0 &\hbox{in} &\mathcal{O}_T,\\
    			\noalign{\smallskip}\dis
    			\varphi^\varepsilon\cdot \nu=0, \quad [D(\varphi^\varepsilon)\nu+A\varphi^\varepsilon]_{tan}=0&\hbox{on} & \varLambda_T,\\
    			\noalign{\smallskip}\dis
    			{\partial \psi^\varepsilon\over \partial\nu} +B \psi^\varepsilon =0&\hbox{on} & \varLambda_T,\\
    			\noalign{\smallskip}\dis
    			\varphi^\varepsilon(T,\cdot\,)=-\frac{1}{\varepsilon}z^\varepsilon(T,\cdot\,), \quad \psi^\varepsilon(T,\cdot\,)=-\frac{1}{\varepsilon}h^\varepsilon(T,\cdot\,) 
			&\hbox{in} & \mathcal{O}.
		\end{array}
	\right.
\end{equation*}
The duality between $(\varphi^\varepsilon,\psi^\varepsilon)$ and $(Z,H)$ give
\begin{equation*}
	-\dfrac{1}{\varepsilon}\int_\Oo \left[z^\varepsilon(T,\cdot\,)\cdot Z(T,\cdot\,)+ h^\varepsilon(T,\cdot\,)H(T,\cdot\,)\right]\,=\iint_{\Oo_T} (v\cdot\varphi^\varepsilon +w\,\psi^\varepsilon)\chi_\omega \,
\end{equation*}
which, combined with \eqref{I50}, yields
\begin{equation*}
	\iint_{\Oo_T}(v\,\cdot\, \varphi^\varepsilon +w\,\psi^\varepsilon)\chi_\omega \,
	=\iint_{\Oo_T}\kappa(t) \left( v^\varepsilon \cdot  v+ w^\varepsilon w\right)\chi_\omega\,
\end{equation*}
for all $(v,w)\in L^2(\Oo_T)^n\times L^2(\Oo_T)$. Consequently, we have the following identify
\begin{equation*}
	v^\varepsilon=\kappa^{-1}(t)\varphi^\varepsilon\quad\hbox{and}\quad w^\varepsilon=\kappa^{-1}(t)\psi^\varepsilon\big.
\end{equation*}
The duality between the systems fulfilled by $(z^\varepsilon,h^\varepsilon)$ and $(\varphi^\varepsilon,\psi^\varepsilon)$, gives
\begin{equation*}
	-\dfrac{1}{\varepsilon}\left(\|z^\varepsilon(T,\cdot\,)\|^2+\!\| w^\varepsilon(T,\cdot\,)\|^2\right)\!=\!
	\int_{\Oo}\!\!z_*\cdot \,\varphi^\varepsilon(0,\cdot\,)+h_*\psi^\varepsilon(0,\cdot\,)\,+\iint_{\Oo_T}\kappa^{-1}(t)\left(|\varphi^\varepsilon|^2+ |\psi^\varepsilon|^2\right)\chi_\omega.
\end{equation*}
Moreover, the Carleman inequality \eqref{carleman_local} applied to $(\varphi^\varepsilon,\psi^\varepsilon)$ gives
\begin{equation*}
	\|\varphi^\varepsilon(0,\cdot\,)\|^2+\| \psi^\varepsilon(0,\cdot\,)\|^2\leq 
	C(\Oo,\omega,T,a,b,c,A,B)\iint_{\Oo_T}\kappa^{-1}(t)(|\varphi^\varepsilon|^2+|\psi^\varepsilon|^2 )\chi_\omega.
\end{equation*}
Hence, we conclude that
\begin{equation}\label{eq:vw_epsilon}
\begin{alignedat}{2}
	\frac{1}{\varepsilon}\|z^\varepsilon(T,\cdot\,)\|^2+\frac{1}{\varepsilon}\|w^\varepsilon(T,\cdot\,)\|^2+
	\|\kappa^{1/2}v^\varepsilon\chi_\omega\|^2+\|\kappa^{1/2}w^\varepsilon\chi_\omega\|^2\leq 
	C\left(\|z_*\|^2+\|h_*\|^2\right)\quad \forall \varepsilon>0.
	\end{alignedat}
\end{equation}

	Let us now estimate the norms of $(v^\varepsilon,w^\varepsilon)$ in $\mathcal{H}$. To this purpose, let us introduce the functions $(\widetilde{\varphi}^\varepsilon, \widetilde{\pi}^\varepsilon, \widetilde{\psi}^\varepsilon):=\kappa(t)^{-1}(\varphi^\varepsilon,\pi^\varepsilon,\psi^\varepsilon)$, which satisfy
\begin{equation*}
	\left\{
		\begin{array}{lll}
    			-\partial_t \widetilde{\varphi}^\varepsilon-\Delta \widetilde{\varphi}^\varepsilon-(a\cdot\nabla)\widetilde{\varphi}^\varepsilon-D(\widetilde{\varphi}^\varepsilon) b
			+\nabla \widetilde{\pi}^\varepsilon = c\nabla \widetilde{\psi}^\varepsilon -\partial_t[\kappa(t)^{-1}]\varphi^\varepsilon    &\hbox{in} & \mathcal{O}_T,\\
    			\noalign{\smallskip}\dis
    			-\partial_t \widetilde{\psi}^\varepsilon-\D \widetilde{\psi}^\varepsilon-(a+b)\cdot \nabla \widetilde{\psi}^\varepsilon=\widetilde{\varphi}^\varepsilon \cdot e_n 
			-\partial_t[\kappa(t)^{-1}] \psi^\varepsilon        &\hbox{in} & \mathcal{O}_T,\\
    			\noalign{\smallskip}\dis
    			\nabla\cdot \widetilde{\varphi}^\varepsilon=0 &\hbox{in} &\mathcal{O}_T,\\
    			\noalign{\smallskip}\dis
    			\widetilde{\varphi}^\varepsilon\cdot \nu=0, \quad [D(\widetilde{\varphi}^\varepsilon)\nu+A\widetilde{\varphi}^\varepsilon]_{tan}=0&\hbox{on} & 
			\varLambda_T,\\
    			\noalign{\smallskip}\dis
    			 {\partial\widetilde{\psi}^\varepsilon\over \partial \nu }+ B \widetilde{\psi}^\varepsilon=0&\hbox{on} & \varLambda_T,\\
    			\noalign{\smallskip}\dis
    			\widetilde{\varphi}^\varepsilon(T,\cdot\,)=0, \quad \widetilde{\psi}^\varepsilon(T,\cdot\,)=0 &\hbox{in} & \mathcal{O}.
		\end{array}
	\right.
\end{equation*}
    Then, it is not difficult to deduce
    the following energy estimate
\begin{equation*}
	\|\widetilde{\varphi}^\varepsilon\|_{L^2(H^1)}+\|\widetilde{\psi}^\varepsilon\|_{L^2(H^1)}
	\leq Ce^{CT(\|a\|^2_\infty+\|b\|^2_\infty+\|c\|^2_\infty+\|A\|^2_\infty+\|B\|^2_\infty)}
	\left(\|(\kappa^{-1})_t\varphi^\varepsilon\|+\|(\kappa^{-1})_t\psi^\varepsilon\|\right).
\end{equation*}
    Now, applying strong energy estimates for the Stokes equations with Navier slip boundary conditions, see \cite[Proposition $1.1$]{guerrero_CEL_NS}, and for the heat equation with Robin boundary conditions, see \cite[Proposition $2$]{C-G-B-P-2}, we deduce that
\begin{equation*}
	\|\widetilde{\varphi}^\varepsilon_t\|+\|\widetilde{\psi}^\varepsilon_t\|
	+\|\widetilde{\varphi}^\varepsilon\|_{L^2(H^2)}+\|\widetilde{\psi}^\varepsilon\|_{L^2(H^2)}+\|\widetilde{\varphi}^\varepsilon\|_{C^0(H^1)}+\|\widetilde{\psi}^\varepsilon\|_{C^0(H^1)}
	\leq C \left(\|(\kappa^{-1})_t\varphi^\varepsilon\|+\|(\kappa^{-1})_t\psi^\varepsilon\|\right),
\end{equation*}
where $C=C(\Oo,\omega,T,a,b,c,A,B)$. Combining the previous estimate and the Carleman estimate \eqref{carleman_local}, we get 
\begin{equation*}
	\|v^\varepsilon_t\|+\|w^\varepsilon_t\|+\|v^\varepsilon\|_{L^2(H^2)}+\|w^\varepsilon\|_{L^2(H^2)}
	+\|v^\varepsilon\|_{C^0(H^1)}+\|w^\varepsilon\|_{C^0(H^1)}
	\leq C \left(\|\kappa^{1/2}v^\varepsilon\chi_\omega\|+\|\kappa^{1/2}w^\varepsilon\chi_\omega\|\right).
\end{equation*}

Finally, thanks to \eqref{eq:vw_epsilon}, there exists a control pair
    $(v,w)$ such that 
    $$(v,w)\in H^1(0,T;L^2(\Oo)^{n+1})\cap L^2(0,T;H^2(\Oo)^{n+1})\cap C^0(0,T;H^1(\Oo)^{n+1}),$$
\begin{equation*}
	\|v\|_{H^1(L^2)}+\|w\|_{H^1(L^2)}+\|v\|_{L^2(H^2)}+\|w\|_{L^2(H^2)}+\|v\|_{C^0(H^1)}+\|w\|_{C^0(H^1)}
	\leq C\left(\|z_*\|^2_H+\|h_*\|^2\right).
\end{equation*}
	and the associated solution to \eqref{equa_local_linear}, denoted by $(z,q,h)$, satisfies \eqref{cond} and, moreover,
\begin{equation*}
	\|\kappa^{1/2}v\chi_\omega\|^2+\|\kappa^{1/2}w\chi_\omega\|^2\leq C(\|z_*\|^2+\|h_*\|^2).
\end{equation*}

	This concludes the proof.
\end{proof}
%


\subsection{Local exact controllability to the trajectories of the Boussinesq system}

	This section is devoted to prove the local exact controllability to the trajectories of \eqref{equa_local}.
	Let $(\overline{u},\overline{p},\overline{\te})$ an uncontrolled solution of \eqref{equa_local}, that is, a solution of 
	\begin{equation*}\label{equa_trajecto}
	\left\{
		\begin{array}{lll}
    			\partial_t \overline{u}-\Delta \overline{u}+(\overline{u}\cdot\nabla)\overline{u}+\nabla \overline{p} = \overline{\te} e_n  &\hbox{in} & \mathcal{O}_T,\\
    			\noalign{\smallskip}\dis
    			\partial_t \overline{\te}-\D \overline{\te}+\overline{u}\cdot \nabla \overline{\te}=0 &\hbox{in} & \mathcal{O}_T,\\
    			\noalign{\smallskip}\dis
    			\nabla\cdot \overline{u}=0 &\hbox{in} &\mathcal{O}_T,\\
    			\noalign{\smallskip}\dis
    			\overline{u}\cdot \nu=0, \quad N(\overline{u})+[f(\overline{u})]_{tan}=0&\hbox{on} & \varLambda_T,\\
    			\noalign{\smallskip}\dis
    			R (\overline{\te})+g(\overline{\te})=0&\hbox{on} & \varLambda_T,\\
    			\noalign{\smallskip}\dis
    			\overline{u}(0,\cdot\,)=\overline{u}_*, \quad \overline{\te}(0,\cdot\,)=\overline{\te}_* &\hbox{in} & \mathcal{O}.
		\end{array}
	\right.
\end{equation*}

	We will assume the following regularity for the trajectories:
	\begin{equation}\label{cond_te_H3}
	\begin{array}{llll}
		&\overline{u}\in X:=H^{(3-\ell)/2}(0,T; H^{\vartheta_2+1/2}(\Oo)^n\cap L^2_{div}(\Oo)^n)\cap H^{1-\ell}(0,T;W^{\vartheta_1+1/2,\vartheta_1+1}(\Oo)^n),\\
			 \noalign{\smallskip}\dis
		& \overline{u}_*\in H^3(\Oo)^n\cap L^2_{div}(\Oo)^n,\quad N(\overline{u}_*)+\left[f(\overline{u}_*)\right]_{tan}=0\quad  \text{on} \quad \partial\Oo,\\
		 \noalign{\smallskip}\dis
		 &\overline{\theta}\in Y:=H^{(3-\ell)/2}(0,T; H^{\vartheta_2+1/2}(\Oo))\cap H^{1-\ell}(0,T;W^{\vartheta_1+1/2,\vartheta_1+1}(\Oo)),\\
		 \noalign{\smallskip}\dis
		 & \overline{\theta}_*\in H^3(\Oo), \quad R( \overline{\te}_*)+ g(\overline{\te}_*)=0\quad \text{on}\quad \partial \Oo.
	\end{array}
	\end{equation}

We have the following result:
\begin{prop}\label{local_non_lin}
	Let $f\in C^3(\R^n;\R^n)$, $g\in C^3(\R)$ , $T>0$, $(\overline{u}_*, \overline{\theta}_*)$ satisfying \eqref{cond_te_H3}. Then, there exists  $\delta>0$ such that, for every $(u_*, \theta_*) \in [H^3(\Oo)^n\cap L^2_{div}(\Oo)^n] \times H^3(\Oo)$ satisfying the compatibility condition
	\begin{equation}\label{compati_bound}
	 N({u}_*)+\left[f({u}_*)\right]_{tan}=0,\qquad R({\te}_*)+g({\te}_*)=0\quad \text{on} \quad \partial\Oo
	\end{equation}
	and such that $\|u_* -\overline{u}_*\|_{H^3}\leq \delta$ and $\|\theta_*-\overline{\theta}_*\|_{H^3}\leq \delta$, 
	 there exist controls $(v,w)\in \mathcal{H}$ and associated solutions $(u,p,\theta)$ to \eqref{equa_local} satisfying
	\begin{equation}
 		u(T,\cdot\,)=\overline{u}(T,\cdot\,)\quad\hbox{and}\quad \theta(T,\cdot\,)=\overline{\theta}(T,\cdot\,)\quad\hbox{in} \ \Oo.
	\end{equation}
\end{prop}
\begin{proof}
	Let us denote $z=u-\overline{u}$ and $h=\te-\overline{\te}$. Making the difference between the system fulfilled by  $\overline{u}$ and system \eqref{equa_local}, we have:
	\begin{equation*}
	\left\{
	\begin{array}{lll}
    		\partial_t z-\Delta z+(z\cdot\nabla)z+(z\cdot\nabla)\overline{u}+(\overline{u}\cdot\nabla)z+\nabla q = h e_n + v \chi_{\omega} &\hbox{in} & \mathcal{O}_T,\\
    		\noalign{\smallskip}\dis
    		\partial_t h-\D h+z\cdot \nabla h+\overline{u}\cdot\nabla h+z\cdot \nabla \overline{\te}=w \chi_{\omega} &\hbox{in} & \mathcal{O}_T,\\
    		\noalign{\smallskip}\dis
    		\nabla\cdot  z=0 &\hbox{in} &\mathcal{O}_T,\\
    		\noalign{\smallskip}\dis
    		z\cdot \nu=0, \quad N(z)+[F(\overline{u},z)z]_{tan}=0&\hbox{on} & \varLambda_T,\\
    		\noalign{\smallskip}\dis
    		R( h)+G(\overline{\theta},h)h=0&\hbox{on} & \varLambda_T,\\
    		\noalign{\smallskip}\dis
    		z(0,\cdot\,)=u_*-\overline{u}_*=z_*, \quad h(0,\cdot\,)=\te_*-\overline{\te}_*=h_* &\hbox{in} & \mathcal{O},
  	\end{array}
    	\right.
	\end{equation*}
	where 
	\begin{equation*}
		F_i(\overline{u},z)=\int_0^1\nabla f_i(\overline{u}+s z)\,ds,~~ F(\overline{u},z)=(F_1(\overline{u},z),\ldots,F_n(\overline{u},z))~~\text{and}~~ G(\overline{\theta},h)=\int_0^1 g^\prime(\overline{\theta}+s h)\,ds.
	\end{equation*}
	\indent Now, our goal is to find controls $(v,w)$ such that $z(T,\cdot\,)\equiv 0$ and $h(T,\cdot\,)\equiv 0$. To this purpose, we will use Proposition \ref{result_null_linear}  and  a fixed-point argument.
	
	First, we introduce the following closed linear manifolds
	 \begin{equation*}
	 	X_0=\{\Phi\in X:\,\, \Phi(0,\cdot\, )=z_*\,\text{ in } \,\Oo\}\quad\hbox{and}\quad Y_0=\{\Psi\in Y:\,\, \Psi(0,\cdot\, )=h_*\,\text{ in } \,\Oo\}.
	\end{equation*}
	
	Then, for each $(\Phi,\Psi)\in X_0\times Y_0$, we can apply Proposition \ref{result_null_linear} to guarantee the existence of controls 
	$(v_{(\Phi,\Psi)},w_{(\Phi,\Psi)}) \in \mathcal{H}$ such that the associated solution to 
	\begin{equation}\label{eq:Phi}
	\left\{
	\begin{array}{lll}
    		\partial_t z-\Delta z+(\Phi\cdot\nabla)z+(z\cdot\nabla)\overline{u}+(\overline{u}\cdot\nabla)z+\nabla q = h e_n +  v\chi_\omega &\hbox{in} & \mathcal{O}_T,\\
    		\noalign{\smallskip}\dis
    		\partial_t h-\D h+\Phi\cdot \nabla h+\overline{u}\cdot\nabla h+z\cdot \nabla \overline{\te}= w\chi_\omega &\hbox{in} & \mathcal{O}_T,\\
    		\noalign{\smallskip}\dis
    		\nabla\cdot  z=0 &\hbox{in} &\mathcal{O}_T,\\
    		\noalign{\smallskip}\dis
    		z\cdot \nu=0, \quad N(z)+[F(\overline{u},\Phi)z]_{tan}=0&\hbox{on} & \varLambda_T,\\
    		\noalign{\smallskip}\dis
    		R( h)+G(\overline{\theta}, \Psi)h=0&\hbox{on} & \varLambda_T,\\
    		\noalign{\smallskip}\dis
    		z(0,\cdot\,)=u_*-\overline{u}_*=z_*, \quad h(0,\cdot\,)=\te_*-\overline{\te}_*=h_* &\hbox{in} & \mathcal{O},
  	\end{array}
    	\right.
	\end{equation}
	verifies $z_{(\Phi,\Psi)}(T,\cdot\,)=0$, $h_{(\Phi,\Psi)}(T,\cdot\,)=0$.  Since $F\in C^2(\R^n\times \R^n;\R^{n\times n})$ and
	$G\in C^2(\R^2;\R)$ and $\overline{u}$ and $\overline{\theta}$ verify \eqref{cond_te_H3},  
	we have $F_i(\overline{u}, \Phi)\in X$ for all $\Phi \in X_0$, $i=1,\ldots,n$, and $G(\overline{\theta}, \Psi)\in Y$ for all $\Psi \in Y_0$. Moreover, since $f$ is an irrotational field, we have that $F(\overline{u}, \Phi)$ is symmetric.
	


	Furthermore, these controls can be chosen satisfying
	\begin{equation}\label{est_v_w}
		\|v_{(\Phi,\Psi)}\|_{H^1(L^2)}+\|w_{(\Phi,\Psi)}\|_{H^1(L^2)}
		+\|v_{(\Phi,\Psi)}\|_{L^\infty(H^1)}+\|w_{(\Phi,\Psi)}\|_{L^\infty(H^1)}
		\leq C\left(\|z_*\|^2+\|h_*\|^2\right),	
	\end{equation}
	for some positive constant $C=C(\Oo,\omega,T,\overline{u},\overline{\te},\|\Phi\|_X, \|A\|_P,\|B\|_Q,\|F(\overline{u}, \Phi)\|_X,\|G(\overline{\theta}, \Psi)\|_Y)$.
	
	Next, since we can prove that the terms 
	$(\Phi\cdot\nabla)z_{(\Phi,\Psi)}$,  $(z_{(\Phi,\Psi)}\cdot\nabla)\overline{u}$, $(\overline{u}\cdot\nabla)z_{(\Phi,\Psi)}$,  $h_{(\Phi,\Psi)} e_n$ and  $ v_{(\Phi,\Psi)}\chi_\omega$ belong to $L^\infty(0,T;H^1(\Oo)^n)\cap H^1(0,T;L^2(\Oo)^n)$.
	Thanks to \eqref{cond_te_H3} and \eqref{compati_bound}, we see that $z_*\in H^3(\Oo)^n\cap L^2_{div}(\Oo)^n$ satisfies the compatibility condition $N(z_*)+\left[F(\overline{u},\Phi)(0,\cdot)z_*\right]_{tan}=0$.Hence, we can apply \cite[Proposition 1.2]{guerrero_CEL_NS} to deduce that
	\begin{equation*}
		z_{(\Phi,\Psi)}\in \widetilde{X}:=H^2(0,T;L^2_{div}(\Oo))\cap H^1(0,T;H^2(\Oo)^n\cap L^2_{div}(\Oo)). 
	\end{equation*}
	
	Likewise, one can prove that the terms $\Phi\cdot \nabla h_{(\Phi,\Psi)}$, $\overline{u}\cdot\nabla h_{(\Phi,\Psi)}$, $z_{(\Phi,\Psi)}\cdot \nabla \overline{\te}$ and $ w_{(\Phi,\Psi)}\chi_\omega$ belong to 
    	$L^\infty(0,T;H^1(\Oo))\cap H^1(0,T;L^2(\Oo))$ and, thanks to equations \eqref{cond_te_H3} and \eqref{compati_bound}, $h_*\in H^3(\Oo)$ satisfies the compatibility condition $R(h_*)+G(\overline{\theta}, \Psi)(0,\cdot)h_*=0$. Therefore, we have 
	\begin{equation*}
	h_{(\Phi,\Psi)}\in \widetilde{Y}:=H^2(0,T;L^2(\Oo))\cap H^1(0,T;H^2(\Oo)).
	\end{equation*}
	 Moreover, there exists a positive constant $\widetilde{C}=\widetilde{C}(\Oo,\omega,T,\overline{u},\overline{\te},\|\Phi\|_X,  \|A\|_P,\|B\|_Q,\|F(\overline{u}, \Phi)\|_X,\|G(\overline{\theta}, \Psi)\|_Y) $ such that
	\begin{equation}\label{C_tilde}
		\|(z_{(\Phi,\Psi)},h_{(\Phi,\Psi)})\|_{\widetilde{X}\times \widetilde{Y}}\leq \widetilde{C}\left( \|z_*\|_{H^3\cap W}+\|h_*\|_{H^3}   \right).
	\end{equation}
	 From well-known  interpolation  arguments, one has $\widetilde{X} \subset \subset X$ and $\widetilde{Y} \subset \subset Y$, where the notation $\subset \subset$ stands for compact embedding.
		
		For each $(\Phi,\Psi)\in X\times Y$, let  the set of admissible controls $\Lambda(\Phi,\Psi)$ be, by definition, the family of controls  
	$(v_{(\Phi,\Psi)},w_{(\Phi,\Psi)})\in \mathcal{H}$
	that drive the solution $(z_{(\Phi,\Psi)},h_{(\Phi,\Psi)})$ to zero at time $T$ and such that \eqref{est_v_w} holds.
		On the other hand, let us set 
	\begin{equation*}
		\mathcal{E}(\Phi,\Psi):=\left\{ (z_{(\Phi,\Psi)},h_{(\Phi,\Psi)}) \in\!  \widetilde{X}\times \widetilde{Y} \!:  
		(z_{(\Phi,\Psi)},h_{(\Phi,\Psi)}) \text{ solves \eqref{eq:Phi} with } (v_{(\Phi,\Psi)},w_{(\Phi,\Psi)})\in \!\Lambda(\Phi,\Psi)  \right\}.
	\end{equation*}
	Notice that $\mathcal{E}(\Phi,\Psi)\subset  \widetilde{X}\times \widetilde{Y} \subset\subset X\times Y$. 

		In what follows, we will prove that the set-valued mapping $\mathcal{E}\colon  X_0\times Y_0 \mapsto 2^{X_0\times Y_0}$ possesses at least one  fixed point. We will use the additional hypothesis 
	$\|(z_*,h_*)\|_{H^3\times H^3}\leq \delta$ for some sufficiently small $\delta$ depending on $\Oo$, $\omega$ and $T$ and we will apply Kakutani's Theorem. More precisely, we will check that the mapping $\mathcal{E}$ satisfies the following assumptions:
	\begin{enumerate}
		\item [i)] $\mathcal{E}(\Phi,\Psi)$ is a non-empty closed and convex set of $X_0\times Y_0 $ for all $(\Phi,\Psi) \in X_0\times Y_0 $;
		\item [ii)] There exists a convex compact set $K\subset X_0\times Y_0 $ such that $\mathcal{E}((\Phi,\Psi) )\subset K$ for all $(\Phi,\Psi) \in K$;
		\item [iii)] $\mathcal{E}(\Phi,\Psi)$
		is upper-hemicontinuous in $X_0\times Y_0$, i.e. 
		for any $\varUpsilon \in X_0^\prime\times Y_0^\prime $ the mapping
		\[
		(\Phi,\Psi)\mapsto\sup_{(\overline{\Phi},\overline{\Psi})\in \mathcal{E}(\Phi,\Psi)} \left\langle \,\varUpsilon,(\overline{\Phi},\overline{\Psi})\right\rangle_{X_0^\prime\times Y_0^\prime ,X_0\times Y_0 }
		\]
		is upper semicontinuous.
	\end{enumerate}
	
	Then, in view of Kakutani's Theorem, there exists $(z,h) \in K$ such that $(z,h)\in \mathcal{E}(z,h)$.\\
	
\noindent \textbf{Proof of assumption}  \text{i)} \textbf{of Kakutani's Theorem.} This is easy. Indeed, for every $(\Phi,\Psi) \in X_0\times Y_0$, $\mathcal{E}(\Phi,\Psi)$ is a non-empty set because of the null controllability property of \eqref{eq:Phi}. On the other hand, since  \eqref{eq:Phi} is linear, we readily have that $\mathcal{E}(\Phi,\Psi)$ is closed and convex. \\
	
\noindent \textbf{Proof of assumption}  \text{ii)} \textbf{of Kakutani's Theorem.} Let $R>0$ be given and let us introduce
	\[
	C(R):=\sup_{\|(\Phi,\Psi)\|_{X\times Y}\leq R}\widetilde{C}(\Oo,\omega,T,\overline{u},\overline{\te},\|\Phi\|_X, \|A\|_P,\|B\|_Q,\|F(\overline{u}, \Phi)\|_X,\|G(\overline{\theta}, \Psi)\|_Y),
	\]
	where $\widetilde{C}$ is the constant arising in \eqref{C_tilde}. If we choose $\delta\leq R/C(R)$ and from \eqref{C_tilde} and the fact 
	that $\widetilde{X}\times \widetilde{Y}\subset\subset X\times Y$, we see that $\mathcal{E}$ maps the closed convex set 
	$$\widetilde{K}=\left\{(\Phi,\Psi)\in X_0\times Y_0; \|(\Phi,\Psi)\|_{X\times Y}\leq R\right\}$$ into a compact set $K\subset \widetilde{K}$.\\
	
\noindent\textbf{Proof of assumption} \text{iii)} \textbf{of Kakutani's Theorem.} Let us prove that $\mathcal{E}$ is upper-hemicontinuous. In fact, let $\{(\Phi_k,\Psi_k)\}$ be such that 
	\[
	(\Phi_k,\Psi_k)\to (\Phi,\Psi) \quad\text{in} \quad X{_0}\times Y_0.
	\]
	From the compactness of $\mathcal{E}(\Phi_k,\Psi_k)$ into $X\times Y$, we deduce that
	there exist $(z_k,h_k)\in \mathcal{E}(\Phi_k,\Psi_k)$ for $k=1,2,\cdots$ such that
	\[
	\sup_{(\overline{\Phi},\overline{\Psi})\in \mathcal{E}(\Phi_k,\Psi_k)} \left\langle \,\varUpsilon,(\overline{\Phi},\overline{\Psi})\right\rangle_{X^\prime\times Y^\prime,X\times Y}
	=\left\langle \,\varUpsilon,(z_k,h_k)\right\rangle_{X^\prime\times Y^\prime,X\times Y}\quad \forall \, k\geq1.
	\]
	We can choose a subsequence $\{(\Phi_{k^\prime},\Psi_{k^\prime})\}$ such that 
	\[
	\limsup_{k\to \infty}\sup_{(\overline{\Phi},\overline{\Psi})\in \mathcal{E}(\Phi_{k},\Psi_{k})} \left\langle \,\varUpsilon,(\overline{\Phi},\overline{\Psi})\right\rangle_{X^\prime\times Y^\prime,X\times Y}
	=\lim_{k^\prime\to \infty}
	\left\langle \,\varUpsilon,(z_{k'},h_{k'})\right\rangle_{X^\prime\times Y^\prime,X\times Y}.
	\]
	Denote by $(v_{k^\prime},w_{k^\prime})\in \Lambda(\Phi_{k^\prime},\Psi_{k^\prime})$ controls associated to $(z_{k^\prime},h_{k^\prime})$ solution of following systems:
	\begin{equation*}
	\left\{
	\begin{array}{lll}
    		\partial_t z_{k^\prime}-\Delta z_{k^\prime}+(\Phi_{k^\prime}\cdot\nabla)z_{k^\prime}+(z_{k^\prime}\cdot\nabla)\overline{u}+(\overline{u}\cdot\nabla)z_{k^\prime}+\nabla q_{k^\prime} = h_{k^\prime} e_n + v_{k^\prime}\chi_\omega  &\hbox{in} & \mathcal{O}_T,\\
    		\noalign{\smallskip}\dis
    	    \partial_t h_{k^\prime}-\D h_{k^\prime}+\Phi_{k^\prime}\cdot \nabla h_{k^\prime}+\overline{u}\cdot\nabla h_{k^\prime}+z_{k^\prime}\cdot \nabla \overline{\te}=w_{k^\prime}\chi_\omega  &\hbox{in} & \mathcal{O}_T,\\
    		\noalign{\smallskip}\dis
    		\nabla\cdot  z_{k^\prime}=0 &\hbox{in} &\mathcal{O}_T,\\
    		\noalign{\smallskip}\dis
    		z_{k^\prime}\cdot \nu=0, \quad N(z_{k^\prime})+[F(\overline{u},\Phi_{k^\prime})z_{k^\prime}]_{tan}=0&\hbox{on} & \varLambda_T,\\
    		\noalign{\smallskip}\dis
    		R(h_{k^\prime})+G(\overline{\theta},\Psi_{k^\prime})h_{k^\prime}=0&\hbox{on} & \varLambda_T,\\
    		\noalign{\smallskip}\dis
    		z_{k^\prime}(0,\cdot\,)=z_*, \quad h_{k^\prime}(0,\cdot\,)=h_* &\hbox{in} & \mathcal{O}.
  	\end{array}
    	\right.
	\end{equation*}
	{Then, using the fact the $F(\overline{u},\Phi_{k^\prime})\to F(\overline{u},\Phi)$ in $X$ and $G(\overline{\theta},\Psi_{k^\prime})\to G(\overline{\theta},\Phi)$ in $Y$, we find that the constants in \eqref{est_v_w} and \eqref{C_tilde} can be chosen independent of $k^\prime$. 
	Therefore, the compact embedding $\widetilde{X}\times \widetilde{Y} \subset\subset X\times Y$, together with the estimates \eqref{est_v_w} and \eqref{C_tilde}, guarantees that, at least for a subsequence, 
	we have
	\[
	(z_{k^\prime},h_{k^\prime})\to (z,h) \quad\text{in}\quad X\times Y,
	\]
	\[
	v_{k^\prime}\to v \quad\text{weakly in}\quad H^1(0,T; L^2(\Oo)^n) \cap L^\infty(0,T;H^1(\Oo)^n),
	\]
	\[
	w_{k^\prime}\to  w \quad\text{weakly in}\quad H^1(0,T; L^2(\Oo)) \cap L^\infty(0,T;H^1(\Oo)).
	\]}
	 It is not difficult to conclude that $(v, w)\in \Lambda(\Phi,\Psi)$ and that $ (z,h)\in \mathcal{E}(\Phi,\Psi)$. Therefore, one has:
	\[
	\begin{array}{ll}
	\limsup_{k\to \infty}\sup_{(\overline{\Phi},\overline{\Psi})\in \mathcal{E}(\Phi_{k},\Psi_{k})} \left\langle \,\varUpsilon,(\overline{\Phi},\overline{\Psi})\right\rangle_{X^\prime\times Y^\prime,X\times Y}	\!\!&=\left\langle \,\varUpsilon,(z,h)\right\rangle_{X^\prime\times Y^\prime,X\times Y}\\
	\!\!&\leq \sup_{(\bar z,\bar h)\in \mathcal{E}(\Phi,\Psi)} \left\langle \,\varUpsilon,(\bar z,\bar h)\right\rangle_{X^\prime\times Y^\prime,X\times Y}.
	\end{array}
	\]
	This proves the upper-hemicontinuity of $\mathcal{E}$.\\
	
	\indent Thus, we have proved that $\mathcal{E}$ has a fixed-point $(z,h)$ and this achieves the proof of Proposition~\ref{local_non_lin}.
\end{proof}


	\section{Global controllability to the trajectories}\label{control_trajec}

This section is devoted to  explain how the previous arguments can be chained in order to prove our main result, that is, Theorem \ref{theo_main}.\\
\indent First, we reduce the controllability to weak trajectories to controllability smooth trajectories as follows.\\
\indent Despite that $(\overline{u},\overline{p},\overline{\te})$ is only a weak solution in $[0,T]$, there exists an interval time $[T_1,T_2] \subset (0,T)$ such that $(\overline{u},\overline{p},\overline{\te})$ is smooth in $[T_1,T_2]$. Then, we can start our control strategy by doing nothing in $[0,T_1]$, that is, taking $v=w=\sigma = 0$ in \eqref{def_solu}, and wait for the reference trajectory to be regularized. Thus, the weak trajectory will move from $(u_*,\te_*)$ to some $(u,\te)(T_1,\cdot\,)$, that will be considered the new initial data. 
Hence, without loss of generality, we can 
work with a smooth reference trajectory.\\

\indent We split the control strategy into four steps.\\

\noindent \textbf{Step 1: Regularization of the data.} We begin by extending $\Omega$ to a new domain $\Oo$, as explained in Section \ref{sec:domainextension}. We also use Proposition \ref{prop_extension} to guarantee the existence of $(u_*, \theta_*, \sigma_*) \in L^2(\Oo)^n \times L^2(\Oo) \times C^{\infty}_c(\omega_0)$ satisfying \eqref{eq:extension}. We set $\sigma(t,x) := \beta(t/T) \sigma_*(x)$ with $\beta$ a smooth decreasing function such that $\beta \equiv 1$ near 0 and $\beta\equiv0$ near $1/8$. The function $\sigma$ must satisfy the compatibility condition $\nabla \cdot u_* = \sigma(0,\cdot\,)$. Then, we let system \eqref{def_solu} evolve with $v=w=0$ in the time interval $(0,T/8)$ in order to reach some data $(u,\theta)(T/8, \cdot\,) \in L_{div}^2(\Oo)^n \times L^2(\Oo)$. Next, by using the smoothing effect of the uncontrolled Boussinesq system starting from divergence free data (see Lemma \ref{lemma9}), we deduce that there exists $T_1 \in (0,T/4)$ such that $(u, \theta)(T_1,\cdot\,) \in H^3(\Oo)^{n} \cap L_{div}^{2}(\Oo)^n \times H^3(\Oo)$. Accordingly, we can apply Lemma \ref{lem:constructionu1theta^1}.\\

\noindent \textbf{Step 2: Global approximate controllability result in $L^2(\Oo)$.} Let us set $T_2 := T/2$. Starting from the new initial data $(u,\theta)(T_1,\cdot\,)$, we use the global approximate controllability result stated in Proposition \ref{lemma_trajec_aprox} in a time interval of size $T_2 - T_1 \geq T/4$. Thus, for any $\delta >0$, we can build a trajectory starting from $(u,\theta)(T_1,\cdot\,)$ and such that 
\[\|(u, \theta)(T_2,\cdot\,) - (\overline{u}, \overline{\theta})(T_2,\cdot\,)\| \leq \delta.\] 
In particular, we can find $\delta$ small enough such that \begin{equation}
\label{eq:Psileqdelta}
\Psi_{T/4}(\delta) \leq \delta_{T/4},
\end{equation}
where $\delta_{T/4}$ is the radius of local controllability result given in Proposition \ref{local_non_lin} , for $f= 0$ and $g = 0$, and the function $\Psi_{T/4}$ appears in the regularity result for the free Boussinesq system; see Lemma \ref{lemma9}. \\

\noindent \textbf{Step 3: Regularizing argument.} Now, we use again Lemma \ref{lemma9} to obtain the existence of a time $T_3 \in (T/2, 3T/4)$ 
 such that 
 \[
 \|(u,\theta)(T_3, \cdot\,) - (\overline{u},\overline{\theta})(T_3, \cdot\,)\|_{H^3} \leq \Psi_{T/4}(\delta) \leq \delta_{T/4}. 
 \]

\noindent \textbf{Step 4: Local controllability in $H^3(\Oo)$.} Finally, we use the local controllability result in $[T_3, T_3 + T/4]$,  
and get
$$(u,\theta)(T_3 + T/4, \cdot\,)= (\overline{u}, \overline{\theta})(T_3 + T/4,\cdot\,).$$
Then, extending the control by zero for $t \in [T_3 + T/4]$, we obtain  \eqref{eq:fin_cond} and  the proof is complete.
\\

\section{Comments and open questions
}\label{AAA}

\subsection{Controlling with less controls}

A natural extension of our main result would be the global exact controllability with a reduced number of controls acting on a small part of the boundary. Unfortunately, in this situation, one cannot use the extension domain technique.\\ 
\indent { However, in the spirit of \cite{Cara_NS_BOU6,guerrero_motoya}, one could try to establish a small-time global null controllability for the internal control system \eqref{def_solu} in $2$-D by acting only on the temperature. Roughly speaking, the intuition behind a result of this kind is the following: the temperature $\theta$ is directly controlled by $w$, then $\theta$ acts as an indirect control through the coupling term $\theta e_2$ to control the component $u_2$, then $u_2$ acts also as an indirect control through the incompressibility condition to control the component $u_1$. Results of this kind will be analyzed in a forthcoming paper.\\
\indent One could also try get the local control result acting only on the motion equation, that is, with $w=0$ in \eqref{equa_local}. However, at least in the case of Neumann boundary conditions for $\theta$, that is, with $m\equiv0$ and $g\equiv0$, the system does not seem to be controllable. To justify this assertion, note that, by integrating in $ \mathcal{O}$ the equation satisfied by $\theta$, integrating by parts and using the incompressibility and impermeability conditions, we find that the total mass of $\theta$ is conserved:
\begin{equation*}
	\int_\Oo \te(T,\cdot\,)=\int_\Oo \te_*.
\end{equation*}
Therefore, we cannot expect general null controllability. 


\subsection{Other boundary conditions}

Another natural question is if Theorem \ref{theo_main} holds with $u$ and $\te$ subject to other boundary conditions. 
\\
\indent For Dirichlet boundary conditions on the temperature, this is an interesting open problem. As noticed for the slip case in Remark \ref{remark:dirichlet}, the main difficulty is to obtain good estimates for the remainder terms.\\
\indent When we consider Dirichlet boundary conditions for the velocity, we face a challenging open problem. This is related to a well know conjecture by Jacques-Louis Lions.
  As pointed in \cite{coron_marbach}, the boundary layer has a behavior which is not good as in the case of Navier boundary conditions. This implies many difficulties to estimate the boundary layer profiles and the remainder terms.

\begin{appendices}

\section[Existence result]{Sketch of the proof of Proposition \ref{prop_existence}}
\label{appendix_existence}
In this section, we give a sketch of the proof of Proposition \ref{prop_existence}.

First, since the divergence source term is smooth, we start by solving a Stokes problem in order to lift the non homogeneous 
	divergence condition. To do that, we define $(u_\sigma,p_\sigma)$ as the solution to:
\begin{equation*}
    \left\{
    \begin{array}{lll}
    		\partial_t u_\sigma-\Delta u_\sigma+\nabla p_\sigma = 0 &\hbox{in}& \mathcal{O}_T,\\
    \noalign{\smallskip}\dis
    		\nabla\cdot u_\sigma=\sigma&\hbox{in}& \mathcal{O}_T,\\
    	\noalign{\smallskip}\dis
    		u_\sigma\cdot \nu=0, \quad N(u_\sigma)=0&\hbox{on}&\varLambda_T,\\
    \noalign{\smallskip}\dis
    		u_\sigma(0,\cdot\,)=0&\hbox{in}& \mathcal{O}.
    \end{array}
    \right.
\end{equation*}
Smoothness (in time and space) of $\sigma$ immediately gives smoothness on $u_\sigma$. These are standard maximal regularity estimates for the Stokes problem in the case of the Dirichlet boundary condition. For Navier slip-with-friction boundary conditions, we refer to  \cite{shimada}, \cite{shibata1} and \cite{shibata2}. Then, by using Sobolev embeddings, we get that there exists a positive constant $C>0$ depending on $\sigma$ such that
\begin{equation}
\label{eq_maximalregularityStokes}
\|u_{\sigma}\|_{L^{\infty}(0,T;W^{1,\infty}(\Oo))} \leq C.
\end{equation}
\indent Decomposing $u = u_{\sigma}+ u_h$ and $p = p_{\sigma}+ p_h$, we obtain the following system for $(u_h,p_h, \theta)$:
\begin{equation}\label{u2}
    \!\!\!\!\left\{\!\!\!
    \begin{array}{lll}
    \partial_t u_h\!- \!\Delta  u_h+( u_\sigma\cdot\nabla) u_h \!+( u_h\cdot\nabla) u_{\sigma}\!+(  u_h\cdot\nabla)  u_h \!+\! \nabla p_h
    \!= \!\te e_n + v - (u_\sigma\cdot\nabla)u_\sigma\!\!&\hbox{in}&\! \mathcal{O}_T,\\
    \noalign{\smallskip}\dis
    \partial_t \te-\D \te+(u_h+u_\sigma)\cdot \nabla \te=w &\hbox{in} &\! \mathcal{O}_T,\\
    \noalign{\smallskip}\dis
		\nabla\cdot u_h=0&\hbox{in}&\! \mathcal{O}_T,\\
	 \noalign{\smallskip}\dis
	  u_h\cdot \nu=0, \quad N(u_h)=0&\hbox{on}&\! \varLambda_T,\\
	       \noalign{\smallskip}\dis
	    R(\te)=0&\hbox{on} &\! \varLambda_T,\\
    		\noalign{\smallskip}\dis
	   u_h(0,\cdot\,)=u_{*},\quad \te(0,\cdot\,)=\te_*&\hbox{in}&\! \mathcal{O}.
    \end{array}
\right.
\end{equation}
	So, it is sufficient to obtain the existence result for the system \eqref{u2}.
We define weak solutions to~\eqref{u2} as follows.

Recall that
\begin{equation*}
  W_{T} (\Oo) = [C_{w}^0([0,T];L_{div}^2(\Oo)^n) \cap L^2(0,T;H^1(\Oo)^n)] \times [C_w^0([0,T];L^2(\Oo)) \cap L^2(0,T;H^1(\Oo))].
\end{equation*}
We say that $(u_h, \theta) \in W_{T} (\Oo)$ is a weak solution to \eqref{u2} if it satisfies the following:
\begin{align}
     &- \iint_{\Oo_T} u_h \partial_t \phi + \iint_{\Oo_T} ((u_{\sigma}\cdot \nabla ) u_h + (u_h \cdot \nabla) u_{\sigma} + (u_h \cdot \nabla) u_h) \phi + 2 \iint_{\Oo_T}D(u_h)\cdot D(\phi) \notag\\
     & =  \int_{\Oo} u_{*} \cdot \phi(0,x) - 2 \iint_{\partial\Oo_T} (M u_h) \cdot \phi+  \iint_{\Oo_T} (v -(u_{\sigma} \cdot\nabla ) u_{\sigma}) \phi + \iint_{\Oo_T} \theta e_n \phi,
\end{align}
and
\begin{align}
     &- \iint_{\Oo_T} \theta \partial_t \psi + \iint_{\Oo_T} ((u_{h}\cdot \nabla \theta ) + (u_{\sigma} \cdot \nabla \theta) )\psi + \iint_{\Oo_T} \nabla \theta \cdot \nabla \psi\notag\\
     &\qquad  = \int_{\Oo} \theta_{*} \psi(0,x) -\iint_{\partial\Oo_T}m\theta \psi+ \iint_{\Oo_T} w \psi,
\end{align}
for any which is divergence free and tangent to $\partial\Oo$ function $\phi \in C_{c}^{\infty}([0,T)\times\overline{\Oo})^n$ and any $\psi \in C_{c}^{\infty}([0,T)\times\overline{\Oo})$. We moreover require that they satisfy the so-called {\it strong energy inequality} for almost every $t\in (0,T)$
\begin{equation}\label{eq:energy_inequality}
\begin{array}{l}
\dis    \|u_h(t,\cdot\,)\|^2 + \|\theta(t,\cdot\,)\|^2+ 4 \iint_{\Oo_t} |D(u_h)|^2   + 2 \iint_{ \Oo_t} |\nabla \theta|^2 \leq \|u_h(0,\cdot\,)\|^2 +\|\theta(0,\cdot\,)\|^2\\ 
    \noalign{\smallskip}\dis
   \qquad\qquad   -4 \iint_{\partial\Oo_t} (M u_h) \cdot u_h-2\iint_{\partial\Oo_t}m|\theta|^2\\
   \noalign{\smallskip}\dis
     \qquad\qquad+2\iint_{ \Oo_t} \left[\sigma |u_h|^2 - (u_h \cdot \nabla ) u_{\sigma} \cdot u_h + (v - (u_{\sigma} \cdot \nabla)u_{\sigma} + \theta e_n) \cdot u_{h} + \sigma |\theta|^2 +  w \theta\right].
    \end{array}
    \end{equation}
\indent \textit{Proof of the existence of solutions to \eqref{u2}}. We recall the following identity, 
which will be used throughout the paper: 
\begin{equation}
    \label{eq:identity_useful}
    -\int_{\Oo} \Delta \tilde{u} \cdot \tilde{v} = 2 \int_{\Oo} D(\tilde{u}) \cdot D (\tilde{v}) - 2 \int_{\partial\Oo} [D(\tilde{u})\nu 
  ]_{tan} \cdot \tilde{v},
\end{equation}
where $\tilde{u}$ and $\tilde{v}$ are smooth vector fields such that $\tilde{v}$ is divergence free and tangent to the boundary. Therefore, using above $\phi=u_h$ and $\psi=\theta$, we obtain formally the energy equality
\eqref{eq:energy_inequality} replacing $\leq$ by $=$. We can get a bound of the right hand side term of \eqref{eq:energy_inequality} by using a $L^{\infty}$ bound of $\sigma$ and \eqref{eq_maximalregularityStokes}.
Thus, we deduce that there exists a positive constant $C>0$ depending on $\sigma$, $v$ and $w$ such that 
\begin{align}
\label{eq:energy_inequalityBis}
    &\|u_h(t,\cdot\,)\|^2 + \|\theta(t,\cdot\,)\|^2+ 4 \iint_{\Oo_t} |D(u_h)|^2   + 2 \iint_{\Oo_t} |\nabla \theta|^2\notag\\
    &\leq C \Bigg(  \|u_h(0,\cdot\,)\|^2 + \|\theta(0,\cdot\,)\|^2   +\iint_{\Oo_t} |u_h|^2 + |\theta|^2\Bigg){-  4 \iint_{\partial\Oo_t} (M u_h) \cdot u_h- 2\iint_{\partial\Oo_t}m|\theta|^2}.
    \end{align}
From \eqref{eq:energy_inequalityBis}, and Gronwall Lemma, we obtain an a priori bound for $(u_h,\theta)$ in $L^{\infty}(0,T;L^2(\Oo)^n \times L^2(\Oo))$. Before continuing, let us recall the following Korn inequality.
\begin{lema}\label{lemma:korn_inequality}[Second Korn inequality]
There exist two positive constants $C_1, C_2 >0$ such that, for every $u \in H^1(\Oo)^n$, one has 
\begin{equation}
    \label{eq:korn}
    C_1 \left( \|u\| + \|D(u)\| \right) \leq  \|u\|_{H^1} \leq C_2 \left( \|u\| + \|D(u)\| \right).
\end{equation}
\end{lema}
By using the previous a priori bound for $(u_h,\theta)$ in $L^{\infty}(0,T;L^2(\Oo)^{n+1})$, the estimate \eqref{eq:energy_inequalityBis}
and the second Korn inequality, we also obtain an a priori bound in $L^2(0,T;H^1(\Oo)^{n+1})$. A standard Galerkin procedure implies the existence of a solution with this regularity. 

We next justify that this solution can be assumed to verify the energy inequality. We recall the standard argument to justify the energy inequality. Let $(u_h^N, \theta^N)$ be the approximate solution obtained via the Galerkin method. We write the energy inequality \eqref{eq:energy_inequality} that holds true for $(u_h^N, \theta^N)$ and pass to the limit as $N \rightarrow + \infty$. We observe that the right-hand side converges, because $(u_h^N, \theta^N)$ converges strongly to $(u_h, \theta)$ in $L^2(\Oo_T)$ as $N \rightarrow + \infty$; this is a consequence the two previous bounds and, for instance, Aubin-Lions Lemma. For the left-hand side, it is enough to use convexity, lower semicontinuity of the norms and weak convergence.

\section[Regularity estimates]{Proof of regularity for the uncontrolled Boussinesq system}\label{appen:B}
\label{appendix_regularity}

	Let us present the proof of Lemma \ref{lemma9}. In the following, we will use Korn's inequality recurrently, see Lemma \ref{lemma:korn_inequality}. We will also need the following results:
\begin{lema}\label{lemma:boundary_inequality}
There exist positive constants $C_l, C_r, K >0$ such that, for every $u \in H^1(\Oo)^n$, we have 
\begin{equation}
    \label{eq:boundary_norm}
    C_l \|u\|_{K,M}\leq  \|u\|_{H^1} \leq C_r \|u\|_{K,M},
\end{equation}
    where $\displaystyle\| u\|_{K,M}:=\left(K\|u\|^2+\int_{\partial \Oo}Mu\cdot u + \|D(u)\|^2\right)^{1/2}$.
\end{lema}  
\begin{lema}\label{lemma:boundary_inequality_heat}
There exist positive constants $C_l, C_r, \gamma >0$ such that, for every $\theta \in H^1(\Oo)$, we have 
\begin{equation}
    \label{eq:boundary_norm}
    C_l \|\theta\|_{\gamma,m}\leq  \|\theta\|_{H^1} \leq C_r \|\theta\|_{\gamma,m},
\end{equation}
    where $\displaystyle\| \theta\|_{\gamma,m}:=\left(\gamma\|\theta\|^2+\int_{\partial \Oo}m|\theta|^2 + \|\nabla \theta\|^2\right)^{1/2}$.
\end{lema} 
	The proofs of the two above Lemmas rely on the interpolation inequality \cite[Theorem $III.2.36$]{boyer}. In particular, it is used that 
	there exists a positive constant $C$ such that 
\begin{equation*}
  	\|u\|_{L^2(\partial\mathcal{O})} \leq C \|u\|^{1/2} \|u\|_{H^1}^{1/2} \quad \forall u\in H^1(\mathcal{O}).
\end{equation*}

\begin{lema}[Proposition $III.2.35$, \cite{boyer}]\label{lemma:nonlinear_bound}
    Let $p\in[1,+\infty]$ and $q\in[p,p^*]$, where $p^*$ is the critical exponent associated with $p$. Then, there exists $C>0$
    such that
\[
\|u\|_{L^q}  \leq C\|u\|^{1+n/q-n/p}_{L^p}\|u\|^{n/p-n/q}_{W^{1,p}} \quad \forall u\in W^{1,p}(\Oo).
\]
\end{lema}  
{\begin{lema}
[Pages $490$-$494$, \cite{guerrero_CEL_NS}]\label{lemma:stokes_problem_steady}
  Let $f\in L^2(\Oo)^n$ and $g\in H^{1/2}(\partial\Oo)^n$. Then, there exists a unique strong solution $(u,p)\in H^{2}(\Oo)^n\times H^{1}(\Oo)$
  to the Stokes problem
\[
    \left\{
    \begin{array}{lll}
    - \Delta  u+\! \nabla p=f&\hbox{in}& \mathcal{O},\\
    \noalign{\smallskip}\dis
		\nabla\cdot u=0&\hbox{in}& \mathcal{O},\\
	 \noalign{\smallskip}\dis
	  u\cdot \nu=0, \quad N(u)=g&\hbox{on}& \partial\Oo,
    \end{array}
\right.
\]
and there exists a positive constant $C >0$ such that 
\begin{equation}
    \label{eq:operator_norm}
    \|u\|_{H^{2}}+\|p\|_{H^{1}}\leq C(\|f\|+\|g\|_{H^{1/2}}).
\end{equation}
	Moreover, if $f\in H^k(\Oo)^n$ and $g\in H^{k+1/2}(\partial\Oo)^n$ for some $k\geq0$, then $(u,p)\in H^{k+2}(\Oo)^n\times H^{k+1}(\Oo)$ and
    we have
\[
    \|u\|_{H^{k+2}}+\|p\|_{H^{k+1}}\leq C(\|f\|_{H^k}+\|g\|_{H^{k+1/2}}).
\]

\end{lema}  }
{\begin{lema}\label{lemma:stokes_operator}
    Let $S: D(S)\to L^2_{\text{div}}(\mathcal{O})^n$ be the Stokes operator, where 
    $D(S)=\{v\in H^2(\mathcal{O})^n\cap L^2_{\text{div}}(\mathcal{O})^n: N(v)=0\}$ and $S:=-\mathbb{P}\Delta$.
There exists a positive constant $C >0$ such that, for every $u \in D(S)$, we have 
\begin{equation}
    \label{eq:operator_norm}
     \|u\|_{H^2} \leq C \left(  \|Su\|+\|u\|_{H^{1}} \right).
\end{equation}
	Moreover, if $S u\in H^k(\Oo)^n$ for some $k\geq0$, then $u\in H^{k+2}(\Oo)^n$ and
    we have
\[
    \|u\|_{H^{k+2}}\leq C(\|S u\|_{H^k}+\|u\|_{H^{k+1}}).
\]

\end{lema}  }

%
%
%
%
\begin{lema}\label{lemma:robin_regularity}
    Let $u\in H^1(\mathcal{O})$ satisfy $\Delta u\in L^2(\mathcal{O})$ and
\[
         \frac{\partial u}{\partial \nu}+m u=0 \quad\text{on}\quad \partial\mathcal{O},
\]
	where $m\in L^\infty(\partial\Oo)$. Then, there exists a constant $C>0$, only depending on $\mathcal{O}$,  such that
\[
    \|u\|_{H^{2}}\leq C(\|\Delta u\|+\|mu\|_{H^{1/2}(\partial\Oo)}).
\]
    Moreover, if $\Delta u\in H^k(\Oo)$ for some $k\geq0$, then $u\in H^{k+2}(\Oo)$ and
    we have
\[
    \|u\|_{H^{k+2}}\leq C(\|\Delta u\|_{H^k}+\|mu\|_{H^{k+1/2}(\partial\Oo)}).
\]

\end{lema}  
The proof this Lemma is consequence of \cite[Theorem $III.4.3$]{boyer}. 
%
    
    Throughout the proof of Lemma \ref{lemma9}, the constants $C$ can increase from line to line and depend on $T$ and the trajectory $(\overline{u}, \overline{\theta})$. For simplicity, we consider the case $n=3$.\\
    
\textbf{Step 1: Weak estimates in $(0,T/3)$.} Let us first multiply \eqref{equation_lemma9}$_1$ by $r$ and \eqref{equation_lemma9}$_2$ by $q$, integrate by parts, and sum. We get:
\begin{equation*}
\begin{array}{l}
 \dis   \dfrac{1}{2}\dfrac{d}{dt}\left(\| r\|^2+ \| q\|^2 \right) + 2 \|D r\|^2 +\|\nabla q\|^2+ 2\int_{\partial \Oo}Mr\cdot r+\int_{\partial \Oo}m |q|^2 \\
 \noalign{\smallskip}\dis
 \qquad \qquad \qquad= (qe_n,r)-\int_{\mathcal{O}}(r\cdot\nabla)\overline{u}\cdot r -\int_\Oo r\cdot\nabla \overline{\te} \, q.
\end{array}
\end{equation*}

	By Cauchy-Schwarz and Young inequalities, we obtain:
	\begin{equation*}
    \dfrac{1}{2}\dfrac{d}{dt}\left(\| r\|^2+ \| q\|^2 \right) + 2\|D r\|^2 +\|\nabla q\|^2+2\int_{\partial \Oo}Mr \cdot r+\int_{\partial \Oo}m |q|^2 \leq  C(\|r\|^2+\|q\|^2).
\end{equation*}

	Using Lemmas \ref{lemma:boundary_inequality} and \ref{lemma:boundary_inequality_heat}, we deduce
	\begin{equation}\label{eq:energy1}
    \dfrac{1}{2}\dfrac{d}{dt}\left(\| r\|^2+ \| q\|^2 \right) + {2\over C_l^2}(\| r\|^2_{H^1} +\| q\|^2_{H^1}) \leq  (C+2K)\|r\|^2+(C+2\gamma)\|q\|^2.
\end{equation}
By applying Gronwall Lemma, we have for a.e. $t\in [0,T]$ that
\begin{equation}\label{est_1_lemma9}
	\|r(t,\,\cdot\,)\|^2+\|q(t,\,\cdot\,)\|^2+\int_0^t\left(\|r(s,\,\cdot\,)\|^2_{H^1}+\|q(s,\,\cdot\,)\|^2_{H^1}\right)\,ds\leq e^{Ct}\left(\|r_*\|^2+\|q_*\|^2\right).
\end{equation}
	Therefore, from the Mean Value Theorem, we deduce by contradiction that there exists $0\leq t_1\leq T/3$ such that
	\begin{equation}\label{est_t_1}
		\|r(t_1,\,\cdot\,)\|_{H^1}^2+\|q(t_1,\,\cdot\,)\|_{H^1}^2\leq C_1\left(\|r_*\|^2+\|q_*\|^2\right),
	\end{equation}
	for a positive constant $C_1$ independent of $t_1$.

\

\textbf{Step 2: Strong estimates in $(t_1, 2T/3)$.} Let $\mathbb{P}$ be the classical Leray projector. We multiply \eqref{equation_lemma9}$_1$ and \eqref{equation_lemma9}$_2$ by $- S r$ and $-\Delta q$, respectively, then integrate by parts.  Since $M$ is symmetric, we obtain
	\begin{equation}\label{eq:rregular}
		\begin{array}{ll}
		&\dis\dfrac{d}{dt}\left(\|D r\|^2+\int_{\partial \Oo}Mr\cdot r\right)+\|S r\|^2  \\
		\noalign{\smallskip}\dis
		&\qquad\dis= \int_{\partial \Oo} (M_t) r \cdot r+\int_{\mathcal{O}} \Bigg((r\cdot\nabla)r\cdot S r+ (\overline{u}\cdot\nabla)r\cdot S r+(r\cdot\nabla)\overline{u}\cdot Sr -(qe_n,S r)\Bigg)\\
		\noalign{\smallskip}\dis
		&\qquad\leq C\|r\|_{H^1}^2+{1\over2}\|Sr\|^2+ C\|q\|^2+\|r\|_{L^6}^2\|\nabla r\|^2_{L^3}.
		\end{array}
	\end{equation}
Also,
	\begin{equation}\label{eq:qregular}
		\begin{array}{lll}
\noalign{\smallskip}\dis
		\dfrac{1}{2}\dfrac{d}{dt}\left(\|\nabla q\|^2+\int_{\partial \Oo}\!\!m|q|^2\right)+\|\Delta q\|^2&=&
		\dis{1\over 2}\!\int_{\partial \Oo} (m_t) q \cdot q+(r\cdot \nabla q,\Delta q)\\
		\noalign{\smallskip}\dis
		&&+(\overline{u}\cdot\nabla q,\Delta q)+(r\cdot\nabla \overline{\te},\Delta q)\\
		\noalign{\smallskip}\dis
		&\leq&\!\!\!\! C\|q\|_{H^1}^2+{1\over 2}\|\Delta q\|^2+C\|r\|^2+\|r\|_{L^6}^2\| \nabla q\|_{L^3}^2.
		\end{array}
	\end{equation}

	Multiplying \eqref{eq:energy1} by $\varsigma=\max\{K,\gamma\}$, adding the above inequalities and using
	Lemmas \ref{lemma:boundary_inequality} -- \ref{lemma:robin_regularity}, we deduce the following:
\begin{equation}\label{est2_lemma99}
		\begin{array}{lll}
	\!\!\!\!\!\!\!\!\dfrac{d}{dt}\left(\|r\|^2_{\varsigma,M}+\!\|q\|^2_{\varsigma,m}\right )\!+\| r\|^2_{H^2}\!+\!\|q\|^2_{H^2}\!\!\!\!\!&\leq&\!\!\!\!\!  C(\|r\|^2_{\varsigma,M}\!+\!\|q\|^2_{\varsigma,m}
	\!+\!\|r\|_{L^6}^2\|\nabla r\|_{L^3}^2\!+\!\|r\|_{L^6}^2\| \nabla q\|_{L^3}^2)\\
	\noalign{\smallskip}\dis
	&\leq&\!\!\!\!\!  C\left[(\|r\|^2_{\varsigma,M}\!+\!\|q\|^2_{\varsigma,m})+(\|r\|^2_{\varsigma,M}\!+\!\|q\|^2_{\varsigma,m})^3\right].
		\end{array}
\end{equation}

	Introducing $Y(t):=\|r(t,\cdot)\|^2_{\varsigma,M}\!+\!\|q(t,\cdot)\|^2_{\varsigma,m}$, we see that $Y$ is a.e. differentiable and,  from \eqref{est2_lemma99}, we have that
\begin{equation}\label{est_edo_1}
	Y^\prime\leq C (Y^3+Y).
\end{equation}
In view of \eqref{est_edo_1},
we obtain
	$$ Y(t)^2\leq {e^{C(t-t_1)}Y(t_1)^2\over Y(t_1)^2+1-e^{C(t-t_1)}Y(t_1)^2}.$$
 Let us take $t-t_1\leq \tau_1$ small enough and such that $e^{C(t-t_1)}\leq 1+\frac{1}{2Y(t_1)^2}$. Then, $Y(t)^2\leq 2e^{C(t-t_1)}Y(t_1)^2$ and, from \eqref{est_t_1}, we deduce that $Y(t)
 \leq CY_*$, where $Y_*:=\|r_*\|^2+\|q_*\|^2$. Therefore,
$$\|r(t,\cdot\,)\|^2_{\varsigma,M}\!+\!\|q(t,\cdot\,)\|^2_{\varsigma,m}+\int_{t_1}^t(\|r(s,\cdot\,)\|^2_{H^2}+\|q(s,\cdot\,)\|^2_{H^2})ds\leq CY_*+C(Y_*+Y_*^3)\tau_1.$$
Taking $\tau_1$ small enough such that $\tau_1\leq (1+Y_*^2)^{-1}$, 
we have that  $CY_*+C(Y_*+Y_*^3)\tau_1\leq C_2Y_*$.  Therefore, one has
\begin{equation}\label{est_2_lemma9}
\|r(t,\cdot\,)\|^2_{\varsigma,M}\!+\!\|q(t,\cdot\,)\|^2_{\varsigma,m}+\int_{t_1}^t(\|r(s,\cdot\,)\|^2_{H^2}+\|q(s,\cdot\,)\|^2_{H^2})ds\leq  C_2\left(\|r_*\|^2+\|q_*\|^2\right)
\end{equation}
for $t_1\leq t\leq t_1+\tau_1$. This ensures the existence of $t_1\leq t_2< \min\{2T/3,t_1+\tau_1\}$ such that
$$\|r(t_2,\,\cdot\,)\|^2_{H^2}+\|q(t_2,\,\cdot\,)\|^2_{H^2}\leq \dfrac{C_2}{\tau_1}\left(\|r_*\|^2+\|q_*\|^2\right).$$

\textbf{Step 3: Third energy estimate in $(t_2, T)$.}
	At this point, we differentiate \eqref{equation_lemma9} with respect to time and multiply  by $\partial_t r$ and $\partial_t q$. Then, we integrate by parts to obtain
	\begin{equation*}
	\begin{array}{ll}
	    &\dis\dfrac{1}{2}\dfrac{d}{dt}\| r_t\|^2+2\|D r_t\|^2+2\int_{\partial \Oo}Mr_t\cdot r_t  \\
	    \noalign{\smallskip}\dis
		&\dis=-2\int_{\partial \Oo}M_tr\cdot r_t+ (q_te_n,r_t)-(r_t\cdot\nabla)r\cdot r_t-(\overline{u}_t\cdot\nabla)r\cdot r_t-(r_t\cdot\nabla)\overline{u}\cdot r_t-(r\cdot\nabla)\overline{u}_t\cdot r_t\\
		\noalign{\smallskip}\dis 
		&\dis\leq C\left(\|r\|_{H^1}\| r_t\|_{H^1}+ \|q_t\|^2+ \|r_t\|^2+\|r_t\|_3\|\nabla r\|\| r_t\|_6+\|r\|^2_{H^1}\right)
		\end{array}
	\end{equation*}
and
\begin{equation*} 
	\begin{array}{ll}\dis
		\dfrac{1}{2}\dfrac{d}{dt}\| q_t\|^2+\|\nabla q_t\|^2+\int_{\partial \Oo}\!\!\!m|q_t|^2\!\!\!\!&\dis=-\int_{\partial \Oo}\!\!\!m_tqq_t-((r_t+\overline{u})\cdot\nabla q, q_t)-(r_t\cdot\nabla \overline{\te}, q_t)-(r\cdot\nabla\overline{\te}_t, q_t)\\
		\noalign{\smallskip}\dis
		&\leq C\left(\|q\|_{H^1}\|q_t\|_{H^1}+\| q\|^2_{H^1}+\|q_t\|^2+\|r_t\|^2+\|r_t\|_3\|\nabla q\|\|r_t\|_6     \right).       
	\end{array}
\end{equation*}
Consequently, using Lemmas \ref{lemma:boundary_inequality} -- \ref{lemma:nonlinear_bound}  and adding the two above inequalities, we have
\begin{equation*}
	\begin{array}{l}\dis
		\dfrac{d}{dt}\left(\| r_t\|^2+\| q_t\|^2\right)+\|r_t\|^2_{H^1}+\|q_t\|^2_{H^1}\\
		\noalign{\smallskip}\dis
		\qquad\qquad\leq  C\left(\left(\|r\|^4_{H^1}+ \|q\|_{H^1}^4+1\right)\|r_t\|^2+\|q_t\|^2+\| r\|_{H^1}^2+\| q\|_{H^1}^2\right).
	\end{array}
\end{equation*}

Now, introducing $Z(t):=\| r_t(t,\,\cdot)\|^2+\| q_t(t,\,\cdot)\|^2$,
	  we find from \eqref{est_2_lemma9} that
\begin{equation}\label{est_edo_2}
	Z^\prime\leq C [(1+Y_*^2)Z+Y_*]
\end{equation}
for $t_2\leq t\leq t_1+\tau_1$. 
By applying Gronwall's Lemma, we have for a.e. $t\in [t_2,t_1+\tau_1]$
    \[
    Z(t)\leq e^{C(1+Y_*^2)(t-t_2)}\left(Z(t_2)+CY_*(t-t_2)\right).
    \]
%
 Since we have $Z(t_2)\leq \Psi_1(Y_*)$ 
 for some nonnegative regular $\Psi_1$ with $\Psi_1(0)=0$, 
we find that $Z(t)\leq \Psi_2(Y_*)$, with
 \[
 	\Psi_2(s):=e^{C(1+s^2)}(\Psi_1(s)+Cs)\quad \forall s\geq 0.
\]
 
 Therefore, 
\begin{equation}\label{est_3_lemma9}
	\| r_t(t,\,\cdot\,)\|^2+\| q_t(t\,\cdot\,)\|^2+\int_{t_2}^t\left( \|r_t(s,\,\cdot\,)\|^2_{H^1}+ \|q_t(s,\,\cdot\,)\|^2_{H^1} \right) \, ds\leq \Psi_3(Y_*)\quad \forall t\in [t_2,t_1+\tau_1],
\end{equation}
where 
$\Psi_3(s):=C [(1+s^2)\Psi_2(s)+s]$. In particular, this yields the existence of $t_3\in (t_2,t_1+\tau_1)$ such that
\begin{equation}\label{est_5_lemma9} 
	\|r_t(t_3,\,\cdot\,)\|^2_{H^1}+ \|q_t(t_3,\,\cdot\,)\|^2_{H^1}\leq {\Psi_3(Y_*)\over (t_1-t_2+\tau_1)}.
\end{equation}
Actually, it is not difficult to check that the set of times $t_3\in(t_2,t_1+\tau_1)$ satisfying \eqref{est_5_lemma9} has a positive measure.

%

\

\textbf{Step 4: Conclusion.}	Using \eqref{est_2_lemma9} and \eqref{est_3_lemma9}, we deduce an estimate of $r$ in $L^\infty(H^2)$. It suffices to view \eqref{equation_lemma9}$_1$ as a family of Stokes problems (see Lemma \ref{lemma:stokes_problem_steady} and the arguments presented in \cite[Theorem 3.8]{teman}). Then, looking \eqref{equation_lemma9}$_2$ as a family of elliptic problems, we also find $L^\infty(H^2)$ estimates for $q$, see Lemma \ref{lemma:robin_regularity}. Both estimates depend on $Y_*$ continuously. Therefore, repeating the procedure, we see
that $\left(r(t_3),q(t_3)\right)\in H^3\times H^3$ with an estimate of the form $\Psi(Y_*)$.



\section{Proof of the global Carleman estimate}
\label{sec:proofCarleman}
This section is dedicated to the proof of Proposition \ref{Carlemantheo}.\\

The proof is divided into eight steps and is inspired by the ideas of \cite{guerrero_CEL_NS}. In the following, the positive constants $C$ vary from line to line and  depend only on $\Oo$ and $\omega$.

Let the non-empty  open sets $\omega^{'}$ and $\omega_c$ be given, with $\omega^{'}\subset\subset \omega_0\subset\subset \omega_c$.\\





\noindent \textbf{Step 1: Global Carleman estimates for $\phi$ and $\psi$ and absorption of global terms.}\\
\indent We apply the Carleman estimate \cite[Proposition $2.1$]{guerrero_CEL_NS} for the heat system $\eqref{equa_local_adjunta}_1$ with source term $G:=c\nabla\psi-\nabla\pi+D\varphi b+(a\cdot\nabla)\varphi$, to get 
\begin{align}
    	 	I(s,\lambda; \varphi)&\leq C\bigg( s^3\lambda^4\displaystyle\iint_{(0,T)\times\omega_0} e^{-2s\alpha}\xi^3|\varphi|^2\,+\lambda(\|a\|^2_\infty+\|b\|^2_\infty)\displaystyle\iint_{\Oo_T} e^{-2s\alpha}|\nabla\varphi|^2\,\notag\\
    		&\qquad+\lambda\iint_{\Oo_T} e^{-2s\alpha}|\nabla\pi|^2\,+\lambda\|c\|^2_\infty\iint_{\Oo_T} e^{-2s\alpha}|\nabla\psi|^2\,\bigg)\label{I1},
\end{align}
for $\lambda\geq \widehat{\lambda}e^{\widehat{\lambda}T\|A\|^2_{P}}(1+\|A\|^5_{P})$ and $s\geq \widehat{s}e^{4\lambda\|\eta^0\|_\infty}(T^6+T^8)$ and $\widehat{\lambda}$, $\widehat{s}$ only depend on $\Oo$ and $\omega$. Thanks to the definition of $\xi$, we have $1\leq CT^8\xi\leq Cs\xi$ and we can eliminate the second term in the right-hand side of \eqref{I1} with the term in $s\lambda^2$ that appears in the expression of $I (s,\lambda; \varphi)$. Indeed,  if we take $\lambda\geq \widehat{\lambda}(\|a\|^2_\infty+\|b\|^2_\infty)$, we get 

\begin{equation}\label{I2}
    \begin{array}{ccc}
    		\dis I(s,\lambda; \varphi)&\leq & C\biggl( s^3\lambda^4\displaystyle\iint_{(0,T)\times\omega_0} e^{-2s\alpha}\xi^3|\varphi|^2\,+\lambda\iint_{\Oo_T} e^{-2s\alpha}|\nabla\pi|^2\,\\
    		\noalign{\smallskip}\dis
    		&&+\displaystyle\lambda\|c\|^2_\infty\iint_{\Oo_T} e^{-2s\alpha}|\nabla\psi|^2\, \biggr),
    \end{array}
 \end{equation}
 for any $\lambda\geq \widehat{\lambda}e^{\widehat{\lambda}T\|A\|^2_{P}}(1+\|a\|^2_\infty+\|b\|^2_\infty+\|A\|^5_{P})$ and any $s\geq \widehat{s}e^{4\lambda\|\eta^0\|_\infty}( T^6+T^8)$.

	 Next, we apply the known Carleman estimates 
	for the heat equation with homogeneous Robin boundary condition fulfilled by $\psi$, which gives
\begin{equation*}
    \begin{array}{ccc}
    		\dis I(s,\lambda; \psi)&\leq & C\bigg( s^3\lambda^4\displaystyle\iint_{(0,T)\times\omega_0} \!\!\!\!\!\!e^{-2s\alpha}\xi^3|\psi|^2\,+(\|a\|^2_\infty+\|b\|^2_\infty)\displaystyle\iint_{\Oo_T} e^{-2s\alpha}|\nabla\psi|^2\,\\
  		\noalign{\smallskip}\dis
    		&&+\displaystyle\iint_{\Oo_T} e^{-2s\alpha}|\varphi\cdot e_n|^2\,\bigg),
    \end{array}
 \end{equation*}
for $\lambda\geq \widehat{\lambda}e^{\widehat{\lambda}T(\|A\|^2_{P}+\|B\|^2_{Q})}(1+\|a\|^2_\infty+\|b\|^2_\infty+\|B\|^5_Q+\|A\|^5_{P})$ and $s\geq \widehat{s}e^{4\lambda\|\eta^0\|_\infty}(T^4+T^8)$. The same argument above yields
\begin{equation}\label{I3}
   \begin{array}{ccc}
    		\dis I(s,\lambda; \psi)&\leq & C\left( s^3\lambda^4\displaystyle\iint_{(0,T)\times\omega_0} \!\!\!\!\!\!\!\!e^{-2s\alpha}\xi^3|\psi|^2\,+\displaystyle\iint_{\Oo_T} e^{-2s\alpha}|\varphi\cdot e_n|^2\,\right),
    \end{array}
\end{equation}
for any $\lambda\geq \widehat{\lambda}e^{\widehat{\lambda}T(\|A\|^2_{P}+\|B\|^2_{Q})}(1+\|a\|^2_\infty+\|b\|^2_\infty+\|B\|^5_Q+\|A\|^5_{P})$ and any $s\geq \widehat{s}e^{4\lambda\|\eta^0\|_\infty}(T^4+T^8)$.\\
\indent From  \eqref{I2} and \eqref{I3}, we get 
\begin{equation}\label{I4}
 \begin{array}{ccc}
    		\dis I(s,\lambda; \psi)+\dis I(s,\lambda; \varphi)\!\!\!\!&\leq &\!\!\!\!C\bigg(s^3\lambda^4\displaystyle\iint_{(0,T)\times\omega_0} \!\!\!\!\!\!\!\!\!e^{-2s\alpha}\xi^3|\varphi|^2+s^3\lambda^4\displaystyle\iint_{(0,T)\times\omega_0} \!\!\!\!\!\!\!\!\!\!e^{-2s\alpha}\xi^3|\psi|^2\\
    		\noalign{\smallskip}\dis
    		&&\!\!\!\!\!\!\!+\displaystyle\lambda\iint_{\Oo_T} e^{-2s\alpha}|\nabla\pi|^2\,+\displaystyle\lambda\|c\|^2_\infty\iint_{\Oo_T} e^{-2s\alpha}|\nabla\psi|^2\,\\
    		\noalign{\smallskip}\dis
    		&&+\displaystyle\iint_{\Oo_T}  
    		 e^{-2s\alpha}|\varphi\cdot 
    		e_n|^2\,\bigg),
    \end{array}
\end{equation}	
for any $\lambda\geq \widehat{\lambda}e^{\widehat{\lambda}T(\|A\|^2_{P}+\|B\|^2_{Q})}(1+\|a\|^2_\infty+\|b\|^2_\infty+\|B\|^5_Q+\|A\|^5_{P})$ and any $s\geq \widehat{s}e^{4\lambda\|\eta^0\|_\infty}(T^4+T^8)$. Using the parameters $s^3\lambda^4$, $s\lambda^2$ appearing in $I (s,\lambda; \varphi)$ and $I (s,\lambda;\psi)$ we can absorb the lower order terms on the right-hand side of \eqref{I4}. This way, we have
\begin{equation}\label{I5}
    \begin{array}{ccc}
    		\dis I(s,\lambda; \psi)+\dis I(s,\lambda; \varphi)\!\!\!\!&\leq &\!\!\!\!C\bigg(s^3\lambda^4\displaystyle\iint_{(0,T)\times\omega_0} \!\!\!\!\!\!\!\!\!\!e^{-2s\alpha}\xi^3|\varphi|^2+s^3\lambda^4\displaystyle\iint_{(0,T)\times\omega_0} \!\!\!\!\!\!\!\!\!\!e^{-2s\alpha}\xi^3|\psi|^2 \\
    		\noalign{\smallskip}\dis
   	&&+\displaystyle\lambda\iint_{\Oo_T} e^{-2s\alpha}|\nabla\pi|^2\,\bigg),
\end{array}
\end{equation}   		
 for every $\lambda\geq \widehat{\lambda}e^{\widehat{\lambda}T(\|A\|^2_{P}+\|B\|^2_{Q})}(1+\|a\|^2_\infty+\|b\|^2_\infty+\|c\|^2_\infty+\|B\|^5_Q+\|A\|^5_{P})$ and any $s\geq \widehat{s}e^{4\lambda\|\eta^0\|_\infty}(T^4+T^8)$.\\

\noindent \textbf{Step 2: Localization of the pressure term by a global elliptic Carleman estimate.}\\
\indent We estimate the integral on the pressure term in \eqref{I5}. 
To do that, let us take the divergence operator in the equation verified by $\varphi$, thus
\begin{equation}\label{Ideltapi}
	\Delta\pi(t,\cdot\,)=\nabla\cdot((a\cdot\nabla)\varphi+D\varphi b+c\nabla \psi)\quad \hbox{in}\quad\Oo\quad\hbox{a.e.} \quad t\in(0,T).
\end{equation}
Now, since the  right-hand side of \eqref{Ideltapi} is a $H^{-1}$ term,  we can  apply the elliptic Carleman inequality given in \cite[Theorem 0.1]{Imanuvilov_Puel}. Hence, there exist two positive constants $\tilde{\tau} \geq 1 $ and $\tilde{\lambda} \geq 1$, such that
\begin{equation*}\label{I6}
	\begin{array}{ll}
		\displaystyle\int_\Oo e^{2\tau\eta}|\nabla\pi(t,\cdot\,)|^2+&\!\!\!\!\!\dis\tau^2\lambda^2\int_\Oo e^{2\tau\eta}\eta^2|\pi(t,\cdot\,)|^2\\ 
				\noalign{\smallskip}\dis
		& \leq  \displaystyle C\bigg(\tau^\frac{1}{2}e^{2\tau}\|\pi(t,\cdot\,)\|^2_{H^\frac{1}{2}(\partial\Omega)}+\tau\int_\Oo e^{2\tau\eta}\eta|(a\cdot\nabla)\varphi+D\varphi b+c\nabla \psi|^2\\
		\noalign{\smallskip}\dis		
		& \ \ \ \ +\displaystyle\tau^2\lambda^2\int_{\omega^{'}}e^{2\tau\eta}\eta^2|\pi(t,\cdot\,)|^2+ \int_{\omega^{'}} e^{2\tau\eta}|\nabla\pi(t,\cdot\,)|^2\bigg)\\
		\noalign{\smallskip}\dis
		&\leq\displaystyle C\bigg(\tau(\|a\|^2_\infty+\|b\|^2_\infty)\int_\Oo e^{2\tau\eta}\eta|\nabla\varphi(t,\cdot\,)|^2+\tau\|c\|^2_\infty\int_\Oo e^{2\tau\eta}\eta|\nabla\psi(t,\cdot\,)|^2\\
		\noalign{\smallskip}\dis
		&\ \ \ \ +\tau^\frac{1}{2}e^{2\tau}\|\pi(t,\cdot\,)\|^2_{H^\frac{1}{2}(\partial\Oo)}+\displaystyle\tau^2\lambda^2\int_{\omega^{'}}e^{2\tau\eta}\eta^2|\pi(t,\cdot\,)|^2+ \int_{\omega^{'}} e^{2\tau\eta}|\nabla\pi(t,\cdot\,)|^2\bigg)		
	\end{array}
\end{equation*}
for $\tau\geq \widehat{\tau}$ and $\lambda\geq \widehat{\lambda}e^{\widehat{\lambda}T(\|A\|^2_{P}+\|B\|^2_{Q})}(1+\|a\|^2_\infty+\|b\|^2_\infty+\|c\|^2_\infty+\|B\|^5_Q+\|A\|^5_{P})$. Here, for each $\lambda>0$, the function $\eta$ is given by $\eta(x)=e^{\lambda\eta^0(x)}$
where the function $\eta^0$ is defined in \eqref{def_eta_0}. Let us now set $\tau=s/(t^4(T-t)^4)$. We multiply the previous inequality by $\exp(-2se^{2\lambda\|\eta^0\|_\infty}/(t^4(T-t)^4))$ and integrate between $t = 0$ and $t = T$. It is not difficult to see that 
\begin{align}
		&\iint_{\Oo_T} e^{-2s\alpha}|\nabla\pi|^2\,+s^2\lambda^2\iint_{\Oo_T} e^{-2s\alpha}\xi^2|\pi|^2\,\notag \\
		 &\qquad\qquad\leq  \displaystyle C\bigg(s(\|a\|^2_\infty+\|b\|^2_\infty)\iint_{\Oo_T} e^{-2s\alpha}\xi|\nabla\varphi|^2\,+s\|c\|^2_\infty\iint_{\Oo_T} e^{-2s\alpha}\xi|\nabla\psi|^2\,\notag\\
		\noalign{\smallskip}\dis
		&\qquad\qquad\quad\displaystyle+s^\frac{1}{2}\int_0^Te^{-2s\alpha^*}(\xi^*)^\frac{1}{2}\|\pi(t,\cdot\,)\|^2_{H^\frac{1}{2}(\partial\Oo)}+\displaystyle s^2\lambda^2\iint_{(0,T)\times\omega^\prime}e^{-2s\alpha}\xi^2|\pi|^2\,\notag\\
		\noalign{\smallskip}\dis		
		&\qquad\qquad\quad +\dis\iint_{(0,T)\times\omega^\prime}e^{-2s\alpha}|\nabla\pi|^2\,\bigg)\label{I7}
\end{align}
for $\lambda\geq \widehat{\lambda}e^{\widehat{\lambda}T(\|A\|^2_{P}+\|B\|^2_{Q})}(1+\|a\|^2_\infty+\|b\|^2_\infty+\|c\|^2_\infty+\|B\|^5_Q+\|A\|^5_{P})$ and $s\geq \widehat{s}e^{4\lambda\|\eta^0\|_\infty}(T^4+T^8)$. Combining \eqref{I7} with \eqref{I5}, we can absorb the first and second terms in the right hand side of \eqref{I7} to get
\begin{align}
    		I(s,\lambda; \psi)+\dis I(s,\lambda; \varphi)&\leq C\bigg(s^3\lambda^4\displaystyle\iint_{(0,T)\times\omega_0} e^{-2s\alpha}\xi^3|\varphi|^2\,+s^3\lambda^4\displaystyle\iint_{(0,T)\times\omega_0} e^{-2s\alpha}\xi^3|\psi|^2\,\notag\\
    		&+\displaystyle s^\frac{1}{2}\lambda\int_0^Te^{-2s\alpha^*}(\xi^*)^\frac{1}{2}\|\pi(t,\cdot\,)\|^2_{H^\frac{1}{2}(\partial\Oo)}+\displaystyle s^2\lambda^3\iint_{(0,T)\times\omega^\prime}\!e^{-2s\alpha}\xi^2|\pi|^2\,\notag\\
		&+ \lambda\iint_{(0,T)\times\omega^\prime}e^{-2s\alpha}|\nabla\pi|^2\,\bigg)
		\label{I7bis}
\end{align}   	
for $\lambda\geq \widehat{\lambda}e^{\widehat{\lambda}T(\|A\|^2_{P}+\|B\|^2_{Q})}(1+\|a\|^2_\infty+\|b\|^2_\infty+\|c\|^2_\infty+\|B\|^5_Q+\|A\|^5_{P})$ and $s\geq \widehat{s}e^{4\lambda\|\eta^0\|_\infty}(T^4+T^8)$.
\\

\noindent \textbf{Step 3: Estimate of the trace of the pressure.}\\
\indent We introduce the followings functions:
\begin{equation*}
\beta(t)=s^\frac{1}{4}e^{-s\alpha^*}(\xi^*)^\frac{1}{4}, \quad \tilde{\varphi}=\beta \varphi \quad \hbox{and} \quad \tilde{\pi}=\beta\pi,
\end{equation*}
which satisfy
\begin{equation}\label{eq_local_tilde_varphi}
	\left\{
    \begin{array}{lll}
    		-\partial_t \tilde{\varphi}-\Delta \tilde{\varphi}-(a\cdot\nabla)\tilde{\varphi}-D\tilde{\varphi} b+\nabla \tilde{\pi} = \beta c\nabla \psi -\beta_t\varphi&\hbox{in} & \mathcal{O}_T,\\
   		\noalign{\smallskip}\dis
    	\nabla\cdot \tilde{\varphi}=0 &\hbox{in} &\mathcal{O}_T,\\
    	\noalign{\smallskip}\dis
    		\tilde{\varphi}\cdot \nu=0, \quad [D(\tilde{\varphi})\nu+A\tilde{\varphi}]_{tan}=0&\hbox{on} & \varLambda_T,\\
   		\noalign{\smallskip}\dis
    		\tilde{\varphi}(T,\cdot\,)=0 &\hbox{in} &\mathcal{O}.
    \end{array}
    \right.
\end{equation}
Let us regard $\tilde{\varphi}$ as a weak solution to \eqref{eq_local_tilde_varphi}. In particular,  $\tilde{\varphi}$ satisfies, by well-known energy estimates for the Stokes equation (see the beginning of the proof of \cite[Proposition $1.1$]{guerrero_CEL_NS}), the following:
\begin{equation*}
\|\tilde{\varphi}\|^2_{L^2(H^1)}\leq e^{CT(\|a\|^2_\infty+\|b\|^2_\infty+\|A\|^2_\infty)}\|\beta c\nabla \psi -\beta_t\varphi\|^2.
\end{equation*}
Again from energy estimates,  using the fact that $P \hookrightarrow L^\infty(\varLambda_T)^{n \times n}$, we have 
\begin{align*}
		\displaystyle\|\tilde{\pi}\|_{L^2(H^1)}^2&\leq \displaystyle Ce^{CT(\|a\|_\infty^2+\|b\|^2_\infty+\|A\|_P^2)}(1+\|A\|^4_{P})(1+\|a\|^2_\infty+\|b\|^2_\infty)\\
		&\quad\times \bigg(s^\frac{5}{2}e^{4\lambda\|\eta^0\|_\infty}T^2\displaystyle\iint_{\Oo_T}e^{-2s\alpha^*}(\xi^*)^3|\varphi|^2\,+\|c\|^2_\infty s^\frac{1}{2}\iint_{\Oo_T}e^{-2s\alpha^*}(\xi^*)^\frac{1}{2}|\nabla\psi|^2\,\bigg),
\end{align*}
where we have used that $\| \alpha^*_t\| +  \| \xi^*_t\| \leq CTe^{2\lambda\|\eta^0\|_\infty}(\xi^*)^{5/4}.$ 

Taking  $\lambda\geq \widehat{\lambda}e^{\widehat{\lambda}T(\|a\|^2_\infty+\|b\|^2_\infty+\|A\|_{P}^2+\|B\|^2_Q)}(1+\|a\|^2_\infty+\|b\|^2_\infty)(1+\|A\|^5_{P}+\|B\|^5_Q)(1+\|c\|^2_\infty)$ and $s\geq \widehat{s}e^{8\lambda\|\eta^0\|_\infty}(T^4 +T^8)$, from this last estimate and \eqref{I7}, we get:
\begin{equation*}
	\begin{array}{ll}
		\displaystyle\iint_{\Oo_T} e^{-2s\alpha}|\nabla\pi|^2\,+& \!\!\!\!\! \dis s^2\lambda^2\iint_{\Oo_T} e^{-2s\alpha}\xi^2|\pi|^2\, \\ 
	&\!\!\!\!\!\!\!\!	\leq  \displaystyle C\bigg(s(\|a\|^2_\infty+\|b\|^2_\infty)\iint_{\Oo_T} e^{-2s\alpha}\xi|\nabla\varphi|^2\,\\
		&\displaystyle+s\|c\|^2_\infty\iint_{\Oo_T} e^{-2s\alpha}\xi|\nabla\psi|^2\,+\displaystyle s^2\lambda^2\iint_{(0,T)\times\omega^\prime}e^{-2s\alpha}\xi^2|\pi|^2\,\\
		\noalign{\smallskip}\dis
		&+\dis\iint_{(0,T)\times\omega^\prime}e^{-2s\alpha}|\nabla\pi|^2\,+\displaystyle s^3\lambda\iint_{\Oo_T} e^{-2s\alpha}\xi^3|\varphi|^2\,.
	\end{array}
\end{equation*}
Combining this and \eqref{I7bis} and absorbing the lower order terms, we also get the estimates
\begin{equation}\label{I8}
\begin{array}{ll}
    		\dis I(s,\lambda; \psi)+\dis I(s,\lambda; \varphi)\!\!\!&\leq  C\bigg(s^3\lambda^4\displaystyle\iint_{(0,T)\times\omega_0} \!\!\!e^{-2s\alpha}\xi^3|\varphi|^2\,  \\
		\noalign{\smallskip}\dis
		&+s^3\lambda^4\displaystyle\iint_{(0,T)\times\omega_0}\!\!\! e^{-2s\alpha}\xi^3|\psi|^2\,  \\
		\noalign{\smallskip}\dis
    		&\dis+ s^2\lambda^3\iint_{\omega^\prime\times(0,T)} \!\!\!e^{-2s\alpha}\xi^2|\pi|^2\,  +\lambda \iint_{(0,T)\times\omega^\prime}\!\!\!e^{-2s\alpha}|\nabla\pi|^2\,  \bigg)
\end{array}
\end{equation}
for $\lambda\geq\widehat{\lambda} e^{\widehat{\lambda}T(\|a\|^2_\infty+\|b\|^2_\infty+\|A\|_{P}^2+\|B\|_{Q}^2)}(1+\|a\|^2_\infty+\|b\|^2_\infty)( 1+\|A\|^5_{P}+\|B\|^5_Q)(1+\|c\|^2_\infty)$ and $s\geq \widehat{s}e^{8\lambda\|\eta^0\|_\infty}(T^4+T^8)$.\\

\noindent \textbf{Step 4: Local estimates of the pressure.}\\
\indent We now follow the ideas of \cite{Cara_NS} to estimate the local terms on the pressure. Indeed, we assume that the pressure $\pi$ has mean-value zero in $\omega^\prime$:
\begin{equation*}
		\int_{\omega^\prime}\pi(t,\cdot\,)=0\quad \hbox{a.e.} \quad t\in (0,T).
\end{equation*}
Then, using that $e^{-2s\alpha} \xi^2 \leq e^{-2s\hat{\alpha}}\hat{\xi}^2$ and the Poincar\'e-Wirtinger's inequality, we have
\begin{equation*}
		s^2\lambda^3\iint_{(0,T)\times\omega^\prime} \!\!\!\!\!\!\!\!\!\!e^{-2s\alpha}\xi^2|\pi|^2\,  \leq Cs^2\lambda^3\iint_{(0,T)\times\omega^\prime} \!\!\!\!\!\!\!\!\!\!e^{-2s\hat{\alpha}}\hat{\xi}^2|\nabla\pi|^2\,  
\end{equation*}
and
\begin{equation*}
\lambda \iint_{(0,T)\times\omega^\prime}e^{-2s\alpha}|\nabla\pi|^2\,  \leq C s^2 \lambda^3 \iint_{(0,T)\times\omega^\prime}e^{-2s\hat{\alpha}}\hat{\xi}^2|\nabla\pi|^2\,  .
\end{equation*}
Now, using that
\begin{equation*}
		\nabla \pi = \partial_t \varphi+\Delta \varphi+(a\cdot\nabla)\varphi+D\varphi b+c\nabla \psi,
\end{equation*}
the estimate \eqref{I8} gives
 \begin{align}
   	 I(s,\lambda; \psi)+\dis I(s,\lambda; \varphi)&\leq C\bigg(s^3\lambda^4\displaystyle\iint_{(0,T)\times\omega_0} e^{-2s\alpha}\xi^3|\varphi|^2\,  +s^3\lambda^4\displaystyle\iint_{(0,T)\times\omega_0} e^{-2s\alpha}\xi^3|\psi|^2\,  \notag\\
    	&+ s^2\lambda^3\iint_{(0,T)\times\omega^\prime} e^{-2s\hat{\alpha}}\hat{\xi}^2|\varphi_t|^2\,  +s^2\lambda^3\iint_{(0,T)\times\omega^\prime} e^{-2s\hat{\alpha}}\hat{\xi}^2|\Delta\varphi|^2\,  \notag\\
   		&+s^2\lambda^3(\|a\|^2_\infty+\|b\|^2_\infty)\iint_{(0,T)\times\omega^\prime} e^{-2s\hat{\alpha}}\hat{\xi}^2|\nabla\varphi|^2\,  \notag\\
    		&
    		+s^2\lambda^3\|c\|^2_\infty\iint_{(0,T)\times\omega^\prime} e^{-2s\hat{\alpha}}\hat{\xi}^2|\nabla\psi|^2\,  \bigg)
\label{I9}
 \end{align}
for $\lambda\geq\widehat{\lambda} e^{\widehat{\lambda}T(\|a\|^2_\infty+\|b\|^2_\infty+\|A\|_{P}^2+\|B\|_{Q}^2)}(1+\|a\|^2_\infty+\|b\|^2_\infty)( 1+\|A\|^5_{P}+\|B\|^5_Q)(1+\|c\|^2_\infty)$ and $s\geq \widehat{s}e^{8\lambda \|\eta^0\|\infty}(T^4+T^8)$.\\



\noindent \textbf{Step 5: Local estimate of the term on  $\Delta\varphi$.}\\
\indent Now, we present a local estimate of the integral on $\Delta\varphi$ in the right-hand side of \eqref{I9}; this  follows the ideas included in \cite[Step 4 of the proof of Theorem 1]{Cara_NS}.
	
	Let us introduce an additional open set  $\omega_1$ such that $\omega^\prime\subset\subset\omega_1\subset\subset\omega_0\subset\subset\omega_c$,  $dist(\partial \omega^\prime,\partial \omega^1)\geq dist(\partial \omega^1, \partial \omega^0)$ and a positive function $\zeta\in \mathcal{D}(\omega_0)$ satisfying 
	$\zeta=1$ in $\omega_1$. Let $\widehat{\eta}(t):=s\lambda^{\frac{3}{2}}e^{-s\hat{\alpha}(t)}\hat{\xi}(t)$ and
\begin{equation*}
		\tilde{u}(t,x):=\hat{\eta}(t)\zeta(x)\Delta\varphi(T-t,x)\quad\hbox{in}\quad (0,T)\times\R^n,
\end{equation*}
where $\tilde{u}$ has been extended by zero outside $\omega_0$. 




Applying Laplace operator to  $\eqref{equa_local_adjunta}_1$,
we get
\begin{equation}\label{I10}
		(\Delta\varphi(T-t,\cdot\,))_t-\Delta(\Delta\varphi(T-t,\cdot\,))=\tilde{f}\quad\hbox{in} \quad Q,
\end{equation}
where
\begin{equation*}
   \begin{array}{lll}
   		\dis \tilde{f}&:=&\Delta((a\cdot\nabla)\varphi)(T-t,\cdot\,)+\Delta(D\varphi b)(T-t,\cdot\,)+\Delta(c\nabla\psi)(T-t,\cdot\,)\\
  		\noalign{\smallskip}\dis
  		&&-\nabla\left(\nabla\cdot((a\cdot\nabla)\varphi)(T-t,\cdot\,)\right)-\nabla(\nabla\cdot(D\varphi b))(T-t,\cdot\,)-\nabla(\nabla\cdot(c\nabla\psi))(T-t,\cdot\,).
   \end{array}
\end{equation*}
From \eqref{I10}, we deduce that $\tilde{u}$ solves
\begin{equation}\label{I11}
	\left\{
    \begin{array}{lll}
   \partial_t  \tilde{u}-\Delta \tilde{u} = \tilde{F}&\hbox{in}& (0,T)\times\R^n,\\
   \noalign{\smallskip}\dis
   \tilde{u}(0,\cdot\,)=0&\hbox{in}& \R^n,
    \end{array}
    \right.
\end{equation}
with
\begin{equation*}\label{I12.1}
		\tilde{F}=\hat{\eta}\zeta\tilde{f}+\hat{\eta}^\prime\zeta\Delta\varphi(T-t,\cdot\,)-2\hat{\eta}\nabla\zeta\cdot\nabla\Delta\varphi(T-t,\cdot\,)-\hat{\eta}\Delta\zeta\Delta\varphi(T-t,\cdot\,).
\end{equation*}
Notice that $\tilde{F}  \in L^2(0, T ; H^{-2}(\R^n )^n )$ and we a priori know that $\tilde{u}\in L^2((0, T )\times\R^n )^n$ (from its definition). From \eqref{I11},  we have that $\tilde{u} _t\in L^2(0,T;H^{-2}(\R^n)^n)$, so that $u(0,\cdot\,)$ makes sense. Now, we rewrite $\tilde{F}$  in a more appropriate way, so that it is given by the sum of two functions: in the first one, we include all the terms with derivatives of second order of $(a\cdot\nabla)\varphi$, $D\varphi b$, $c\nabla\psi$ and $\varphi$; in the second one, we consider all the other terms. Notice that this second function has a support contained in $\omega_0 \setminus\overline{ \omega}_1$ (because derivatives of $\zeta$ appear everywhere). More precisely, we set $\tilde{F}  = \tilde{F}_1  + \tilde{F}_2$, with
\begin{equation*}\label{I12}
   \begin{array}{ccc}
   		\dis \tilde{F}_1&= &\hat{\eta}\Delta\left(\zeta((a\cdot\nabla)\varphi)(T-t,\cdot\,)\right)+\hat{\eta}\Delta\left(\zeta(D\varphi b)(T-t,\cdot\,)\right)+\hat{\eta}\Delta\left(\zeta(c\nabla\psi)(T-t,\cdot\,)\right)\\
 		\noalign{\smallskip}\dis
		&&-\hat{\eta}\nabla\left( \nabla\cdot( \zeta        ((a\cdot\nabla)\varphi)  (T-t,\cdot\,))  \right)-\hat{\eta}\nabla\left( \nabla\cdot( \zeta        (D\varphi b)  (T-t,\cdot\,))  \right)\\
 		\noalign{\smallskip}\dis
  		&&-\hat{\eta}\nabla\left( \nabla\cdot( \zeta        (c\nabla\psi)  (T-t,\cdot\,))  \right)+\hat{\eta}^\prime\Delta(\zeta\varphi(T-t,\cdot\,)),
   \end{array}
   \end{equation*}
and
\begin{equation*}\label{I13}
    \begin{array}{ccc}
    		\dis \tilde{F}_2&= &-2\hat{\eta}\nabla\zeta\cdot\nabla(   (a\cdot\nabla)\varphi  )(T-t,\cdot\,)-\hat{\eta}\Delta\zeta    ((a\cdot\nabla)\varphi)    (T-t,\cdot\,)-2\hat{\eta}\nabla\zeta\cdot\nabla(   D\varphi b  )(T-t,\cdot\,)\\
   		\noalign{\smallskip}\dis
  		&&-\hat{\eta}\Delta\zeta    (D\varphi b)    (T-t,\cdot\,)-2\hat{\eta}\nabla\zeta\cdot\nabla(   c\nabla\psi   )(T-t,\cdot\,)-\hat{\eta}\Delta\zeta    (c\nabla\psi)    (T-t,\cdot\,)\\
   		\noalign{\smallskip}\dis
   		&&+\hat{\eta}\nabla\left( \nabla\zeta\cdot(     (a\cdot\nabla)\varphi      )  (T-t,\cdot\,))  \right)+\hat{\eta}\nabla\zeta(\nabla \cdot (     (a\cdot\nabla)\varphi    ) (T-t,\cdot\,))\\
  		\noalign{\smallskip}\dis
  		&&+\hat{\eta}\nabla\left( \nabla\zeta\cdot(     D\varphi b     )  (T-t,\cdot\,))  \right)+\hat{\eta}\nabla\zeta(\nabla \cdot (     D\varphi b    ) (T-t,\cdot\,))\\
   		\noalign{\smallskip}\dis
  		&&+\hat{\eta}\nabla\left( \nabla\zeta\cdot(     c\nabla\psi     )  (T-t,\cdot\,))  \right)+\hat{\eta}\nabla\zeta(\nabla \cdot (     c\nabla\psi    ) (T-t,\cdot\,))\\
   		\noalign{\smallskip}\dis
  		&&-2\hat{\eta}^\prime \nabla \zeta\cdot\nabla \varphi(T-t,\cdot\,)-\hat{\eta}^\prime \Delta \zeta \varphi(T-t,\cdot\,)-2\hat{\eta}\nabla\zeta\cdot\nabla\Delta\varphi(T-t,\cdot\,)-\hat{\eta}\Delta\zeta\Delta\varphi(T-t,\cdot\,).
    \end{array}
\end{equation*}
	Notice that $\tilde{F}$, $\tilde{F}_1 \in L^2(0,T; H^{-2}(\R^n)^n)$, while $\tilde{F}_2\in L^2(0,T;H^{-1}(\R^n)^n)$. 
	
	Next, we introduce two functions $\tilde{u}^1$ and $\tilde{u}^2$ in $L^2((0,T)\times\R^n)^n$ satisfying
\begin{equation}\label{I14}
	\left\{
    \begin{array}{lll}
   \partial_t  \tilde{u}^i-\Delta \tilde{u}^i = \tilde{F}_i&\hbox{in}& (0,T)\times\R^n,\\
    \noalign{\smallskip}\dis
    \tilde{u}^i(0,\cdot\,)=0&\hbox{in}& \R^n,
  \end{array}
      \right.
\end{equation}
for $i=1,2$. It is clear that $\tilde{u}=\tilde{u}^1+\tilde{u}^2$ then 
\begin{equation*}\label{I15}
\iint_{(0,T)\times\omega^\prime}|\tilde{u}|^2\,   \leq 2 \left( \iint_{(0,T)\times\omega^\prime}|\tilde{u}^1|^2\,   + \iint_{(0,T)\times\omega^\prime}|\tilde{u}^2|^2\,  \right).
\end{equation*}

\noindent \textbf{Step 5.a: Estimates of $\tilde{u}^1$.}\\
\indent We see $\tilde{u}^1$ as the transposition solution of the Cauchy problem for the heat equation \eqref{I14} for $i = 1$. This means that $\tilde{u}^1$ is the unique function in $L^2((0,T)\times\R^n)^n$ that, for each $h\in L^2((0,T)\times\R^n)^n$, one has
\begin{equation*}\label{I16}
    \begin{array}{rcl}
   		\dis \iint_{(0,T)\times \R^n}\tilde{u}^1\cdot h \ \,  &=&\dis \iint_{(0,T)\times \R^n}\left(\hat{\eta}\zeta (        (a\cdot\nabla)\varphi+D\varphi b+c\nabla \psi      )(T-t,\cdot\,)\right)\cdot \Delta z \,  \\
 		\noalign{\smallskip}\dis
 		&&\dis -\iint_{(0,T)\times \R^n}\hat{\eta}\zeta (        (a\cdot\nabla)\varphi+D\varphi b+c\nabla \psi      )(T-t,\cdot\,)\cdot \nabla(\nabla\cdot z) \,  \\
   		\noalign{\smallskip}\dis 
   		&&\dis +\iint_{(0,T)\times \R^n}\hat{\eta}^\prime\zeta \varphi(T-t,\cdot\,)\cdot \Delta z \,  ,
   \end{array}
\end{equation*}
where $z$ is the solution of 
\begin{equation}\label{I17}
	\left\{
    \begin{array}{lll}
    -\partial_t z-\Delta z = h&\hbox{in}& (0,T)\times\R^n,\\
    \noalign{\smallskip}\dis
   z(T,\cdot\,)=0&\hbox{in}& \R^n.
   \end{array}
      \right.
\end{equation}
Remark that, for every $h \in L^2((0, T )\times\R^n)^n$, equation \eqref{I17} possesses exactly one solution $z \in L^2(0, T ; H^2(\R^n )^n )$ that depends continuously on $h$. Therefore, $\tilde{u}^1$ is well defined and
\begin{equation}\label{I18}
\|\tilde{u}^1\|_{L^2((0, T )\times\R^n)^n}\leq C \|\tilde{F}_1\|_{L^2(0, T ; H^{-2}(\R^n )^n )}.
\end{equation}
Furthermore, it is not difficult to show that $\tilde{u}^1\in C^0([0,T];H^{-2}(\R^n)^n)$ and solves \eqref{I14} for $i = 1$ in the distributional sense. Moreover,  from \eqref{I18} it follows that
\begin{equation*}\label{I19}
   \begin{array}{ccc}
    		\dis \iint_{(0,T)\times \R^n}|\tilde{u}^1|^2\,  &\leq &\dis C\bigg( \iint_{(0,T)\times \R^n}|\hat{\eta}\zeta (a\cdot\nabla)\varphi|^2 \,  + \iint_{(0,T)\times \R^n}|\hat{\eta}\zeta D\varphi b|^2 \,  \\
  		\noalign{\smallskip}\dis
   		&&\dis + \iint_{(0,T)\times \R^n}|\hat{\eta}\zeta c\nabla \psi|^2 \,  + \iint_{(0,T)\times \R^n}|\hat{\eta}^\prime\zeta \varphi|^2 \,  \bigg).
    \end{array}
\end{equation*}
Here, we have used the fact that $\hat{\eta}(T-t,\cdot\,)=\hat{\eta}(t,\cdot\,) \quad \forall  t\in (0,T)$. Thanks to the properties of $\zeta$, we finally get
\begin{equation}\label{I20}
    \begin{array}{ccl}
    		\dis \iint_{(0,T)\times \omega^\prime}|\tilde{u}^1|^2\,  &\leq &\dis \iint_{(0,T)\times \R^n}|\tilde{u}^1|^2\,  \\
 		\noalign{\smallskip}\dis
    		&\leq&\dis C\bigg( \iint_{(0,T)\times \omega_0}|\hat{\eta} (a\cdot\nabla)\varphi|^2 \,  + \iint_{(0,T)\times \omega_0}|\hat{\eta} D\varphi b|^2 \,  \\
   		\noalign{\smallskip}\dis
    		&&\dis + \iint_{(0,T)\times \omega_0}|\hat{\eta} c\nabla \psi|^2 \,  + \iint_{(0,T)\times \omega_0}|\hat{\eta}^\prime \varphi|^2 \,  \bigg).
    \end{array}
\end{equation}

\noindent \textbf{Step 5.b: Estimates of $\tilde{u}^2$.}\\
\indent Now, we deal with the Cauchy problem \eqref{I14} for $i = 2$, where the right-hand side is in $L^2(0, T ; H^{ -1}(\R^n )^n )$. The existence and uniqueness of a solution $\tilde{u}^2 \in L^2(0, T ; H^1(\R^n )^n )$ is classical. Recall that $\tilde{F}_2(t,\cdot\,)$ has support in $\omega_0\setminus \overline{\omega}_1$ for almos every $t$, while we would like to estimate the $L^2$-norm of the solution in $\omega^\prime$ and $\omega^\prime$  is disjoint of  $\omega_0\setminus \overline{\omega}_1$. We will start by writing $\tilde{u}^2$ in terms of the fundamental solution $G = G( t,x)$ of the heat equation. To do this, we first notice that $\tilde{F}_2$  can be written in the form
\begin{equation*}
		\tilde{F}_2=\tilde{F}_{21}+\nabla\cdot \tilde{F}_{22},
\end{equation*}
	where $\tilde{F}_{21}$  and $\tilde{F}_{22}$  are $L^2$ functions supported in $[0, T ]\times(\omega_0\setminus \overline{\omega}_1)$ which can be written as sums of derivatives up to the second order of products $\hat{\eta}D^\beta\zeta  \varphi $, $\hat{\eta}D^\beta\zeta  (a\cdot\nabla)\varphi $, $\hat{\eta}D^\beta\zeta  D\varphi b $, $\hat{\eta}D^\beta\zeta c\nabla \psi  $ and $\hat{\eta}^\prime D^\beta\zeta \varphi  $ with $1\leq |\beta| \leq 4$. Thus, we have:
\begin{equation}\label{I21}
		\tilde{u}^2(t,x)=\int_0^t\!\!\!\int_{\omega_0\setminus \overline{\omega}_1}\!\!\!G(t-s,x-y)\tilde{F}_{21}(s,y) \ dy\,ds-\int_0^t\!\!\!\int_{\omega_0\setminus \overline{\omega}_1}\nabla_y G(t-s,x-y)\cdot\tilde{F}_{22}(s,y) \ dy\,ds,
\end{equation}
for all $(t,x)\in(0,T)\times\omega^\prime$, where $G$ is the fundamental solution for the heat operator given by
\begin{equation*}\label{I22}
		G(t,x)=\dfrac{e^{-|x|^2/2t}}{(4\pi t)^{n/2}} \quad\forall 	x\in \R^n, \quad \forall t>0.
\end{equation*}

	Notice that the above formula makes sense because the integration is over a region far from the singularity of $G$, i.e. for any $y \in \omega_0\setminus \overline{\omega}_1$ and any $x\in \omega^\prime$, one has $|x-y|\geq  dist(\partial\omega_1,\partial\omega_0)>0$. Integrating by parts with respect to $y$ in \eqref{I21} and passing all the derivatives from $\tilde{F}_{21}$ and $\tilde{F}_{22}$ to $G$ and $\nabla _y G$,  we obtain an expression for $\tilde{u}^2$ of the form
\begin{equation*}\label{I23}
		\tilde{u}^2(t,x)=\iint_{(0,t)\times(\omega_0\setminus \overline{\omega}_1)}\sum_{\alpha\in I, \beta\in J}D_y^\alpha G(t-s,x-y)D_y^\beta\zeta(y)  z_{\alpha,\beta}(s,y)dy\,ds,
\end{equation*}
where all $\alpha\in I$ satisfy $|\alpha|\leq 3$, all $\beta\in J $ satisfy $1\leq|\beta|\leq 4$ and
\begin{equation*}
		\begin{array}{ccc}
		z_{\alpha, \beta}(s,y)&=&\hat{\eta}(s)\left( C_{\alpha, \beta}\varphi(s,y)+D_{\alpha, \beta}((a\cdot\nabla)\varphi)(s,y)+E_{\alpha, \beta}(D\varphi b)(s,y)+F_{\alpha, \beta}(c\nabla \psi)(s,y)                 \right)\\
		&&+L_{\alpha, \beta}\hat{\eta}^\prime (s)\varphi(s,y),
		\end{array}
\end{equation*}
with $C_{\alpha, \beta}$, $D_{\alpha, \beta}$, $E_{\alpha, \beta}$, $F_{\alpha, \beta}$,  $L_{\alpha, \beta} \in \R$. The expression for 	$\tilde{u}^2$ yields 
\begin{equation*}\label{I24.1}
		|\tilde{u}^2(t,x)|\leq \iint_{(0,t)\times(\omega_0\setminus \overline{\omega}_1)}\sum_{\alpha\in I}| D_y^\alpha G(t-s,x-y)| |z(s,y)|dy\,ds
\end{equation*}
for all $(t,x)\in (0,T)\times \omega^\prime$, where
\begin{equation*}
		z(s,y)=\hat{\eta}(s)\left( C_1\varphi(s,y)+C_2((a\cdot\nabla)\varphi)(s,y)+C_3(D\varphi b)(s,y)+C_4(c\nabla \psi)(s,y)                 \right)+C_5\hat{\eta}^\prime (s)\varphi(s,y).
\end{equation*}
Now, for every $0 < \delta < dist(\partial\omega_1,\partial\omega_0)$ there exists a positive constant $C(\delta, \omega_c)$ such that
\begin{equation*}
|D^\alpha G(t-s,x-y)|\leq C \exp\left(\dfrac{-\delta^2}{2(t-s)}\right),\forall \alpha \in I,\ (t,x) \in  (0,T)\times \omega^\prime,\ \forall (s,y)\in (0,t)\times (\omega_0\setminus \overline{\omega}_1).
\end{equation*}
Thus, we have that
\begin{equation*}\label{I24}
		|\tilde{u}^2(t,x)|\leq C \iint_{(0,t)\times(\omega_0\setminus \overline{\omega}_1)}    \exp\left(\dfrac{-\delta^2}{2(t-s)}     \right)   |z(s,y)|\,dy\,ds.
\end{equation*}

Next, we integrate this last estimate  in $(0, T )\times \omega^\prime$ and use Cauchy-Scwharz inequality to obtain
\begin{equation*}
	\begin{array}{lll}
	\dis\iint_{(0, T )\times \omega^\prime}|\tilde{u}^2(t,x)|^2\,  &\leq&C\dis \int^T_0\left(    \int_0^t\int_{\omega_0\setminus \overline{\omega}_1}    \exp\left(\dfrac{-\delta^2}{2(t-s)}     \right)   |z(s,y)|dyds \right)^2dt\\
	\noalign{\smallskip}\dis 
	&\leq&C T\dis \int^T_0\left(    \int_0^t    \exp\left(\dfrac{-\delta^2}{2(t-s)}     \right)  \|z(s)\|^2_{L^2(\omega_0)}ds \right)dt.
	\end{array}
\end{equation*}

	Finally, observe that we can write the last term of the previous estimate as a convolution, i.e.
\begin{equation*}
	\int^T_0\left(    \int_0^t    \exp\left(\dfrac{-\delta^2}{2(t-s)}     \right)  \|z(s,\cdot\,)\|^2_{L^2(\omega_0)}ds \right)dt = \int_0^T (f_1*f_2)(t)dt,
\end{equation*}
where
\begin{equation*}
	f_1(t):=e^{-\delta^2/ 2 t}1_{[0,T]}(t) 
	\quad \text{and} \quad f_2(t):=\|z(t,\cdot\,)\|^2_{L^2(\omega_0)}1_{[0,T]}(t),
\end{equation*}
that is, $f_1$, $f_2\in L^1(\R) $. From Young's inequality, we obtain
\begin{equation*}
	\dis\iint_{(0, T )\times \omega^\prime}|\tilde{u}^2(t,x)|^2\,  \leq CT^2 \dis\iint_{(0, T )\times \omega_0}|z(t,x)|^2\,  
\end{equation*}
and the definition of $z$ gives
\begin{equation*}
	\dis\iint_{(0, T )\times \omega^\prime}\!\!\!\!\!\!\!\!\!\!\!\! |\tilde{u}^2(t,x)|^2\,  \leq CT^2 \left(\dis\iint_{(0, T )\times \omega_0}\!\!\!\!\!\! \!\!\!\!\!\! |\hat{\eta}^\prime \varphi|^2+|\hat{\eta}|^2\left(|\varphi|^2+|(a\cdot\nabla)\varphi)|^2+|D\varphi b|^2+|c\nabla \psi|^2   \right) \,  \right).
\end{equation*}
Hence, from \eqref{I20}, and the previous estimates of $\tilde{u}^1$ and $\tilde{u}^2$, we deduce the following
\begin{equation}\label{eq:local_laplacian}
\begin{array}{lll}
\noalign{\smallskip}\dis	\iint_{(0, T )\times \omega^\prime}\!\!\!\!\!\!\!\!\!\!\!\!  |\hat{\eta}|^2|\Delta \varphi|^2\,  
	&\leq& C(1+T^2) \left(\dis\iint_{(0, T )\times \omega_0}\!\!\!\!\!\! \!\!\!\!\!\! |\hat{\eta}^\prime \varphi|^2+|\hat{\eta}|^2\left(|\varphi|^2+|(a\cdot\nabla)\varphi|^2\right.\right.\\
	\noalign{\smallskip}\dis&&\left.+|D\varphi b|^2+|c\nabla \psi|^2   \right) \,  \Bigg)\\
	&\leq &\!C(1+T^2)\biggl(s^{9/2}\lambda^4\displaystyle\iint_{(0,T)\times\omega_0}\!\!\! \!\!\!\!\!\!\!\!e^{-2s\hat{\alpha}}\hat{\xi}^{9/2}|\varphi|^2\,  \\
  		\noalign{\smallskip}\dis
   		&&\displaystyle +s^2\lambda^3(\|a\|^2_\infty+\|b\|^2_\infty)\iint_{(0,T)\times\omega_0} \!\!\!\!\!\!\!\!\!\!\!e^{-2s\hat{\alpha}}\hat{\xi}^2|\nabla\varphi|^2\,  \\
		\noalign{\smallskip}\dis
    		&&\dis+s^2\lambda^3\|c\|^2_\infty\iint_{(0,T)\times\omega_0} \!\!\!\!\!\!\!\!e^{-2s\hat{\alpha}}\hat{\xi}^2|\nabla\psi|^2\,  \biggr)\\
		&\leq &\!C(1+T^2)\biggl(s^{9/2}\lambda^4\displaystyle\iint_{(0,T)\times\omega_0}\!\!\! \!\!\!\!\!\!\!\!e^{-2s\hat{\alpha}}\hat{\xi}^{9/2}|\varphi|^2\,  \\
  		\noalign{\smallskip}\dis
   		&&\displaystyle +s^2\lambda^4\iint_{(0,T)\times\omega_0} \!\!\!\!\!\!\!\!\!\!\!e^{-2s\hat{\alpha}}\hat{\xi}^2(|\nabla\varphi|^2+|\nabla\psi|^2)\,  \biggr),
\end{array}
\end{equation}
for $\lambda\geq \widehat{\lambda}e^{\widehat{\lambda}T(\|a\|^2_\infty+\|b\|^2_\infty+\|A\|_{P}^2+\|B\|_{Q}^2)}(1+\|a\|^2_\infty+\|b\|^2_\infty)( 1+\|A\|^5_{P}+\|B\|^5_Q)(1+\|c\|^2_\infty)$ and $s\geq \widehat{s}e^{8\lambda \|\eta^0\|\infty}(T^4+T^8)$.\\


\noindent \textbf{Step 6: Local estimate of  $\varphi_t$.} \\
\indent In this step, we estimate the local term on $\varphi_t$ in \eqref{I9}. First, integration by parts gives
\begin{equation*}
    \begin{array}{ccc}
   		\dis s^2\lambda^3\iint_{(0,T)\times\omega^\prime} e^{-2s\hat{\alpha}}\hat{\xi}^2|\varphi_t|^2\,  &=&\dfrac{1}{2}s^2\lambda^3\displaystyle\iint_{(0,T)\times\omega^\prime} \left( e^{-2s\hat{\alpha}}\hat{\xi}^2\right)_{tt}|\varphi|^2\,   \\
    		\noalign{\smallskip}\dis
    		&&\displaystyle -s^2\lambda^3\iint_{(0,T)\times\omega^\prime}e^{-2s\hat{\alpha}}\hat{\xi}^2\varphi\cdot\varphi_{tt}\,  .
    \end{array}
\end{equation*}
Now, since there exists $C>0$ such that 
\begin{equation*}\label{I27}
\left|\left( e^{-2s\hat{\alpha}}\hat{\xi}^2\right)_{tt}\right| \leq Cs^2T^2e^{-2s\hat{\alpha}}\hat{\xi}^{9/2} \ \quad \text{and} \ \quad e^{-2s\hat{\alpha}} \leq Ce^{-4s\hat{\alpha}+2s\alpha^*}
 \end{equation*}
 we have that 
\begin{align}
     \displaystyle s^2\lambda^3\iint_{(0,T)\times\omega^\prime} e^{-2s\hat{\alpha}}\hat{\xi}^2|\varphi_t|^2\,  &\leq 
    		Cs^{15/2}\lambda^8\displaystyle\iint_{(0,T)\times\omega^\prime} e^{-4s\hat{\alpha}+2s\alpha^*}\hat{\xi}^{15/2}|\varphi|^2\,  \notag \\
    		&\displaystyle +\iint_{(0,T)\times\omega^\prime} |\eta^*|^2|\varphi_{tt}|^2\,  ,
    		\label{28}
\end{align}
with
 \begin{equation*}\label{I29}
\eta^*:=s^{-7/4}\lambda^{-1} e^{-s\alpha^*}\hat{\xi}^{-7/4}.
\end{equation*}

In what follows,  we estimate the second term in the right-hand side of \eqref{28}. To do this, we set $(y,q,\phi):=(\eta^*\varphi_t,\eta^*\pi_t,\eta^*\psi_t)$, and note that $(y,q,\phi)$ solves 
 \begin{equation}\label{I30}
	\left\{
    \begin{array}{lll}
    		-\partial_t y-\Delta y-(a\cdot\nabla)y-Dy b+\nabla q = c\nabla \phi + G_1 &\hbox{in} & \mathcal{O}_T,\\
   		\noalign{\smallskip}\dis
    		-\partial_t \phi-\D \phi-(a+b)\cdot \nabla \phi=y\cdot e_n +G_2&\hbox{in} & \mathcal{O}_T,\\
    		\noalign{\smallskip}\dis
    		\nabla\cdot y=0 &\hbox{in} &\mathcal{O}_T,\\
    		\noalign{\smallskip}\dis
    		y\cdot \nu=0, \quad [D(y)\nu+Ay]_{tan}=-\eta^*A_t\varphi&\hbox{on} & \varLambda_T,\\
    		\noalign{\smallskip}\dis
    		 \dfrac{\partial\phi}{\partial\nu}+B\phi=-\eta^* B_t\psi&\hbox{on} & \varLambda_T,\\
    		\noalign{\smallskip}\dis
    		y(T,\cdot\,)=0,\quad \phi(T,\cdot\,)=0 &\hbox{in} & \mathcal{O},
    \end{array}
    \right.
 \end{equation}
where
\begin{equation*}\label{I31}
		G_1=-\eta^*_t\varphi_t+\eta^*(a_t\cdot\nabla)\varphi+\eta^*D\varphi b_t+\eta^*c_t\nabla\psi
\end{equation*}
and
\begin{equation*}\label{I32}
		G_2=-\eta^*_t\psi_t+\eta^*(a+b)_t\cdot\nabla\psi.
\end{equation*}

To see that  $(y,q,\phi)$ solves \eqref{I30}, one can take a sequence   of regular  functions $(a^k,b^k,c^k)$ such that
\begin{equation*}\label{I33}
		(a^k,b^k,c^k)\longrightarrow (a,b,c) \quad \hbox{weakly star in} \quad L^\infty (0,T;L^\infty(\Oo)^n)
\end{equation*}
and 
\begin{equation*}\label{I34}
		(a^k_t,b^k_t,c^k_t)\longrightarrow (a_t,b_t,c_t) \quad \hbox{weakly in} \quad L^2(0,T;L^2(\mathcal{O})^n).
\end{equation*}
Since  there exists a unique solution $(y^k,q^k,\phi^k)$ to \eqref{I30} with $(a,b,c)$ replaced by $(a^k,b^k,c^k)$, one can take  limits and conclude that $(y,q,\phi)$ solves \eqref{I30}.

Next, using the fact that  $\varphi \in L^2(0,T;H^1(\Oo)^n) \cap H^1(0,T;H^{-1}(\Oo)^n)$, $\psi \in L^2(0,T;H^1(\Oo)) \cap H^1(0,T;H^{-1}(\Oo) )$ and the hypothesis on $a$, $b$, $c$, $A$ and $B$,  we see that $G_1\in L^2(0,T;H^{-1}(\mathcal{O})^n)$, $G_2\in L^2(0,T;H^{-1}(\mathcal{O}))$, $\eta^*  A_t\varphi \in L^2(0,T;H^{-1/2}(\mathcal{\partial O})^n)$ and 
$\eta^*  B_t\psi \in L^2(0,T;H^{-1/2}(\mathcal{\partial O}))$. Moreover,  the following estimate holds
\begin{align}
		&\|y\|^2_{L^2(H^1)}+\|\phi\|^2_{L^2(H^1)}\notag\\
		&\qquad\leq Ce^{CT(\|a\|^2_\infty+\|b\|^2_\infty+\|c\|^2_\infty+\|A\|^2_\infty)}
		 \left( \|(G_1,G_2)\|^2_{L^2(H^{-1})} +\|\eta^*(A_t\varphi, B_t\psi)\|^2_{L^2(H^{-1/2})}\right).\label{I35}
\end{align}

Notice that this is still not enough to absorb the local term on $\varphi_{tt}$ in \eqref{28}. Thus, we must show that $y$ is actually  a strong solution of 
	$\eqref{I30}_{1,3,4,6}$,  which will be true if we prove that  $G_1 \in L^2(\Oo_T)^n$ and $\eta^* A_t \varphi \in  H^{(1-l)/2}(0,T;H^{(l-1/2)}(\partial \mathcal{O})^n)\cap L^2(0,T;H^{1/2}(\partial\mathcal{O})^{N})$.

To see that $G_1 \in L^2(\Oo_T)^n$,  we must verify that $\eta^*(a_t\cdot\nabla)\varphi$, $\eta^*D\varphi b_t$ and $\eta^*c_t\nabla\psi$ belong to $L^2(\Oo_T)^n$. In fact, since $y \in L^2(0,T;H^1(\mathcal{O})^n),$ we have that $\eta^*\nabla\varphi \in H^1(0,T;L^2(\mathcal{O})^{n\times n})$ and, using that $\eta^*\nabla \varphi\in L^2(0,T;H^1(\mathcal{O})^{n\times n})$ and \cite[Theorem II.5.14]{boyer}, we conclude that 
 \begin{equation*}\label{fort1}
\eta^*\nabla\varphi \in C^0([0,T]; H^{1/2}(\mathcal{O})^{n\times n}).
\end{equation*}
Analogously, we have that
\begin{equation*}\label{fort2}
\eta^*\nabla\psi\in C^0([0,T]; H^{1/2}(\mathcal{O})^{n}).
\end{equation*}
Hence, from the assumptions on  $a_t$, $b_t$ and $c_t$, we readily see that $\eta^*(a_t\cdot\nabla)\varphi\in L^2(\Oo_T)^n$, $\eta^*D\varphi b_t\in L^2(\Oo_T)^n$ and  $\eta^*c_t\nabla\psi\in L^2(\Oo_T)^n$.  Moreover, the following estimate holds 
\begin{equation*}\label{I36}
	\begin{array}{ccc}
		&&\|\eta^*(a_t\cdot\nabla)\varphi\|^2_{L^2(\Oo_T)^n}+\|\eta^*D\varphi b_t\|^2_{L^2(\Oo_T)^n}   + \|\eta^*c_t\nabla\psi\|^2_{L^2(\Oo_T)^n}\\
			&&\leq C\left(\|a_t\|^2_{L^2(L^r)}+\|b_t\|^2_{L^2(L^r)}+\|c_t\|^2_{L^2(L^r)}\right)
		\\
		\noalign{\smallskip}\dis
		&&
		\times \bigg(\|\eta^*\varphi\|^2_{L^2(H^2)}+\|\eta^*_t\varphi\|^2_{L^2(H^1)}+\|y\|^2_{L^2(H^1)}
		\\
		\noalign{\smallskip}\dis		
		&&
		+\|\eta^*\psi\|^2_{L^2(H^2)}+\|\eta^*_t\psi\|^2_{L^2(H^1)}+\|\phi\|^2_{L^2(H^1)}\bigg).
	\end{array}
\end{equation*}


Let us now prove that $\eta^* A_t \varphi \in H^{(1-l)/2}(0,T;H^{(l-1/2)}(\partial \mathcal{O})^n)\cap L^2(0,T;H^{1/2}(\partial\mathcal{O})^n)$.  Indeed,  from estimate \eqref{I35} we see that  $\eta^*\varphi \in H^1(0, T ; H^{1/2}(\partial\mathcal{O})^n)$ and, together with assumption \eqref{hipo_A_3} on $A$, we obtain
\begin{equation*}
\eta^*A_t\varphi\in 
H^{(1-l)/2}(0,T;H^{\vartheta_2}(\partial\mathcal{O})^n)\subset 
H^{(1-l)/2}(0,T;H^{l-1/2}(\partial\mathcal{O})^n),
\end{equation*}
with the following estimate
\begin{equation*}
	\|\eta^*(\partial_t A)\varphi\|^2_{H^{(1-l)/2}(H^{l-1/2})}\leq C\|A\|^2_{H^{(3-l)/2}(H^{\vartheta_2})}\left(       \|\eta^*_t\varphi\|^2_{L^2(H^1)}+\|y\|^2_{L^2(H^1)}+\|\eta^*\varphi\|^2_{L^2(H^1)}\right).
\end{equation*}
	Also, since $\eta^*\varphi\in  L^2(0, T ; H^2(\mathcal{O})^n )\cap H^1(0, T ; H^1(\mathcal{O})^n )$, we have that $\eta^*\varphi \in H^{1/4}(0, T ; H^{5/4}(\partial \mathcal{O})^n )$, which gives $\eta^* A_t\varphi\in L^2(0,T;H^{1/2}(\partial \mathcal{O})^n)$, because $A_t \in H^{(1-l)/2}(0,T;H^{\vartheta_2}(\partial \mathcal{O}^{n\times n}))$. Moreover, 
 \begin{equation*}
\|\eta^* A_t\varphi\|_{L^2(H^{1/2})}^2\leq C\|A\|^2_{H^{(3-l)/2}(H^{\vartheta_2})}\left(\|\eta^*\varphi\|^2_{L^2(H^2)}+\|\eta^*_t\varphi\|^2_{L^2(H^1)}+\|y\|^2_{L^2(H^1)}\right).
 \end{equation*}	

Thus, we have proved that $y$ is   a strong solution of 
	$\eqref{I30}_{1,3,4,6}$. Recalling  \cite[Proposition $1.1$]{guerrero_CEL_NS}, we deduce  in particular that   $y_t\in L^2(\Oo_T)$ and 
\begin{align*}
\|y_t\|^2_{L^2(\Oo_T)}&\leq C  e^{CT\|A\|^2_P}(1+\|A\|^4_P)\Big(\|G_1\|^2_{L^2(\Oo_T)^n}+\|(a\cdot\nabla)y\|^2_{L^2(\Oo_T)^n}\\
	& +\|Dyb\|^2_{L^2(\Oo_T)^n}+\|c\nabla \phi\|^2_{L^2(\Oo_T)^n}+\|\eta^*A_t\varphi\|^2_{L^2(H^{1/2})}+\|\eta^*A_t\varphi\|^2_{H^{(1-l)/2}(H^{l-1/2})}\Big)\\
&\leq C e^{CT\|A\|^2_P}(1+\|A\|^4_P)(1+\|A\|^2_{P})\bigg[\left(1+\|a_t\|^2_{L^2(L^r)}+\|b_t\|^2_{L^2(L^r)}+\|c_t\|^2_{L^2(L^r)}\right)\\
		&\times\left(\|\eta^*_t\varphi_t\|^2_{L^2(\Oo_T)^n}+\|\eta^*\varphi\|^2_{L^2(H^2)}+\|\eta^*_t\varphi\|^2_{L^2(H^1)}+\|y\|^2_{L^2(H^1)}+\|\eta^*\psi\|^2_{L^2(H^2)}\right.\\
		&\left.+\|\eta^*_t\psi\|^2_{L^2(H^1)}+\|\phi\|^2_{L^2(H^1)}\right)+\left(\|a\|^2_\infty+\|b\|^2_\infty+\|c\|^2_\infty\right)\left(\|y\|^2_{L^2(H^1)}+\|\phi\|^2_{L^2(H^1)}\right)\\
		&+\left(\|\eta^*\varphi\|^2_{L^2(H^2)}+\|\eta^*_t\varphi\|^2_{L^2(H^1)}+\|y\|^2_{L^2(H^1)}\right)    \bigg].	
\label{I37}
 \end{align*}

	Taking now $\lambda \geq \widehat{\lambda}e^{\widehat{\lambda}T(\|a\|^2_\infty+\|b\|^2_\infty+\|A\|^2_P+\|B\|^2_Q)}(1+\|A\|^{2}_{P})(1+\|A\|^5_P+\|B\|^5_Q)(1+\|a\|^{2}_\infty+\|b\|^{2}_\infty+  \|c\|^{2}_\infty+    \|a_t\|^{2}_{L^2(L^r)}+\|b_t\|^{2}_{L^2(L^r)}+\|c_t\|^{2}_{L^2(L^r)})(1+\|c\|^2_\infty)$ and $s\geq \widehat{s}e^{8\lambda \|\eta^0\|\infty}(T^4+T^8)$, from \eqref{I35} we obtain 
\begin{equation*}
		\begin{array}{ccc}
		\|\eta^*\varphi_{tt}\|^2_{L^2(\Oo_T)}
		&\leq&C \lambda^2 \left(\|\eta^*_t\varphi_t\|^2_{L^2(\Oo_T)^n}+\|\eta^*_t\psi_t\|^2_{L^2(\Oo_T)}+\|\eta^*\varphi\|^2_{L^2(H^2)}+\|\eta^*_t\varphi\|^2_{L^2(H^1)}\right.\\
		\noalign{\smallskip}\dis
		&&\left.+\|\eta^*\psi\|^2_{L^2(H^2)} +\|\eta^*_t\psi\|^2_{L^2(H^1)}+\|y\|^2_{L^2(\Oo_T)}+\|\phi\|^2_{L^2(\Oo_T)}\right).
	\end{array}
 \end{equation*}
Since  $|\eta^*_t|\leq \varepsilon\lambda^{-1}s^{-1/2}\hat{\xi}^{-1/2}e^{-s\alpha^*}$,  for $\varepsilon$ sufficiently small, the following is found:
 \begin{equation}\label{I39}
		\begin{array}{lll}
		\|\eta^*\varphi_{tt}\|^2_{L^2(\Oo_T)}
		&\leq&\dis C\varepsilon\biggl(s^{-1}\iint_{\Oo_T} e^{-2s\alpha^*}\hat{\xi}^{-1}(|\varphi_t|^2+|\psi_t|^2) \,  \\
			\noalign{\smallskip}\dis
		&&\dis+s^{-1}\iint_{\Oo_T} e^{-2s\alpha^*}\hat{\xi}^{-1}(|\nabla \varphi|^2+|\nabla\psi|^2) \,  \biggr)\\
		\noalign{\smallskip}\dis
		&&+C\lambda^2\left(\|\eta^*\varphi\|^2_{L^2(H^2)}+\|\eta^*\psi\|^2_{L^2(H^2)}\right).
	\end{array}
 \end{equation}
 
 We need to estimate the terms $\|\eta^*\varphi\|^2_{L^2(H^2)}$ and $\|\eta^*\psi\|^2_{L^2(H^2)}$. Thus, let us set $(\hat{\varphi},\hat{\pi},\hat{\psi}):=\eta^*(\varphi,\pi,\psi)$. One has:
\begin{equation*}
	\left\{
   \begin{array}{lll}
    		-\partial_t \hat{\varphi}-\Delta \hat{\varphi}-(a\cdot\nabla)\hat{\varphi}-D\hat{\varphi} b+\nabla \hat{\pi} = \eta^* c\nabla \psi -\eta^*_t\varphi&\hbox{in} & \mathcal{O}_T,\\
		\noalign{\smallskip}\dis
		- \partial_t \hat{\psi}-\Delta \hat{\psi}-(a+b)\cdot\nabla\hat{\psi} = \eta^*\varphi\cdot e_n-\eta^*_t\psi&\hbox{in} & \mathcal{O}_T,\\
    		\noalign{\smallskip}\dis
    		\nabla\cdot \hat{\varphi}=0 &\hbox{in} &\mathcal{O}_T,\\
    		\noalign{\smallskip}\dis
    		\hat{\varphi}\cdot \nu=0, \quad  [D(\hat{\varphi})\nu+A\hat{\varphi}]_{tan}=0&\hbox{on} & \varLambda_T,\\
		\noalign{\smallskip}\dis
    		\dfrac{\partial\hat{\psi}}{\partial \nu}+B\hat{\psi}=0&\hbox{on} & \varLambda_T,\\
    		\noalign{\smallskip}\dis
    		\hat{\varphi}(T,\cdot\,)=0, \quad \hat{\psi}(T,\cdot\,)=0&\hbox{in} &\mathcal{O}.
    \end{array}
    \right.
 \end{equation*}
 Again, from energy estimates  \cite[Proposition $1.1$]{guerrero_CEL_NS}, we find that
 \begin{equation*}
		\displaystyle\|\hat{\varphi}\|_{L^2(H^2)}^2\leq  \displaystyle Ce^{CT(\|a\|^2_\infty+\|b\|^2_\infty+\|A\|_{P}^2)}(1+\|A\|^4_P)(1+\|a\|^2_\infty+\|b\|^2_\infty) \left(\|\eta^*_t \varphi \|^2+\|c\|^2_\infty \|\eta^*\nabla\psi\|^2 \right),
 \end{equation*}
 and, from maximal $L^2$-regularity estimates for the heat equation with homogeneous Robin boundary conditions (similar arguments as in \cite[Proposition $2$]{C-G-B-P-2}), we deduce that
\begin{equation*}
	{\|\hat{\psi}\|^2_{L^2(H^2)} \leq C e^{CT(\|a\|^2_\infty+\|b\|^2_\infty+\|B\|_Q^2)}(1+\|B\|^4_Q)(1+\|a\|^2_\infty+\|b\|^2_\infty)\left(\|\eta^*\varphi\|^2+\|\eta^*_t\psi\|^2\right)}.
 \end{equation*}
 
 Adding the last two inequalities, we have:
\begin{equation*}\label{eq:h2}
	\begin{array}{lll}
		\displaystyle\|\hat{\varphi}\|_{L^2(H^2)}^2+\|\hat{\psi}\|^2_{L^2(H^2)}&\leq& \lambda\left(\|\eta^*_t \varphi \|^2+\|\eta^*_t\psi\|^2+\|\eta^*\nabla\psi\|^2 +\|\eta^*\varphi\|^2\right)\\
		\noalign{\smallskip}\dis
		&\leq & \displaystyle \lambda  \biggl(\varepsilon s^{-1}\lambda^{-2}\iint_{\Oo_T} e^{-2s\alpha^*}\hat{\xi}^{-1}(|\varphi|^2+|\psi|^2) \,  \\
		\noalign{\smallskip}\dis
		&&\dis+ s^{-7/2}\lambda^{-2}\iint_{\Oo_T}e^{-2s\alpha^*}\hat{\xi}^{-7/2}|\nabla\psi|^2\,   \biggr).
	\end{array}
 \end{equation*}

From this last estimate, \eqref{I39} and \eqref{28}, we see that

\begin{equation}\label{eq:tempo}
\begin{array}{lll}
    		\dis s^2\lambda^3\iint_{(0,T)\times\omega^\prime} e^{-2s\hat{\alpha}}\hat{\xi}^2|\varphi_t|^2\,  &\leq&
    		C\displaystyle s^{15/2}\lambda^8\displaystyle\iint_{(0,T)\times\omega^\prime} e^{-4s\hat{\alpha}+2s\alpha^*}\hat{\xi}^{15/2}|\varphi|^2\,  \\
    		\noalign{\smallskip}\dis
		&&\dis +\varepsilon I(s,\lambda;\varphi)+\varepsilon I(s,\lambda;\psi),
 \end{array}
 \end{equation}
 for  $\lambda \geq \widehat{\lambda}e^{\widehat{\lambda}T(\|a\|^2_\infty+\|b\|^2_\infty+\|A\|^2_P+\|B\|^2_Q)}(1+\|A\|^{2}_{P})(1+\|A\|^5_P)(1+\|B\|^5_Q)(1+\|a\|^{2}_\infty+\|b\|^{2}_\infty+  \|c\|^{2}_\infty+    \|a_t\|^{2}_{L^2(L^r)}+\|b_t\|^{2}_{L^2(L^r)}+\|c_t\|^{2}_{L^2(L^r)}+\|B\|^2_\infty)(1+\|c\|^2_\infty)$ and $s\geq \widehat{s}e^{8\lambda \|\eta^0\|\infty}(T^4+T^8)$.
%
%
%
\\

\noindent \textbf{Step 7: Arrangements.}\\
Combining \eqref{I9}, \eqref{eq:local_laplacian} and \eqref{eq:tempo}, it follows that 
\begin{equation}\label{I40}
    \begin{array}{lll}
    		\dis I(s,\lambda; \psi)+\dis I(s,\lambda; \varphi)&\leq &C(1+T^2)\biggl(s^{15/2}\lambda^8\displaystyle\iint_{(0,T)\times\omega_0} \!\!\!\!\! \!\!\!\!\! e^{-4s\hat{\alpha}+2s\alpha^*}\hat{\xi}^{15/2}|\varphi|^2\,  \\
    		\noalign{\smallskip}\dis
    		&&\displaystyle +s^3\lambda^4\displaystyle\iint_{(0,T)\times\omega_0} \!\!\!\!\!\! \!\!\!\!\! e^{-2s\alpha}\xi^3|\psi|^2\,  \\
		\noalign{\smallskip}\dis
		&&\dis+s^2\lambda^4\iint_{(0,T)\times\omega_0} \!\!\!\!\!\!\!\!\!\!\!\!\! e^{-2s\hat{\alpha}}\hat{\xi}^2(|\nabla\varphi|^2+|\nabla\psi|^2)\,  \biggr),
    \end{array}
 \end{equation}
for $\lambda \geq \widehat{\lambda}e^{\widehat{\lambda}T(\|a\|^2_\infty+\|b\|^2_\infty+\|A\|^2_P+\|B\|^2_Q)}(1+\|A\|^{2}_{P})(1+\|A\|^5_P)(1+\|B\|^5_Q)(1+\|a\|^{2}_\infty+\|b\|^{2}_\infty+  \|c\|^{2}_\infty+    \|a_t\|^{2}_{L^2(L^r)}+\|b_t\|^{2}_{L^2(L^r)}+\|c_t\|^{2}_{L^2(L^r)}+\|B\|^2_\infty)(1+\|c\|^2_\infty)$ and $s\geq \widehat{s}e^{8\lambda \|\eta^0\|\infty}(T^4+T^8)$.\\

\noindent \textbf{Step 8: Estimates of the local gradient terms.}\\
Let us consider a cut-off function $\rho \in C^1(\overline{\omega}_c)$ with $\rho = 1$ in $ \omega_0$, $\hbox{supp} \ \rho\subset\subset \omega_c$. Then,
\begin{equation*}
\begin{array}{lll}
    		\dis s^2\lambda^4\iint_{(0,T)\times\omega_0} \!\!\!\!\!\!\!\!\!\!\!\!\! e^{-2s\hat{\alpha}}\hat{\xi}^2|\nabla\varphi|^2\,   
		&\leq &\dis s^2\lambda^4\iint_{\Oo_T} \!\! e^{-2s\hat{\alpha}}\hat{\xi}^2\rho|\nabla\varphi|^2\,  .
 \end{array}
 \end{equation*}
 	After  integration by parts, thanks to H\"older and Young inequalities, we deduce that
\begin{equation*}
\begin{alignedat}{2}
    		\dis s^2\lambda^4\iint_{(0,T)\times\omega_0}  e^{-2s\hat{\alpha}}\hat{\xi}^2|\nabla\varphi|^2\,   
		 \leq&~		\varepsilon s^{-1}\iint_{\Oo_T}e^{-2s\alpha^*}\hat{\xi}^{-1}(|\Delta \varphi|^2   +|\nabla \varphi|^2)\,  \\
    		&+Cs^5\lambda^8\iint_{(0,T)\times\omega_c} e^{-4s\hat{\alpha}+2s\alpha^*}\hat{\xi}^{5}|\varphi|^2\,  ,
 \end{alignedat}
 \end{equation*}
 	where $\varepsilon$ is a small enough constant. 
 
 	Similar computations yield
 \begin{equation*}
\begin{alignedat}{2}
    		\dis s^2\lambda^4\iint_{(0,T)\times\omega_0}  e^{-2s\hat{\alpha}}\hat{\xi}^2|\nabla\psi|^2\,   
		 \leq&~		\varepsilon s^{-1}\iint_{\Oo_T}e^{-2s\alpha^*}\hat{\xi}^{-1}(|\Delta \psi|^2   +|\nabla \psi|^2)\,  \\
    		&+Cs^5\lambda^8\iint_{(0,T)\times\omega_c} e^{-4s\hat{\alpha}+2s\alpha^*}\hat{\xi}^{5}|\psi|^2\,  .
	 \end{alignedat}
 \end{equation*}
 
	Then, using \eqref{I40}, and these inequalities, we obtain
 \begin{equation*}
 \begin{alignedat}{2}
    		\dis I(s,\lambda; \psi)+\dis I(s,\lambda; \varphi)\leq&~ C(1+T^2)\biggl(s^{15/2}\lambda^8\displaystyle\iint_{(0,T)\times\omega_c} e^{-4s\hat{\alpha}+2s\alpha^*}\hat{\xi}^{15/2}|\varphi|^2\,  \\
    		&\displaystyle+s^5\lambda^8\iint_{(0,T)\times\omega_c}e^{-4s\hat{\alpha}+2s\alpha^*}\hat{\xi}^{5}|\varphi|^2\,  \\
		\noalign{\smallskip}\dis
    		&\dis +s^3\lambda^4\displaystyle\iint_{(0,T)\times\omega_c} e^{-2s\alpha}\xi^3|\psi|^2\,  \\
    		\noalign{\smallskip}\dis
    		&\dis + s^5\lambda^8\iint_{(0,T)\times\omega_c}e^{-4s\hat{\alpha}+2s\alpha^*}\hat{\xi}^{5}|\psi|^2\,  \biggr)\\
    		\noalign{\smallskip}\dis
    		&+\varepsilon I(s,\lambda;\varphi)+\varepsilon I(s,\lambda;\psi),
 \end{alignedat}
 \end{equation*}
 	for $\lambda \geq \widehat{\lambda}e^{\widehat{\lambda}T(\|a\|^2_\infty+\|b\|^2_\infty+\|A\|^2_P+\|B\|^2_Q)}(1+\|A\|^{2}_{P})(1+\|A\|^5_P)(1+\|B\|^5_Q)(1+\|a\|^{2}_\infty+\|b\|^{2}_\infty+  \|c\|^{2}_\infty+    \|a_t\|^{2}_{L^2(L^r)}+\|b_t\|^{2}_{L^2(L^r)}+\|c_t\|^{2}_{L^2(L^r)}+\|B\|^2_\infty)(1+\|c\|^2_\infty)$ and $s\geq \widehat{s}e^{8\lambda \|\eta^0\|\infty}(T^4+T^8)$.
	Finally, we easily obtain the desired Carleman estimate \eqref{carleman_local} by taking $\varepsilon$ sufficiently small. 
	
	This concludes the proof.

\end{appendices}

\end{document}